\documentclass[11pt]{amsart}

\usepackage{epigamath}

\usepackage[english]{babel}

\numberwithin{equation}{subsection}

\usepackage[all]{xy}
\usepackage{amscd}

\newtheorem{lemma}{Lemma}[subsection]
\newtheorem{thm}[lemma]{Theorem}
\newtheorem{theorem}{Theorem}
\newtheorem{prop}[lemma]{Proposition}
\newtheorem{proposition}[lemma]{Proposition}
\newtheorem{cor}[lemma]{Corollary}

\theoremstyle{definition}
\newtheorem{Df}{Definition}
\newtheorem{defn}[lemma]{Definition}

\newtheorem{definition}[lemma]{Definition}
\newtheorem{nota}[lemma]{Notation}

\theoremstyle{remark}
\newtheorem{Qn}{Question}
\newtheorem{qn}[lemma]{Question}
\newtheorem{Rk}{Remark}
\newtheorem{rk}[lemma]{Remark}
\newtheorem{remark}[lemma]{Remark}
\newtheorem{remarks}[lemma]{Remarks}
\newtheorem{ex}[lemma]{Example}
\newtheorem{exs}[lemma]{Examples}

\newtheorem{claim}[lemma]{Claim}

\newcommand{\A}{\mathbf{A}}

\renewcommand{\P}{\mathbf{P}}

\newcommand{\Z}{\mathbb{Z}}
\newcommand{\sA}{\mathcal{A}}
\newcommand{\sB}{\mathcal{B}}
\newcommand{\sC}{\mathcal{C}}
\newcommand{\sD}{\mathcal{D}}
\newcommand{\sE}{\mathcal{E}}
\newcommand{\sF}{\mathcal{F}}

\newcommand{\sO}{\mathcal{O}}

\newcommand{\sS}{\mathcal{S}}
\newcommand{\sT}{\mathcal{T}}

\newcommand{\bZ}{\mathbb{Z}}

\newcommand{\Xb}{{\overline{X}}}

\newcommand{\alphab}{\ol{\alpha}}
\newcommand{\deltab}{\ol{\delta}}

\newcommand{\Mod}{\operatorname{Mod}\hbox{--}}

\newcommand{\Cor}{\operatorname{\mathbf{Cor}}}

\newcommand{\ulSigma}{\ul{\Sigma}}

\newcommand{\Ext}{\operatorname{Ext}}
\newcommand{\ul}[1]{{\underline{#1}}}

\newcommand{\Set}{{\operatorname{\mathbf{Set}}}}
\newcommand{\PST}{{\operatorname{\mathbf{PST}}}}

\newcommand{\NST}{\operatorname{\mathbf{NST}}}

\newcommand{\Hom}{\operatorname{Hom}}

\newcommand{\Coker}{\operatorname{Coker}}

\newcommand{\Spec}{\operatorname{Spec}}

\newcommand{\Sm}{\operatorname{\mathbf{Sm}}}
\newcommand{\Sch}{\operatorname{\mathbf{Sch}}}
\newcommand{\Ab}{\operatorname{\mathbf{Ab}}}

\newcommand{\by}{\xrightarrow}
\newcommand{\yb}{\xleftarrow}
\newcommand{\iso}{\by{\sim}}

\newcommand{\pro}[1]{\text{\rm pro}_{#1}\text{\rm--}}
\newcommand{\tr}{{\operatorname{tr}}}
\newcommand{\proper}{{\operatorname{prop}}}

\newcommand{\sat}{{\operatorname{sat}}}

\newcommand{\fin}{{\operatorname{fin}}}
\renewcommand{\o}{{\operatorname{o}}}
\newcommand{\op}{{\operatorname{op}}}

\newcommand{\red}{{\operatorname{red}}}

\newcommand{\Nis}{{\operatorname{Nis}}}
\newcommand{\et}{{\operatorname{\acute{e}t}}}
\newcommand{\tto}{\dashrightarrow}
\newcommand{\inj}{\hookrightarrow}

\newcommand{\id}{{\operatorname{Id}}}

\newcommand{\adm}{{\operatorname{adm}}}

\newcommand{\codim}{{\operatorname{codim}}}
\newcommand{\ch}{{\operatorname{ch}}}

\newcommand{\pr}{{\operatorname{pr}}}

\renewcommand{\lim}{\operatornamewithlimits{\varprojlim}}
\newcommand{\colim}{\operatornamewithlimits{\varinjlim}}

\newcommand{\ol}{\overline}

\renewcommand{\phi}{\varphi}
\renewcommand{\epsilon}{\varepsilon}

\newcommand{\MCor}{\operatorname{\mathbf{MCor}}}
\newcommand{\MP}{\operatorname{\mathbf{MSm}}}
\newcommand{\MSm}{\operatorname{\mathbf{MSm}}}
\newcommand{\MPS}{\operatorname{\mathbf{MPS}}}
\newcommand{\MPST}{\operatorname{\mathbf{MPST}}}

\newcommand{\Bl}{{\mathbf{Bl}}}

\newcommand{\Sq}{{\operatorname{\mathbf{Sq}}}}

\newcommand{\bcube}{{\ol{\square}}}

\newcommand{\Mb}{{\overline{M}}}
\newcommand{\Nb}{{\overline{N}}}
\newcommand{\Lb}{{\overline{L}}}
\newcommand{\Zb}{{\overline{Z}}}

\newcommand{\ulMP}{\operatorname{\mathbf{\underline{M}Sm}}}
\newcommand{\ulMSm}{\operatorname{\mathbf{\underline{M}Sm}}}
\newcommand{\ulMNS}{\operatorname{\mathbf{\underline{M}NS}}}
\newcommand{\ulMPS}{\operatorname{\mathbf{\underline{M}PS}}}
\newcommand{\ulMPST}{\operatorname{\mathbf{\underline{M}PST}}}
\newcommand{\ulMNST}{\operatorname{\mathbf{\underline{M}NST}}}
\newcommand{\ulMCor}{\operatorname{\mathbf{\underline{M}Cor}}}
\newcommand{\ulomega}{\underline{\omega}}

\newcommand{\Comp}{\operatorname{\mathbf{Comp}}}
\newcommand{\CompM}{\Comp(M)}

\newcommand{\MV}{\operatorname{MV}}
\newcommand{\ulMV}{\operatorname{\underline{MV}}}
\newcommand{\ulMVfin}{\operatorname{\underline{MV}^{\mathrm{fin}}}}

\def\Zp{\Z^p}

\def\bZ{\mathbb{Z}}
\def\Ztr{\bZ_\tr}

\newcommand{\til}{\widetilde}

\def\Ninf{N^\infty}
\def\Minf{M^\infty}

\def\Minf{M^\infty}

\newcounter{spec}
\newenvironment{thlist}{\begin{list}{\rm{(\roman{spec})}}%
{\usecounter{spec}\labelwidth=20pt\itemindent=0pt\labelsep=10pt}}%
{\end{list}}%

\def\ulaNis{\underline{a}_{\Nis}}

\def\Comp{\Comp^{\fin}}

\def\ulb{\ul{b}}

\def\MSm{\operatorname{\mathbf{MSm}}}

\def\ulMSm{\operatorname{\mathbf{\ul{M}Sm}}}

\def\ulMPS{\operatorname{\mathbf{\ul{M}PS}}}

\def\ulMNS{\operatorname{\mathbf{\ul{M}NS}}}

\def\ulMNSfin{\operatorname{\mathbf{\ul{M}NS}^{\fin}}}
\def\ulMNST{\operatorname{\mathbf{\ul{M}NST}}}

\def\Comp{\operatorname{\mathbf{Comp}}}

\EpigaVolumeYear{5}{2021} \EpigaArticleNr{1} \ReceivedOn{December 17,
2019}
\InFinalFormOn{November 26, 2020}
\AcceptedOn{December 19, 2020}

\title{Motives with modulus, I: \\
Modulus sheaves with transfers for non-proper modulus pairs}
\titlemark{Motives with modulus, I}

\author{Bruno Kahn}
\address{IMJ-PRG, Case 247, 4 place Jussieu, 75252 Paris Cedex 05, France}
\email{bruno.kahn@imj-prg.fr}
\author{Hiroyasu Miyazaki}
\address{RIKEN iTHEMS, Wako, Saitama 351-0198, Japan}
\email{hiroyasu.miyazaki@riken.jp}
\author{Shuji Saito}
\address{Graduate School of Mathematical Sciences, University of Tokyo,
3-8-1 Komaba, Tokyo 153-8941, Japan}
\email{sshuji@msb.biglobe.ne.jp}
\author{Takao Yamazaki}
\address{Institute of Mathematics, Tohoku University, Aoba, Sendai 980-8578, Japan}
\email{ytakao@math.tohoku.ac.jp}

\authormark{B.~Kahn, H.~Miyazaki, S.~Saito, and T.~Yamazaki}

\AbstractInEnglish{\footnotesize{We develop a theory of \emph{modulus sheaves with transfers}, which generalizes
Voevodsky's theory of sheaves with transfers. 
This paper and its sequel are foundational  for the theory of motives with modulus, which is developed in \cite{motmod2}.}}



\TitleInFrench{Motifs avec modules, I : faisceaux avec transferts pour les couples modulaires non propres}

\AbstractInFrench{\footnotesize{Nous pr\'esentons une th\'eorie de \emph{faisceaux modulaires avec transferts} qui g\'en\'eralise la th\'eorie des faisceaux avec transferts de Voevodsky. 
Cet article et celui qui lui fait suite sont les fondements d'une th\'eorie de \emph{motifs avec modules}, qui est d\'evelopp\'ee dans \cite{motmod2}.}}

\acknowledgement{\scriptsize{The first author acknowledges the support of Agence Nationale de la Recherche (ANR) under reference ANR-12-BL01-0005. 
The work of the second author is supported in part by Fondation des Sciences Math\'ematiques de Paris (FSMP), and in part by RIKEN Special Postdoctoral Researchers (SPDR) Program and RIKEN Interdisciplinary Theoretical and Mathematical Sciences (iTHEMS) Program. 
The third author is supported by JSPS KAKENHI Grant (15H03606).
The fourth author is supported by JSPS KAKENHI Grant (15K04773). 
}}

\begin{document}


\removeabove{0.6cm}
\removebetween{0.6cm}
\removebelow{0.6cm}

\maketitle

\begin{prelims}

\DisplayAbstractInEnglish

\bigskip

\DisplayKeyWords

\medskip

\DisplayMSCclass

\bigskip

\languagesection{Fran\c{c}ais}

\bigskip

\DisplayTitleInFrench

\medskip

\DisplayAbstractInFrench

\end{prelims}


\newpage

\setcounter{tocdepth}{1}

\tableofcontents


\section*{Introduction}\addcontentsline{toc}{section}{Introduction}

The aim of this paper is to lay a foundation for a theory of 
\emph{motives with modulus},
which will be completed in \cite{motmod2},
generalizing Voevodsky's theory of motives.
Voevodsky's construction  is based on $\A^1$-invariance. It captures many important invariants such as Bloch's higher Chow groups, but not their natural generalisations like additive Chow groups  \cite{BE,Pa} or higher Chow groups with modulus \cite{BS}. Our basic motivation is to build a theory that captures such non $\A^1$-invariant phenomena, as an extension of \cite{rec}.

Let $\Sm$ be the category of smooth separated schemes of finite type 
over a field $k$.
 Voevodsky's construction starts from  an additive category $\Cor$, whose objects are those of $\Sm$ and morphisms are finite correspondences. 
We define $\PST$ as the category of 
additive presheaves of abelian groups on $\Cor$
(\emph{i.e.} functors $\Cor \to \Ab$ that commute with finite sums).
Let $\NST\subset \PST$ be the full subcategory of 
those objects $F\in \PST$ whose restrictions $F_X$ to 
$X_{\Nis}$ is a sheaf for any $X \in \Sm$,
where $X_\Nis$ denotes the small Nisnevich site of $X$,
that is, the category of all \'etale schemes over $X$
equipped with the Nisnevich topology.
Objects of $\NST$ are called \emph{(Nisnevich) sheaves with transfers}.
For $F\in \NST$, we write
\[ H^i_\Nis(X,F) = H^i(X_\Nis, F_X).\]

The following result of Voevodsky \cite[Theorem 3.1.4]{voetri} plays a fundamental r\^ole in his theory of motives.

\begin{theorem}[Voevodsky]\label{thm;Voevodksy}
The following assertions hold.
\begin{itemize}
\item[(1)]
The inclusion $\NST\to \PST$ has an exact left adjoint $a^V_\Nis$ such that for any $F\in \PST$ and $X\in \Sm$, 
$(a^V_\Nis F)_X$ is the Nisnevich sheafication of 
$F_X$ as a presheaf on $X_\Nis$. In particular $\NST$ is a Grothendieck abelian category.
\item[(2)]
For $X\in \Sm$, let $\Ztr(X)=\Cor(-,X) \in \PST$ be the associated representable additive presheaf. Then we have $\Ztr(X)\in \NST$ and there is a canonical
isomorphism for any $i\ge 0$ and $F\in \NST$:
\[H^i_\Nis(X,F)\simeq \Ext^i_{\NST}(\Ztr(X),F) .\]
\end{itemize}
\end{theorem}

Our basic principle for generalizing Voevodsky's theory of sheaves with transfers is that the category $\Cor$ should be replaced by the larger category of \emph{modulus pairs}, $\ulMCor$: objects are pairs 
$M=(\ol{M}, \Minf)$ consisting of a separated $k$-scheme of finite type $\ol{M}$ and an effective (possibly empty) Cartier divisor $\Minf$ on it such that the complement 
$M^\circ:=\ol{M}\setminus \Minf$ is smooth over $k$. The group $\ulMCor(M,N)$ of morphisms is defined as the subgroup of $\Cor(M^\circ,N^\circ)$ 
consisting of finite correspondences between $M^\o$ and $N^\o$ whose closures in 
$\ol{M} \times_k \ol{N}$ are proper\footnote{Here we stress that we do not assume it is finite over $\ol{M}$.} over $\ol{M}$
and satisfy certain admissibility conditions with respect to $\Minf$ and $\Ninf$
(see Definition \ref{d2.2}). 
Let $\MCor\subset \ulMCor$ be the full subcategory consisting of objects $(\ol{M},\Minf)$ with $\ol{M}$ proper over $k$. 

We then define $\ulMPST$ (resp. $\MPST$) as the category of additive presheaves of abelian groups 
on $\ulMCor$ (resp. $\MCor$). We have a functor 
\[\omega:\ulMCor \to \Cor, \quad (\Mb,\Minf) \mapsto \Mb - |\Minf|,\]
and two pairs of adjunctions
\[\MPST\begin{smallmatrix} \tau^*\\ \longleftarrow\\ \tau_!\\ \longrightarrow\\
\end{smallmatrix}\ulMPST, 
\quad
\MPST
\begin{smallmatrix} \omega^*\\ \longleftarrow\\ \omega_!\\ \longrightarrow\\
\end{smallmatrix}\PST,
\]
where $\tau^*$ is induced by the inclusion $\tau:\MCor\to \ulMCor$ and
$\tau_!$ is its left Kan extension, and $\omega^*$ is induced by $\omega$ and $\omega_!$ is its left Kan extension 
(see Propositions \ref{eq.tau} and \ref{lem:counit}). 

The main aim of this paper is to develop a \emph{sheaf theory} on $\ulMCor$
generalizing Voevodsky's theory. 
For $M=(\Mb,\Minf)\in \ulMCor$ and $F \in \ulMPST$, 
let $F_M$ be the presheaf on $\Mb_\Nis$ which associates $F(U,\Minf\times_{\Mb} U)$ to an \'etale map $U\to \Mb$. 

\begin{Df}\label{defintro;sheavesulMCor}
We define $\ulMNST$ to be 
the full subcategory of $\ulMPST$ of objects $F$ such that $F_M$ is a Nisnevich sheaf on $\Mb$ for any $M \in \ulMCor$.
\end{Df}

For $F\in \ulMPST$ and $M=(\Mb,\Minf)$, let $(F_M)_\Nis$ be the Nisnevich sheafication of the preshseaf $F_M$ on $\Mb_\Nis$. 
Let $\ul{\Sigma}^\fin$ be the subcategory of $\ulMCor$ 
which has the same objects as $\ulMCor$ 
and such that a morphism $f \in \ulMCor(M, N)$
belongs to $\ul{\Sigma}^\fin$ if and only if
$f^\o \in \Cor(M^\o, N^\o)$ is the graph of an isomorphism
$M^\o \iso N^\o$ in $\Sm$
that extends to a proper morphism $\ol{f} : \ol{M} \to \ol{N}$
of $k$-schemes such that $\Minf=\ol{f}^*\Ninf$.
(See Theorems \ref{thm:sheafification-ulMNST}, \ref{c3.1v}
and Lemma \ref{lcom3-2}.)

\begin{theorem}\label{thmintro;ulMNST}
The following assertions hold.
\begin{itemize}
\item[(1)]
The inclusion $\ulMNST\to \ulMPST$ has an exact left adjoint $\ulaNis$
such that 
\[(\ulaNis F)(M) =\colim_{N\in \ul{\Sigma}^\fin\downarrow M} (F_N)_\Nis(N)\]
for every $F\in \ulMPST$ and $M \in \ulMCor$. 
In particular $\ulMNST$ is a Grothendieck abelian category.
(See \S \ref{sec:pro-obj} for the comma category
$\ul{\Sigma}^\fin\downarrow M$.)
\item[(2)]
For $M\in \ulMCor$, let $\Ztr(M)=\ulMCor(-,M) \in \ulMPST$ be the associated representable presheaf. Then we have $\Ztr(M)\in \ulMNST$ and there is a canonical isomorphism for any $i\ge 0$ and $F\in \ulMNST$:
\[\Ext^i_{\ulMNST}(\Z_\tr(M),F)\simeq \colim_{N\in \ul{\Sigma}^{\fin}\downarrow M}
H_\Nis^i(\Nb,F_N).
\]
\end{itemize}
\end{theorem}

\begin{Rk}\label{rem;thmintro;ulMNST}
Theorem \ref{thmintro;ulMNST} (2) describes
the extension groups in $\ulMNST$
in terms of classical cohomology.
It also implies that the formation 
\[M \mapsto \colim_{N\in \ul{\Sigma}^{\fin}\downarrow M}
H^i_\Nis(\Nb,F_N)\]
is contravariantly functorial for morphisms in $\ulMCor$, which does not follow
immediately from the definition.
\end{Rk}

The preprint \cite{motmod} contained a mistake, pointed out by Joseph Ayoub: namely, Proposition 3.5.3 of \emph{loc.~cit.} is false. Theorem \ref{thmintro;ulMNST} (1) shows that the only false thing in that proposition is that the functor $\ul{b}^\Nis$ of \emph{loc.~cit.} is not exact, but only left exact (see Proposition \ref{lem;b!ulMNST} of the present paper.) This weakens \cite[Proposition~3.6.2]{motmod} into Theorem \ref{thmintro;ulMNST} (2); see however Question \ref{qn;cdhdescnet} below. What we gain in the present correction is that the notion of sheaf, which was artificially developed in \cite{motmod} for $\ulMCor$, corresponds now to a genuine Grothendieck topology. 

Another proposition incorrectly proven in \cite{motmod} was Proposition 3.7.3. In Part II of this work \cite{modsheafII}, we correct this proof and recover the proposition in full, hence get a good sheaf theory also for \emph{proper} modulus pairs. 
This allows us to develop the categories of motives again in \cite{motmod2}.

%

In the last part of this introduction, we raise the following question.
Its affirmative answer would simplify the right hand side of
Theorem \ref{thmintro;ulMNST} (2)
under two additional conditions (i) and (ii) below.
(These conditions turn out to be essential in \cite{shuji}.)

\begin{Qn}\label{qn;cdhdescnet}
Assume that $F\in \ulMNST$ satisfies the following conditions:
\begin{itemize}
\item[(i)]
$F$ is $\bcube$-invariant, namely, for any $M=(\Mb,\Minf)\in \ulMCor$ the map $F(M) \to F(M\otimes \bcube)$ is an isomorphism, where 
\begin{align*}
\bcube = (\P^1 ,\infty ), \ \ M\otimes \bcube =(\ol{M} \times \P^1, M^\infty \times \P^1 + \ol{M} \times (\infty)).
\end{align*}
\item[(ii)]
$F$ lies in the essential image of $\tau_!:\MPST \to \ulMPST$.
\end{itemize}
Then, is the map 
\[ H^q(\ol{M}_\Nis,F_M) \to \colim_{N\in \ul{\Sigma}^{\fin}\downarrow M} 
H^q(\ol{N}_\Nis,F_N)\]
an isomorphism for $M\in \ulMCor_{ls}$?
Here $\ulMCor_{ls}$ denotes the full subcategory of $\ulMCor$ consisting of the objects $M=(\Mb,\Minf)$ such that $\Mb\in \Sm$ and $|\Minf|$ is a simple normal crossing divisor on $\Mb$. 
\end{Qn}

If $\ch(k)=0$, by resolution of marked ideals (\cite[the case $d=1$ of Theorem~1.3]{BM}),
the above question is reduced to the following.

\begin{Qn}\label{qn2;cdhdescnet}
Let the assumptions be as in Question \ref{qn;cdhdescnet} and 
$M=(\Mb,\Minf)\allowbreak \in \ulMCor_{ls}$.
Let $Z\subset \Minf$ be a regular closed subscheme such that, 
for any point $x$ of $Z$, there exists a system $z_1, \ldots, z_d$ of regular parameters of $\Mb$ at $x$ (with $d=\dim_x \Mb$) satisfying the following conditions:
\begin{itemize}
\item Locally at $x$, $Z =\{z_1= \dots = z_r=0\}$ for 
$r=\codim_{\ol{M}} Z$.
\item Locally at $x$, $|\Minf| = \{\prod_{j \in J} z_j=0 \}$ 
for some $J\subset \{1,\dots,r\}$.
\end{itemize}

Consider $\pi: \Nb=\Bl_Z(\Mb)\to \Mb$ and 
$\Ninf=\Nb \times_{\Mb} \Minf$.
Then, is the map 
\[ H^q(\ol{M}_\Nis,F_M) \to H^q(\ol{N}_\Nis,F_N).\]
an isomorphism?
\end{Qn}

\subsection*{Acknowledgements}
Part of this work was done while the authors stayed at the university of Regensburg supported by the SFB grant ``Higher Invariants". Another part was done in a Research in trio in CIRM, Luminy. Yet another part was done while the fourth author was visiting IMJ-PRG supported by Fondation des Sciences Math\'ematiques de Paris. We are grateful to the support and hospitality received in all places.

We thank Ofer Gabber and Michel Raynaud for their help with Lemma \ref{lrg}, and Kay R\"ulling for pointing out an error and correcting Definition \ref{d2.6}.

We are very grateful to Joseph Ayoub for pointing out 
 a flaw on the computation of the sheafification functor $\ul{a}_{\Nis}$
 and on the non-exactness of the functor $\ul{b}^\Nis$
 in the earlier version. The authors believe that the whole theory has been deepened by the effort to fix it. 
We also thank the referees for a careful reading and many useful comments.

Finally, the influence of Voevodsky's ideas is all-pervasive, as will be evident when reading this paper.

\subsection*{Notation and conventions}
In the whole paper we fix a base field $k$. 
Let $\Sm$ be the category of separated smooth schemes of finite type over $k$,
and let $\Sch$ be the category of separated schemes of finite type over $k$.
We write $\Cor$ for Voevodsky's 
category of finite correspondences \cite{voetri}.

\section{Modulus pairs and admissible correspondences}

\subsection{Admissible correspondences} 

\begin{definition}\label{d2.2}\
\begin{enumerate}
\item
A \emph{modulus pair} $M$ 
consists of $\ol{M} \in \Sch$
and an effective Cartier divisor $M^\infty \subseteq \ol{M}$
such that the open subset $M^\o:=\ol{M} - |M^\infty|$ is smooth over $k$. (The case $|M^\infty|=\emptyset$ is allowed.) We say that $M$ is \emph{proper} if $\Mb$ is proper over $k$.

We write 
$M=(\ol{M}, M^{\infty})$,
since $M$ is completely determined by the pair,
although
we regard $M^\o$ as the main part of $M$.
We call $\ol{M}$ the \emph{ambient space} of $M$ and $M^\o$ the \emph{interior} of $M$. 
\item
Let $M_1, M_2$ be  modulus pairs.
Let $Z \in \Cor(M_1^\o, M_2^\o)$ be an elementary correspondence
(\emph{i.e.} an integral closed subscheme of $M_1^\o \times M_2^\o$
which is finite and surjective over an irreducible component of $M_1^\o$).
We write $\ol{Z}^N$ for the normalization
of the closure $\ol{Z}$ of $Z$ in $\ol{M}_1 \times \ol{M}_2$
and $p_i : \ol{Z}^N \to \ol{M}_i$
for the canonical morphisms for $i=1, 2$.
We say $Z$ is \emph{admissible} for $(M_1, M_2)$
if $p_1^* M_1^\infty \geq p_2^* M_2^\infty$. 
An element of $\Cor(M_1^\o, M_2^\o)$ is called admissible
if all of its irreducible components are admissible.
We write $\Cor_\adm(M_1, M_2)$ for the subgroup of $\Cor(M_1^\o, M_2^\o)$ consisting of all admissible correspondences. 
\end{enumerate}
\end{definition}

\begin{remarks}\ \label{r2.1} 
\begin{enumerate}
\item 
In \cite[Definition~2.1.1]{rec}, we used a different notion of modulus pair, where $\Mb$ is supposed proper, $M^\o$ smooth quasi-affine and 
$M^\infty$ is any closed subscheme of $\ol{M}$. 
Definition \ref{d2.2} (1) is the right one for the present work. 
Definition \ref{d2.2} (2) is the same as \cite[Definition~2.6.1]{rec}, mutatis mutandis. 
An analogous condition was considered much earlier 
in the context of the additive Chow groups
(see, e.g. \cite[(6.4)]{BE}, \cite[Definition~2.2]{Pa}, \cite[Definition~3.1]{Kay}).
\item 
In the first version of this paper,
we imposed the condition that $\ol{M}$ is locally integral;
it is now removed.
The main reason for this change is that
this condition is not stable under products or extension of the base field. The next remark shows that this removal is reasonable (see also Remark \ref{l3.3}).
\item 
Let $M$ be a modulus pair.
Then $M^\o$ is dense in $\Mb$, since the Cartier divisor $M^\infty$ is everywhere of codimension $1$. Moreover,  $\ol{M}$ is reduced.
(In particular, $\ol{M}$ has no embedded component.)
Indeed, take $x \in \ol{M}$
and let $f \in \sO_{\ol{M}, x}$ be a local equation for $M^\infty$.
Then $f$ is not a zero-divisor (since $M^\infty$ is Cartier), 
and hence
$\sO_{\ol{M}, x} \to \sO_{\ol{M}, x}[1/f]$ is injective,
but $\sO_{\ol{M}, x}[1/f]$ is reduced as $M^\o$ is smooth. In particular, $\Mb$ is integral if $M^\o$ is.
\item Let $M$ be a modulus pair, and let $f:\ol{M}_1\to \ol{M}$ be a morphism such that $f(T)\not\subset |M^\infty|$ for any irreducible component $T$ of $\ol{M}_1$ and $M_1^\o:=\ol{M}_1-|f^*M^\infty|$ is smooth. Then $M_1=(\ol{M}_1,f^*M^\infty)$ defines a modulus pair. 
We call it the \emph{minimal modulus structure induced by $f$}. We shall use this construction several times.
Also, $f$ defines a minimal morphism $f:M_1\to M$ in the sense of Definition~\ref{def:minimality} below.
\item If $Z$ is an admissible elementary correspondence as in Definition \ref{d2.2} (2), then 
$$|M_1^\infty\times \ol{M}_2|\cap \ol{Z} 
\supseteq |\ol{M}_1\times M_2^\infty|\cap \ol{Z}$$
since $\ol{Z}^N\to \ol{Z}$ is surjective. On the other hand, 
the inequality 
$(M_1^\infty\times \ol{M}_2)|_{\ol{Z}} 
\ge
(\ol{M}_1\times M_2^\infty)|_{\ol{Z}}$
may fail.
%
%
%
As an example, let $C$ be the affine cusp curve $\Spec k[x,y]/(x^2-y^3)$. Its normalization is $\A^1$, via the morphism $t\mapsto (t^3,t^2)$. Let $M_1=(C,(x))$ and $M_2=(C,(y))$. Then $1_C$ defines an adm\-iss\-ible correspondence $M_1\to M_2$, 
even though $(x) \ge (y)$ does not hold on $C$.
\end{enumerate}
\end{remarks}

The following lemma will play a key r\^ole:

\begin{lemma}\label{lem:mod-exists}
Let $\Xb \in \Sch$ and let $X$ be an open dense subscheme of $\ol{X}$. Assume that $X \in \Sm$ and that $\Xb-X$ is the support of a Cartier divisor. Then for any  modulus pair $N$
we have
\[
 \bigcup_M \Cor_\adm(M, N) = \Cor(X, N^\o),
\]
where $M$ ranges over all modulus pairs
such that $\Mb=\Xb$ and $M^\o = X$.
$($Note that by definition
we have $\Cor_\adm(M, N) \subset \Cor(X, N^\o)$.$)$
%
\end{lemma}

\begin{proof}
This is proven in {\cite[Lemma 2.6.2]{rec}}. In \emph{loc.~cit.} $X$ and $N^\o$ are assumed to be quasi-affine,
and $\ol{X}$ and $\ol{N}$ proper and normal (see Remark \ref{r2.1}).
But these assumptions are not used in the proof. (Nor is the assumption on Cartier divisors, but the latter is essential for the proof of Proposition \ref{lem:comp-admcorr} below.)
\end{proof}

\subsection{Composition} 
To discuss composability of admissible correspondences, 
we need the following lemma of 
Krishna and Park \cite[Lemma 2.2]{KP}.

\begin{lemma}\label{lKL} 
Let $f:X\to Y$ be a surjective morphism of normal integral schemes, 
and let $D,D'$ be two Cartier divisors on $Y$. 
If $f^*D'\le f^*D$, then $D'\le D$.
\end{lemma}

We also need the following ``containment lemma'' from \cite[Proposition~2.4]{KP}, \cite[Lemma~2.1]{BS}, \cite[Lemma~2.4]{Mi}.
We provide a proof for self-containedness.

\begin{lemma}\label{containment-lemma}
Let $M = (\ol{M},M^\infty)$ be a modulus pair. 
Let $V'\subset V \subset M^\circ = \ol{M} - |M^\infty |$ be two integral closed subschemes. 
Let $\ol{V}$ and $\ol{V'}$ be their closures in $\ol{M}$ and $\ol{V}^N \to \ol{V}$, $\ol{V'}^{N} \to \ol{V'}$ the normalizations. 
If $M^\infty |_{\overline{V}^N }$ is effective, so is $M^\infty |_{\ol{V'}^N }$.
\end{lemma}

\begin{proof}
Set $Z:=\overline{V}^N \times_{\overline{V}} \overline{V'}$ and consider the following commutative diagram:
\[\xymatrix{
Z_\gamma^N \ar[r] \ar@(ur,ul)[rrr]^f \ar[d]^{h}_{\mathrm{fin.\,\, surj.}} & Z_\gamma \ar[r] \ar[rd]_(0.4){\mathrm{fin.\,\, surj.}} & Z \ar[r] \ar[d] \ar@{}[rd] |\square & \overline{V}^N \ar[d] \ar[rd] \\
\overline{V'}^N \ar[rr] & & \overline{V'} \ar[r]_{\mathrm{incl.}} & \overline{V} \ar[r]_{\mathrm{incl.}} & \overline{M}.
}\]
Here, $Z_\gamma \subset Z$ is an irreducible component
 of $Z$ such that the composite map \[Z_\gamma \hookrightarrow Z \to \overline{V'} \] is finite surjective. To see that such a $Z_\gamma$ exists, it suffices to note that $\overline{V}^N \to \overline{V}$ is finite surjective, hence so is its base change $Z \to \overline{V'}$ (recall that for any scheme $S$ of finite type over $k$, the normalization $S^N \to S$ is a finite surjective morphism).  Then $Z_\gamma^N$ is also irreducible. 
Since $Z_\gamma^N \to \overline{V'}$ is dominant, the vertical map $h$ on the left exists by the universal property of normalization, and 
is finite surjective. 
Note that we can pullback the Cartier divisor $M^\infty $ to any scheme except for $Z$ in the diagram, since none of their irreducible components maps into the support $|M^\infty | \subset \overline{M}$.
Since the pullback of an effective Cartier divisor is effective, the assumption that $M^\infty |_{\overline{V}^N }$ is effective implies that \[f^\ast (M^\infty |_{\overline{V}^N }) = M^\infty |_{Z_\gamma^N } = h^\ast (M^\infty |_{\overline{V'}^N })\] is effective. By Lemma \ref{lKL}, $M^\infty |_{\overline{V'}^N }$ is effective since $h$ is surjective.
This finishes the proof.
\end{proof}

\begin{defn}\label{dcomp} Let $M_1, M_2, M_3$ be three  modulus pairs,
and let us consider $\alpha \in \Cor_\adm(M_1, M_2)$ and  $\beta \in \Cor_\adm(M_2, M_3)$.
We say that $\alpha$ and $\beta$ are \emph{composable} if their composition $\beta \alpha$ in $\Cor(M_1^\o, M_3^\o)$ is admissible.
\end{defn}

\begin{prop}\label{lem:comp-admcorr} With the above notations, assume $\alpha$ and $\beta$ are integral and let $\bar\alpha$ and $\bar\beta$ be their closures in $\Mb_1\times \Mb_2$ and $\Mb_2\times \Mb_3$ respectively. Then
$\alpha$ and $\beta$ are composable provided the projection $\bar\alpha\times_{\Mb_2}\bar\beta\to \Mb_1\times \Mb_3$ is proper. This happens in the following cases:
\begin{thlist}
\item $\bar\alpha\to \Mb_1$ is proper.
\item $\bar\beta\to \Mb_3$ is proper.
\item $\Mb_2$ is proper over $k$.
\end{thlist}
\end{prop}

\begin{proof} Note that $\alpha \times_{M_2^\o} \beta$ is a closed subscheme of 
$(M_1^\o \times M_2^\o) \times_{M_2^\o} (M_2^\o \times M_3^\o)
=M_1^\o \times M_2^\o \times M_3^\o$; we have
$|\beta \alpha| = |p_{13*}( \alpha \times_{M_2^\o} \beta)|$
where $p_{13} : M_1^\o \times M_2^\o \times M_3^\o
\to M_1^\o \times M_3^\o$ is the projection. Let $\gamma$ be a component of $\alpha \times_{M_2^\o} \beta$.
We have a commutative diagram
\[\xymatrix{
\gamma  \ar@{^{(}->}[r]  \ar@{^{(}->}[d]& \alpha \times_{M_2^\o} \beta  \ar@{^{(}->}[r] \ar@{^{(}->}[d]& M_1^\o\times M_2^\o \times M_3^\o\ar[r]^(0.6){p_{13}} \ar@{^{(}->}[d] &M_1^\o\times  M_3^\o \ar@{^{(}->}[d]& \delta  \ar@{_{(}->}[l] \ar@{^{(}->}[d]\\
\bar\gamma  
\ar@{^{(}->}[r] & \bar\alpha \times_{\Mb_2} \bar\beta  \ar@{^{(}->}[r]& \Mb_1\times \Mb_2 \times \Mb_3 \ar@{->}[r]^(0.6){\ol{p}_{13}} &\Mb_1\times  \Mb_3& \bar\delta \ar@{_{(}->}[l] 
}\]
where 
$\ol{p}_{ij} : \Mb_1\times \Mb_2 \times \Mb_3 \to  \Mb_i\times \Mb_j$ denotes the projection,
$\delta=p_{13}(\gamma)$,
and $\bar{}$ denotes closure.
The hypothesis implies that $\bar \gamma\to \bar\delta$ is proper surjective.
The same holds for  $\pi_{\gamma \delta}^N$ 
appearing in
the second of the two other commutative diagrams:
\[
\xymatrix{
\bar\alpha& \bar\alpha^N \ar[r]^-{\phi_{\alpha}}\ar[l]
& \Mb_1 \times \Mb_2
\\
\bar\gamma\ar[u]_-{\pi_{\gamma \alpha}}\ar[d]^-{\pi_{\gamma\beta}} & \bar\gamma^N 
\ar[r]^-{\phi_\gamma}\ar[l] &\Mb_1\times \Mb_2\times \Mb_3\ar[d]^-{\bar p_{23}}
\ar[u]_-{\bar p_{12}}
\\
\bar \beta&\bar\beta^N \ar[r]^-{\phi_{\beta}}\ar[l] & \Mb_2 \times \Mb_3
} 
\quad
\xymatrix{
\bar\gamma^N 
\ar[r]^-{\phi_\gamma}\ar[d]^-{\pi_{\gamma\delta}^N} &\Mb_1\times \Mb_2\times \Mb_3\ar[d]^-{\bar p_{13}}
\\
\bar\delta^N \ar[r]^-{\phi_\delta} &\Mb_1\times \Mb_3
} 
\]
where $^N$ means normalisation. (Note that $\pi_{\gamma\alpha}$ and $\pi_{\gamma\beta}$ need not extend to the normalisations, as they need not be dominant.) We have the admissibility conditions for $\alpha$ and $\beta$:
\begin{align}
\phi_\alpha^*(\Mb_1\times M_2^\infty)&\le \phi_\alpha^*(M_1^\infty\times \Mb_2)\label{eq1a}\\
\phi_\beta^*(\Mb_2\times M_3^\infty)&\le \phi_\beta^*(M_2^\infty\times \Mb_3).\label{eq1b}
\end{align}
Applying\footnote{To apply this lemma, factor $\pi_{\gamma\alpha}$ and $\pi_{\gamma\beta}$ into dominant morphisms followed by closed immersions.} Lemma \ref{containment-lemma}, we get inequalities
\[ \phi_\gamma^*(\Mb_1\times \Mb_2\times M_3^\infty)\le  \phi_\gamma^*(\Mb_1\times M_2^\infty\times \Mb_3)\le  \phi_\gamma^*(M_1^\infty\times \Mb_2\times \Mb_3),
\]
which implies 
by the right half of the above diagram 
\begin{equation}\label{eq1c}
(\pi_{\gamma\delta}^N)^* \phi_\delta^*(\Mb_1\times M_3^\infty)\le (\pi_{\gamma\delta}^N)^*\phi_\delta^*(M_1^\infty\times \Mb_3)
\end{equation}
hence $\phi_\delta^*(\Mb_1\times M_3^\infty)\le \phi_\delta^*(M_1^\infty\times \Mb_3) $ by Lemma \ref{lKL}. 

Finally, 
one trivially checks that (i) or (ii) implies that 
the projection $\ol{\alpha} \times_{\ol{M}_2} \ol{\beta}
\to \ol{M}_1 \times \ol{M}_3$ is proper,
and that (iii) implies both of (i) and (ii).
\end{proof}

\begin{ex}\label{ex1.1} Let $\Mb_1=\Mb_3=\P^1$, $\Mb_2=\A^1$, $\Mb_i^\o=\A^1$, $M_1^\infty =\infty$, $M_2^\infty=0$, $M_3^\infty=2 \cdot \infty$, and $\alpha=\beta=$ graph of the identity on $\A^1$. Then $\alpha$ and $\beta$ are admissible but $\beta\circ \alpha$ is not admissible
because $\infty \ge 2 \cdot \infty$ does not hold.
(Note that 
neither of $\ol{\alpha}=\alpha$ or $\ol{\beta}=\beta$
is proper over $\P^1$.)
\end{ex}

\begin{defn}\label{d2.4} Let $M,N$ be two modulus pairs. A correspondence $\alpha\in \Cor(M^\o, N^\o)$ is \emph{left proper} (relatively to $M,N$) if the closures of all components of $\alpha$ are proper over $\Mb$; this is automatic if $\Nb$ is proper.  
\end{defn}

\begin{prop}\label{p1a} Let $M_1,M_2,M_3$ be three modulus pairs and let $\alpha\in \Cor(M_1^\o,M_2^\o)$, $ \beta\in \Cor(M_2^\o,M_3^\o)$ be left proper. Then $\beta\alpha$ is left proper.
\end{prop}

\begin{proof}  
We may assume $\alpha$ and $\beta$ are irreducible. 
The assumption on $\beta$ means
$\ol{\beta} \to \ol{M}_2$ is proper, 
hence so is its base change
$\ol{\alpha} \times_{\Mb_2} \ol{\beta} \to \ol{\alpha}$. 
The assumption on $\alpha$ means 
$\ol{\alpha} \to \ol{M}_1$ is proper,
hence so is 
$\ol{\alpha} \times_{\Mb_2} \ol{\beta} \to \Mb_1$ 
as a composition of proper morphisms.
This implies the left properness of $\beta\alpha$,
since $\ol{\beta\alpha}$ is
the image of $\ol{\alpha} \times_{\Mb_2} \ol{\beta}$ 
in $\Mb_1\times \Mb_3$.
%
\end{proof}

\subsection{Categories of modulus pairs}

\begin{defn}\label{d2.4a} By Propositions \ref{lem:comp-admcorr} and \ref{p1a}, modulus pairs and left proper admissible correspondences define an additive category that we denote by $\ulMCor$. We write $\MCor$ for the full subcategory of $\ulMCor$ whose objects  are proper modulus pairs (see Definition \ref{d2.2} (1)).
\end{defn}

In the context of modulus pairs, the category $\Sm$ and the graph functor $\Sm\to \Cor$ are replaced by the following:

\begin{defn}\label{d2.5} We write $\ulMP$ for the category with the same objects as $\ulMCor$, and a morphism of $\ulMP(M_1,M_2)$ is given by a (scheme-theoretic) $k$-morphism $f:M_1^\o\to M_2^\o$ whose graph belongs to $\ulMCor(M_1,M_2)$. We write $\MP$ for the full subcategory of $\ulMP$ whose objects are proper modulus pairs. 
\end{defn}

We will need some variants of these categories.

\begin{defn}\label{d1.1} 
\ 
\begin{enumerate}
\item 
We write $\ulMCor^\fin$ for the subcategory of $\ulMCor$ with the same objects and the following condition on morphisms: $\alpha\in \ulMCor(M,N)$ belongs to $\ulMCor^\fin(M,N)$ if and only if, for any component $Z$ of $\alpha$, the projection $\Zb\to \Mb$ is \emph{finite},
where $\ol{Z}$ is the closure of $Z$ in $\ol{M} \times \ol{N}$. 
The same argument as in the proof of Proposition \ref{p1a} shows that $\ulMCor^\fin$ is indeed a subcategory of $\ulMCor$.
We write $\MCor^\fin$ for the full subcategory of $\ulMCor$ whose objects  are proper modulus pairs.
\item
We write $\ulMP^\fin$ for the subcategory of $\ulMSm$  with the same objects and such that a morphism $f:M\to N$ belongs to $\ulMP^\fin$ if and only if 
$f^\o : M^\o  \to N^\o$ extends to a $k$-morphism $\ol{f} : \ol{M}\to \ol{N}$. 
Such extension $\ol{f}$ is unique
because $M^\o$ is dense in the reduced scheme $\ol{M}$ 
and $\ol{N}$ is separated
(\cite[Chapter~II, Exercise~4.2]{hartshorne}).
This yields a forgetful functor $\ulMSm^\fin \to \Sch$, which sends $M$ to $\ol{M}$.

\noindent We write $\MSm^\fin$ for the full subcategory of $\ulMSm$ whose objects  are proper modulus pairs.
\item
We write
\begin{align}
\notag
&c : \MSm \to \MCor,\quad
\\ 
\label{eq:def-c}
&\ul{c} : \ulMSm \to \ulMCor,\quad
\\
\notag
&\ul{c}^\fin : \ulMSm^\fin \to \ulMCor^\fin
\end{align}
for the functors which are the identity on objects
and which carry a morphism $f$ to the graph of $f^\o$.
\end{enumerate}
\end{defn}

 Let $f : M \to N$ be a morphism in $\ulMSm^\fin$. Since $\ol{f}(M^\o)\subseteq N^\o$, none of the images of the generic points of the irreducible components of $\ol{M}$ is contained in $|N^\infty|$, hence the pullback of the Cartier divisor $\ol{f}^\ast N^\infty$ is well-defined. For ease of notation, we simply write it $f^*N^\infty$.

\begin{defn}\label{def:minimality}
A morphism $f : M \to N$ in $\ulMSm^\fin$ is \emph{minimal} if we have $f^\ast N^\infty = M^\infty$.
\end{defn}

\begin{rk}\label{rk-graph-trick}
We remark the following.
\begin{enumerate}
\item \label{gt1}
Assume that $\Mb$ is normal. 
Then Zariski's connectedness theorem implies
that for any $N$
\[\ulMP(M,N)\cap \ulMCor^\fin(M,N)=\ulMP^\fin(M,N).\]
(Indeed, given an elementary correspondence 
belonging to the left hand side,
its closure in $\ol{M} \times \ol{N}$
is birational and finite over an irreducible component of $\ol{M}$,
but such a morphism is an isomorphism 
if $\ol{M}$ is normal by \cite[corollaire~4.4.9]{EGA3}).
If $f^\o : M^\o \to N^\o$ extends to a morphism between ambient spaces $\ol{f}:\ol{M} \to \ol{N}$, then 
the graph of $f^\o$ is admissible if and only if we have $M^\infty \geq \ol{f}^\ast N^\infty$. 
\item \label{gt2} For $M \in \ulMSm^\fin$,
set $M^N:=(\ol{M}^N, M^\infty |_{\ol{M}^N})$
where $p:\ol{M}^N \to \ol{M}$ is the normalization
and $M^\infty |_{\ol{M}^N}$ is the pull-back of $M^\infty$
to $\ol{M}^N$.
Then  $p : M^N \to M$ is an isomorphism in $\ulMCor^\fin$ and $\ulMSm$ (but not in $\ulMSm^\fin$ in general).
\item \label{gt4} Let $M=(\ol{M},M^\infty)$ and $N=(\ol{N},N^\infty)$ be two modulus pairs and let  $\ol{Z}\subset \ol{M}\times \ol{N}$ be an integral closed subscheme which is finite and surjective over an irreducible component of $\ol{M}$, such that $\ol{Z}\not \subset \ol{M}\times N^\infty$ and that $M^\infty|_{\ol{Z}^N}\ge N^\infty|_{\ol{Z}^N}$, where $\ol{Z}^N$ is the normalization of $\ol{Z}$. Then $Z=\ol{Z}\cap (M^\o\times \ol{N})$ belongs to $\Cor(M^\o,N^\o)$ and its closure in $\ol{M}\times \ol{N}$ is $\ol{Z}$: this follows from Remark \ref{r2.1} (4).
\item \label{gt3} For any morphism $f : M \to N$ in $\ulMSm$, there exists a morphism  $M' \to M$  in $\ulMSm^\fin$ which is invertible in $\ulMSm$ such that the induced morphism $M' \to N$ is in $\ulMSm^\fin$.
More generally, we have the following lemma.
\end{enumerate}
\end{rk}

\begin{lemma}[The graph trick]\label{l-gt}
Let $f : M \to N$ be a morphism in $\ulMSm$. Then there exists a minimal morphism $p : M_1 \to M$ in $\ulMSm^\fin$ 
such that it is invertible in $\ulMSm$ and the composite $f \circ p : M_1 \to M \to N$ is a morphism in $\ulMSm^\fin$. 
Moreover, if $f^\o : M^\o \to N^\o$ extends to a morphism $\ol{U} \to \ol{N}$ for an open subset $\ol{U} \subset \ol{M}$, then we can choose $M_1$ such that $\ol{M}_1 \to \ol{M}$ is an isomorphism over $\ol{U}$ (note that we can always take $\ol{U}=M^\o$).
\end{lemma}

\begin{proof}
Let $\Gamma$ be the graph of the morphism $\ol{U} \to \ol{N}$, and let $\ol{\Gamma}$ be its closure in $\ol{M} \times \ol{N}$.
Then we have natural projections $p_1 : \ol{\Gamma} \to \ol{M}$ and $p_2 : \ol{\Gamma} \to \ol{N}$.
Since we have $\Gamma \cong \ol{U}$, Lemma \ref{nexfib} below implies that $p_1$ is an isomorphism over $\ol{U}$ and we have $p_1^{-1}(\ol{U}) = \Gamma$. 
Defining $M_1 := (\ol{\Gamma}, p_1^\ast M^\infty)$, the morphism $p_1$ induces a morphism $p_1 : M_1 \to M$ in $\ulMSm^\fin$ such that $f \circ p_1  : M_1 \to M \to N$ comes from $\ulMSm^\fin$ defined by $p_2$. 
Also note that $\ol{\Gamma} \to \ol{M}$ is proper since $f$ is, 
which implies that $p_1 : M_1 \to M$ is an isomorphism in $\ulMSm$. This finishes the proof. 
\end{proof}

\begin{lemma}[No extra fibre lemma]\label{nexfib} 
Let $f : X \to Y$ be a separated morphism of schemes, and let $U \subset X$ be an open dense subset. 
Assume that the image $f(U)$ of $U$ is open in $Y$, and the induced morphism $U \to f(U)$ is proper (e.g., an isomorphism).
Then, we have $f^{-1}(f(U)) = U$.\end{lemma}

\begin{proof}
Consider the commutative diagram
\[\begin{xymatrix}{
U \ar[rd]_{\mathrm{proper}} \ar[r]^(0.38)j & f^{-1}(f(U)) \ar[r] \ar@{}[rd]|\square \ar[d]^{\mathrm{sep.}} & X \ar[d]_f^{\mathrm{sep.}} \\
&  f(U) \ar[r] & Y}\end{xymatrix}\]
where all the horizontal arrows are open immersions, the square is cartesian and the two vertical morphisms are separated.
The triangle diagram on the left implies that $j$ is proper
(\cite[Chapter~II, Corollary~4.8]{hartshorne}), 
hence it is a closed (and open) immersion.
Since $U$ is dense in $X$, it is dense in $f^{-1}(f(U))$ as well,
hence the conclusion.
%
\end{proof}

\begin{rk}\label{l3.3} Let $M\in \ulMCor^\fin$. Assume that $M^\o=M_1^\o\coprod M_2^\o$; let $\ol{M}_i$ be the closure of $M_i^\o$ in $\ol{M}$ and $M_i^\infty$ be the pull-back of $M^\infty$ to $\ol{M}_i$. Then $M_i=(\ol{M}_i,M_i^\infty)$ are modulus pairs, the inclusions $M_i^\o\inj M^\o$ yield morphisms $M_i\to M$ in $\ulMSm^\fin$, and the induced morphism in $\ulMCor^\fin$ 
\[M_1\oplus M_2\to M\]
is an isomorphism in $\ulMCor^\fin$.
The proof is easy and left to the reader. 

This remark may help in reducing some reasonings to the case where $M^\o$ is irreducible.
\end{rk}

\subsection{The functors $(-)^{(n)}$}

\begin{defn}\label{dn1} Let $n\ge 1$ and $M=(\ol M,M^\infty)\in \ulMCor$. We write
\[M^{(n)} = (\ol M,nM^\infty).\]
This defines an endofunctor of $\ulMCor$. These come with natural transformations
\begin{equation}\label{eqn1}
M^{(n)}\to M^{(m)}\quad \text{if } m\le n.
\end{equation}
\end{defn}

\begin{lemma}\label{ln1} The functor $(-)^{(n)}$ is fully faithful.
\end{lemma}
\begin{proof}
This follows from the definition and the fact that
if $A$ is an integral domain with quotient field $K$,
then $a \in K$ is integral over $A$ if and only if so is $a^n$.
\end{proof}

\subsection{Changes of categories}\label{s1.2}
We now have a basic diagram of additive categories and functors
\begin{equation}\label{eq.taulambda}
\xymatrix{
\MCor \ar[rr]^\tau\ar[rd]^\omega && \ulMCor\ar@<4pt>[ld]^\ulomega\\
&\Cor\ar[ur]^\lambda
}
\end{equation}
with 
\[\tau(M)=M;\quad \omega(M)=M^\o;\quad  \ulomega(M)=M^\o; \quad \lambda(X)=(X,\emptyset).\]

All these functors are faithful, and $\tau$ is fully faithful; they ``restrict'' to analogous functors $\tau_s,\omega_s,\ulomega_s,\lambda_s$ between $\MP$, $\ulMP$ and $\Sm$. Note that $\ulomega\circ (-)^{(n)}=\ulomega$ for any $n$.
 Moreover:

\begin{lemma}\label{l1.2} 
We have $\ulomega\tau=\omega$. 
Moreover, $\lambda$ is left adjoint to $\ulomega$,
and the restriction of $\lambda$ to $\Cor^\proper$ 
$($finite correspondences on smooth proper schemes over $k)$ 
is ``right adjoint'' to $\ulomega$.
$($\emph{i.e.}, $\Cor(\ulomega(M), X)=\ulMCor(M, \lambda(X))$
for $M \in \ulMCor$ and $X \in \Cor^\proper.)$ The same statements are valid for $\tau_s,\omega_s,\ulomega_s,\lambda_s$ when restricted to $\MSm$, $\ulMSm$ and $\Sm$.
\end{lemma}

\begin{proof} 
The first identity is obvious. For the adjointness, let $X\in \Cor$, $M\in \ulMCor$ and $\alpha\in \Cor(X,M^\o)$ be an integral finite correspondence. Then $\alpha$ is closed in $X\times \Mb$, since it is finite over $X$ and $\Mb$ is separated; it is evidently finite (hence proper) over $X$. It also satisfies $q^*M^\infty=0$ where $q$ is the composition $\alpha^N\to \alpha\to M^\o\to \Mb$, because $M^\infty|_{M^\o}=0$.
Therefore $\alpha\in \ulMCor(\lambda(X),M)$.

For the second statement, assume $X$ proper and let $\beta\in \Cor(M^\o,X)$ be an integral finite correspondence. Then $\beta$ is trivially admissible, and its closure in $\Mb\times X$ is proper over $\Mb$, so $\beta\in \ulMCor(M,\lambda(X))$. The last claim is immediate.
\end{proof}

The following theorem is an important refinement of Lemma \ref{l1.2}. 
The proof starts from \S \ref{starts-pf-1.5.2} and is completed in \S \ref{end-pf-1.5.2}.

\begin{thm}\label{t2.1} The functors 
$\omega$, $\tau$, $\omega_s$ and $\tau_s$ 
have pro-left adjoints 
$\omega^!$, $\tau^!$, $\omega_s^!$ and $\tau_s^!$ (see \S \ref{s1.1}).
\end{thm}

General definitions and results 
on pro-objects and pro-adjoints
are gathered in \S \S \ref{sec:pro-obj} and \ref{s1.1}.
We shall freely use results from there.

\subsection{The closure of a finite correspondence}
We shall need the following result
for the proof of Theorem \ref{t2.1}.

\begin{lemma}\label{lrg} Let $X$ be a Noetherian scheme, $(\pi_i:Z_i\to X)_{1\le i\le n}$ a finite set of proper surjective morphisms with $Z_i$ integral, and let $U\subseteq X$ be a normal open subset. Suppose that $\pi_i:\pi_i^{-1}(U)\to U$ is finite for every $i$. Then there exists a proper birational morphism $X'\to X$ which is an isomorphism over $U$, such that the closure of $\pi_i^{-1}(U)$ in $Z_i\times_X X'$ is finite over $X'$ for every $i$.
\end{lemma}

\begin{proof} 
By induction, we reduce to $n=1$; then this follows from \cite[Corollary~5.7.10]{rg} applied with $(S,X,U)\equiv (X,Z_1,U)$ and $n=0$ (note that a morphism is finite if and only if it is quasi-finite and proper, and that an admissible blow-up of an algebraic space is a scheme if the algebraic space happens to be a scheme).
\end{proof}

\begin{thm}\label{trg} Let $X,Y\in \Sch$. Let $U$ be a normal dense open subscheme of $X$, and let $\alpha$ be a finite correspondence from $U$ to $Y$.  Suppose that the closure $\Zb$ of $Z$ in $X\times Y$ is proper over $X$ for any component $Z$ of $\alpha$. Then there is a proper birational morphism $X'\to X$ which is an isomorphism over $U$, such that $\alpha$ extends to a finite correspondence from $X'$ to $Y$.
\end{thm}

\begin{proof} Apply Lemma \ref{lrg}, noting that $Z= \Zb\times_X U$ by \cite[Lemma 2.6.3]{rec}.
\end{proof}

The following lemma also relies on \cite{rg}: it will be used several times in the sequel.

\begin{lemma}\label{mainlem;blowup}
Let $f: U\to X$ be an \'etale morphism of quasi-compact and quasi-separated integral schemes. Let $g: V\to U$ be a proper birational morphism, $T \subset U$ a closed subset such that 
$g$ is an isomorphism over $U-T$
and $S$ the closure of $f(T)$ in $X$.
Then there exists a closed subscheme $Z\subset X$ supported in $S$ such that
$U\times_X \Bl_Z(X)  \to U$ factors through $V$. 
\end{lemma}

\begin{proof}
The following argument is taken from the proof of \cite[Proposition~5.9]{sv}.
Noting $V$ is \'etale over $X-S$, we apply the platification theorem \cite[Corollary~5.7.11]{rg} to 
$V \to X$ and conclude that there exists a closed subscheme $Z$ supported in $S$ such that the proper transform $V'$ of $V$ under $X'=\Bl_Z(X)\to X$ is flat over $X'$. By the construction the induced morphism $\phi: V'\to U\times_X X'$ is proper birational. On the other hand $\phi$ is flat since it becomes flat when composed with the \'etale morphism $ U\times_X X'\to X'$ (\cite[Chapter~II, Proposition~8.11 and Chapter~III, Exercise~10.3]{hartshorne}). 
Hence it is an isomorphism. This proves the lemma since $V' \to U$ factors $V\to U$.
\end{proof}

\subsection{Proof of Theorem \ref{t2.1}: case of $\omega$ and $\omega_s$}\label{starts-pf-1.5.2}
We need a definition:

\begin{defn}\label{d.sigma} 
Let $\Sigma$
be the class of all morphisms 
$M_1\to M_2$ in $\MCor$ 
given by the graph of an isomorphism $M_1^\o \iso M_2^\o$ in $\Sm$.
\end{defn}

 In view of Proposition \ref{p2.6}, the existence of the pro-left adjoint of $\omega$ is a consequence of the following more precise result:

\begin{proposition}\label{prop:localization}
\leavevmode
\begin{itemize}
\item[a)] The class $\Sigma$ enjoys a calculus of right fractions.
\item[b)] The functor $\omega$ induces equivalences of categories
\[
\Sigma^{-1}\MCor\iso \Cor.
\]
The same statement holds for $\omega_s:\MSm\to \Sm$.
\end{itemize}
\end{proposition}

\begin{proof} a) We check the axioms of Definition \ref{d.cf}:
\begin{enumerate}
\item Identities, stability under composition: obvious.
\item Given a diagram in $\MCor$
\[\begin{CD}
&& M'_2\\
&&@VVV\\
M_1@>\alpha>> M_2
\end{CD}\]
with $M_2^\o\cong{M_2'}^\o$, 
Lemma \ref{lem:mod-exists} provides a 
$M_1''\in \MCor$ such that ${M_1''}^\o=M_1^\o$ and $\alpha\in \MCor(M''_1,M'_2)$. We may choose $M''_1$ such that $\ol{M''_1} = \ol{M_1}$. Then 
$M'_1=(\ol{M_1}, {M_1'}^\infty)$ with
any ${M_1'}^\infty$ such that 
${M_1'}^\infty \geq M_1^\infty$,
${M_1'}^\infty \geq {M''_1}^\infty$ 
allows us to complete the square in $\MCor$. 

\item Given a diagram
\[M_1\begin{smallmatrix}f\\ \rightrightarrows\\ g
\end{smallmatrix} M_2\by{s} M'_2
\]
with $M_1,M_2,M'_2$ as in (2) and such that $sf=sg$, the underlying correspondences to $f$ and $g$ are equal since the one underlying $s$ is $1_{M_2^\o}$. Hence $f=g$.
\end{enumerate}

The above proof of (2) also shows that
we have 
\[
\colim_{M' \in \Sigma \downarrow M} \MCor(M', N)
=\Cor(M, N).
\]
for any $M, N \in \MCor$. 

Point b) now follows from a) and Corollary \ref{c.cf}, noting that $\omega$ is essentially surjective. Indeed, any smooth $k$-scheme $X$ admits a compactification $\bar X$ by Nagata's theorem; blowing up $\bar X - X$, we then make it a Cartier divisor. The case of $\omega_s$ is exactly parallel.
\end{proof}

Let 
$\omega^!:\Cor\to \pro{}\MCor$
be the pro-left adjoint of $\omega$. 
By Proposition \ref{p2.6}, we have for $X\in \Cor$:
\[\omega^! X = \underset{M\in \Sigma\downarrow X}{``\lim"} M. 
\]
and the same formula for the pro-left adjoint $\omega_s^!$ of $\omega_s$. Let us spell out the indexing set $\MP(X)$ of these pro-objects, and refine them:

\begin{definition}\label{def:mpx}\
\begin{enumerate}
\item
For $X \in \Sm$, 
we define a subcategory $\MP(X)$ 
of $\MP$ as follows. 
The objects are those $M \in \MP$ such that $M^\o=X$.
Given $M_1, M_2 \in \MP(X)$,
we define 
$\MP(X)(M_1, M_2)$ 
to be $\{ 1_X \}$
if $1_X$ belongs to $\MSm$ 
and $\emptyset$ otherwise. 
\item
Let $X \in \Sm$ and 
fix a compactification $\Xb$ such that $\Xb-X$ is the support of a Cartier divisor (for short, a \emph{Cartier compactification}).
Define $\MP(\ol{X}!X)$ 
to be the full subcategory
of $\MP(X)$ 
consisting of objects 
$M \in \MP(X)$ 
such that $\ol{M}=\ol{X}$.
\end{enumerate}
\end{definition}

\begin{lemma}\label{ln2} a) For any $X\in \Sm$ and any Cartier compactification $\Xb$, $\MP(X)$ is a cofiltered ordered set, and $\MP(\ol{X}!X)$ is cofinal in $\MP(X)$.\\
b) Let $X\in \Cor$, and let $M\in \MP(X)$. Then $\{ M^{(n)} \}_{n\ge 1}$ defines a cofinal subcategory of $\MP(X)$.
\end{lemma}

\begin{proof} a) ``Ordered'' is obvious and ``cofiltered'' follows from Propositions \ref{prop:localization} and \ref{p.cf} a); the cofinality follows again from Lemma \ref{lem:mod-exists}. 

b) Let $M=(\Xb,X^\infty)$. By a) it suffices to show that $(M^{(n)})_{n\ge 1}$ defines a cofinal subcategory of $\MP(\ol{X}!X)$. If $(\Xb,Y)\in \MP(\ol{X}!X)$, $Y$ and $X^\infty$ both have support $\Xb-X$, so there exists $n>0$ such that $nX^\infty \ge Y$.
\end{proof}

\subsection{Proof of Theorem \ref{t2.1}: case of $\tau$}\label{end-pf-1.5.2}
We need a definition:

\begin{defn}\label{d2.6} Take $M=(\Mb,M^\infty)\in \ulMSm$.
Let $\CompM$ be the category whose objects are pairs
$(N,j)$ consisting of a modulus pair $N=(\Nb,N^\infty)\allowbreak\in\MSm$ 
equipped with a dense open immersion $j:\Mb\hookrightarrow \Nb$ such that
$N^\infty=M_N^\infty+C$
for some effective Cartier divisors $M_N^\infty, C$ on $\ol{N}$
satisfying
$\ol{N} \setminus |C| = j(\ol{M})$
and 
$j$ induces a minimal morphism $M \to N$ in the sense of Def. \ref{def:minimality}.
Note that for $N\in \CompM$ we have $j(M^\o)=N^\o$ and $N$ is equipped with $j_N\in \ulMSm^\fin(M, N) \subset \ulMSm(M,N)$ 
which is the graph of $j|_{M^\o} : M^\o \cong N^\o$.
For $N_1,N_2\in \CompM$ we define 
\[\CompM(N_1,N_2)=\{\gamma\in \MSm(N_1,N_2)\;|\; \gamma\circ j_{N_1}  =j_{N_2}\}.\]
Note that any $\gamma$ as above induces an isomorphism $N_1^\o \iso N_2^\o$ in $\Sm$.
\end{defn}

\begin{lemma}\label{c1.1} The category
$\CompM$ is a cofiltered ordered set.
\end{lemma}

\begin{proof} That it is ordered is obvious
as $\CompM(N_1,N_2)$ has at most 1 element for any $(N_1,N_2)$. 
For ``cofiltered'',  we first show that $\CompM$ is nonempty. For this, choose a compactification $j_0:\ol{M}\inj \ol{N}_0$, with $\ol{N}_0\in \Sch$ proper. Let $\ol{N}_1=\Bl_{(\ol{N}_0-\ol{M})_\red}(\ol{N_0})$; then $j_0$ lifts to $j_1:\ol{M}\inj \ol{N}_1$
by the universality of the blowup \cite[Chapter~II, Proposition~7.14]{hartshorne},
and $\ol{N}_1-\ol{M}$ is the support of an effective Cartier divisor $C_1$. 
Consider now the scheme-theoretic closure $N_1^\infty$ of $M^\infty$ in $\ol{N}_1$, 
and define $\ol{N}=\Bl_{N_1^\infty}(\ol{N}_1)$, 
$M_N^\infty =$ pull-back of $N_1^\infty$, $C=$ pull-back of $C_1$,
$N^\infty=M_N^\infty+C$ and $N=(\ol{N},N^\infty)$: 
then 
$j_1$ lifts to $j:\ol{M}\inj \ol{N}$ 
(by the same reason as $j_0$),
which defines an object of $\CompM$.

Let $N_1$ and $N_2$ be two objects in $\Comp(M)$.
Let $\Gamma$ be the graph of the rational map $\ol{N}_1 \dashrightarrow
\ol{N}_2$ given by $1_{M^\o}$.
Then we have morphisms of schemes $p : \Gamma \to \ol{N}_1$ and $q :  \Gamma
\to \ol{N}_2$, and there exists a natural open immersion $\ol{M} \to
\Gamma$.
Note that $(\Gamma,p^\ast N_1^\infty)$ and $(\Gamma,q^\ast N_2^\infty)$ are
objects of $\Comp(M)$.
Since $(\Gamma,p^\ast N_1^\infty)$ dominates $N_1$ and $(\Gamma,q^\ast
N_2^\infty)$ dominates $N_2$, we are reduced to the case that $N_1$ and
$N_2$ have the same ambient space $\ol{N}$.
Let $C$ be the effective Cartier divisor on $\ol{N}$ such that $|C|=\ol{N} -
\ol{M}$, which exists since $N_1 \in \Comp(M)$.
Then for a sufficiently large $n$ we have $N_1^\infty + nC \geq N_2^\infty$
since $N_1^\infty \cap \ol{M} = N_2^\infty \cap \ol{M} = M^\infty$. 
Therefore  $N_3 = (\ol{N},N_1^\infty + nC)$ dominates both $N_1$ and $N_2$.
This finishes the proof.
\end{proof}

For $M\in \ulMCor$ and $L\in \MCor$ we have a natural map
\begin{align*}
&\Phi: \underset{N\in \CompM}{\colim} \MCor(N,L) \to \ulMCor(M,\tau L),
\end{align*}
which maps a representative 
$\alpha_N\in \MCor(N,L)$ 
to $\alpha_N\circ j_N$. 
We also have a natural map
for $M, L'\in \ulMCor$ 
\begin{align*}
&\Psi: \ulMCor(L',M)  \to \underset{N\in \CompM}{\lim} \ulMCor(L',\tau N),
\end{align*}
which maps $\beta$ to $\{ j_N \circ \beta \}_N$. 

The following is an analogue to Lemma \ref{lem:mod-exists}:

\begin{lemma}\label{l1.1}
The maps $\Phi$ and $\Psi$
are isomorphisms. In other words, the formula
\[\tau^! M = \underset{N\in \Comp(M)}{``\lim" N},
\]
defines a pro-left adjoint to $\tau$, 
which is fully faithful.
\end{lemma}

\begin{proof}
We start with $\Phi$. 
Injectivity is obvious since both sides are subgroups of $\Cor(M^\o,L^\o)$.
We prove surjectivity.
Choose a dense open immersion $j_1: \Mb \hookrightarrow \Nb_1$ with $\Nb_1$ proper 
such that $\Nb_1-\Mb$ is the support of an effective Cartier divisor $C_1$.
Let $M_1^\infty$ be the scheme-theoretic closure of $M^\infty$ in $\Nb_1$.
 (This may not be Cartier.)
Let $\pi:\Nb_2\to \Nb_1$ be the blowup with center in $M_1^\infty$ and 
put $M_2^\infty=M_1^\infty \times_{\ol{N}_1} \ol{N}_2$
and $C_2=C_1 \times_{\ol{N}_1} \ol{N}_2$.
Note that $M_2^\infty$ and $C_2$ are effective Cartier divisors on $\ol{N}_2$.
By the universal property of the blowup \cite[Chapter~II, Proposition~7.14]{hartshorne}, $j_1$ extends to an open immersion $j_2: \Mb\to \Nb_2$ so that $j_1=\pi j_2$. 
Then $\Nb_2-M^\o$ is the support of the Cartier divisor 
$N_2^\infty:=M_2^\infty+ C_2$ so that
\[ ((\Nb_2, N_2^\infty), j_2) \in \Comp(M).\]
Now the claim for $\Phi$ follows from the following:

\begin{claim}\label{cl2.1}
For any $\alpha\in \ulMCor(M,L)$, there exists an integer $n>0$ such that
$\alpha\in \MCor((\Nb_2,M_2^\infty +nC_2),L)$.
\end{claim}

Indeed we may assume $\alpha$ is an integral closed subscheme of $M^\o\times L^\o$.
We have a commutative diagram
\[\xymatrix{
\alphab^N \ar[r]^{j_1} \ar[d]_{\phi_\alpha} & \alphab_1^N \ar[d]_{\phi_{\alpha_1}} 
& \ar[l]_{\pi}\alphab_2^N \ar[d]_{\phi_{\alpha_2}}\\
\Mb\times\Lb  \ar[r]^{j_1} & \Nb_1\times\Lb  &  \ar[l]_{\pi}\Nb_2\times\Lb \\
}\]
where $\alphab^N$ (resp. $\alphab^N_1$, resp. $\alphab^N_2$) is the normalization of the closure of 
$\alpha\subset M^\o\times L^0$ in $\Mb\times\Lb$ (resp. $\Nb_1\times\Lb$, resp. $\Nb_2\times\Lb$), and
$j_1$ and $\pi$ are induced by $j_1:\Mb\to \Nb_1$ and $\pi:\Nb_2\to \Nb_1$ respectively.
Now the admissibility of $\alpha\in \ulMCor(M,L)$ implies
\[\phi^*_\alpha(\Mb\times L^\infty) \leq \phi^*_\alpha(M^\infty \times \Lb).\]
Since $\alphab_1^N-j_1(\alphab^N)$ is supported on $\phi^{-1}_{\alpha_1}(C_1\times \Lb)$, this  yields an inclusion of closed subschemes
\[\phi^*_{\alpha_1}(\Nb_1\times L^\infty) \subseteq  \phi^*_{\alpha_1}((M_1^\infty + n C_1) \times \Lb)\]
for a sufficiently large $n>0$. Applying $\pi^*$ to this  inclusion, we get  an inequality of Cartier divisors
\[\phi^*_{\alpha_2}(\Nb_2\times L^\infty) \leq \phi^*_{\alpha_2}((M_2^\infty + n C_2) \times \Lb)\]
which proves the claim.

Next we prove that $\Psi$ is an isomorphism.
Injectivity is obvious since both sides are subgroups of $\Cor(L^\o,M^\o)$.
We prove surjectivity. Take
$$\gamma\in \underset{N\in \CompM}{\lim} \MCor(L,N).$$
Then $\gamma\in \Cor(L^\o,M^\o)$ is such that any component $\delta\subset L^\o\times M^\o$ of $\gamma$ satisfies
the following condition: 
take any $(N, j) \in \Comp(M)$
and write
$N^\infty = M_N^\infty+C$ as in Definition \ref{d2.6}.
Let $\deltab^N$ be the normalization of the closure of $\delta$ in $\Lb\times \Nb$
with the natural map $\phi_{\delta}:\deltab^N \to \Lb\times \Nb$.
Then we have
\[ \phi_{\delta}^*(\Lb\times (M^\infty_N + n C))\leq \phi_{\delta}^*(L^\infty\times \Nb)\]
for any integer $n>0$. Clearly this implies that $|\deltab|$ does not intersect with $\ol{L} \times |C|$ so that 
$\deltab\subset \Lb\times\Mb$. 
Since $\deltab$ is proper over $\ol{L}$ by assumption,
this implies $\delta\in \ulMCor(L,M)$ which proves the surjectivity of $\Psi$ as desired. 
\end{proof}

We come back to the proof of Theorem \ref{t2.1}.
It remains to consider $\tau_s$.
The natural maps
\begin{align*}
&\phi: \underset{N\in \Comp(M)}{\colim} \MSm(N,L) \to \ulMSm(M,\tau L),
\\
&\psi: \ulMSm(L',M)  \to \underset{N\in \Comp(M)}{\lim} \ulMSm(L',\tau N)
\end{align*}
are also bijective for any $M, L'\in \ulMCor$ and $L\in \MCor$.
The proof is identical to Lemma \ref{l1.1}.
In particular, 
the inclusion functor
$\tau_s : \MSm \to \ulMSm$
admits a pro-left adjoint given by 
$$\tau_s^{!} M = \underset{N\in \Comp(M)}{``\lim" N},$$ which commutes with the inclusions $\MSm\inj \MCor$ and $\ulMSm\inj \ulMCor$. 
This completes the proof of Theorem \ref{t2.1}.
\qed

\subsection{More on $\ulMP^\fin$ and $\ulMCor^\fin$}

\begin{defn}\label{deff} 
A morphism $f:M\to N$ in $\ulMP^\fin$ is in $\ul{\Sigma}^\fin$ if it is minimal (Definition~\ref{def:minimality}), $\ol{f}:\Mb\to \Nb$ is a proper morphism and $f^\o$ is an isomorphism in $\Sm$. 
We write $\Sigma^\fin$ for the class of morphisms 
in $\ul{\Sigma}^\fin$ that belong to $\MSm$.
\end{defn}

In particular, we have
${\Sigma}^\fin\subset {\Sigma}$ (see Definition \ref{d.sigma})
and
$\ul{\Sigma}^\fin \downarrow M = {\Sigma}^\fin \downarrow M$
for $M \in \MSm$.
Let us consider the inclusion functors
\begin{align}\label{eq:def-b}
&\ul{b}_s : \ulMSm^\fin \to \ulMSm, \quad
\ul{b} : \ulMCor^\fin \to \ulMCor.
\end{align}

The following commutative diagram of categories
will become fundamental
(cf. \eqref{eq:six-cat-diag}):
\begin{equation}\label{eq:six-cat-diag0}
\xymatrix{
\MCor \ar[r]^{\tau} &
\ulMCor &
\ulMCor^\fin \ar[l]_{\ul{b}}
\\
\MSm \ar[r]^{\tau_s} \ar[u]^{c}&
\ulMSm \ar[u]^{\ul{c}} &
\ulMSm^\fin. \ar[l]_{\ul{b}_s} \ar[u]^{\ul{c}^\fin}
}
\end{equation}

\begin{prop}\label{peff1} 
\leavevmode
\begin{itemize}
\item[a)] The class $\ul{\Sigma}^\fin$ enjoys a calculus of right fractions within $\ulMP^\fin$ and $\ulMCor^\fin$.
\item[b)] The functors $\ul{b}_s$ and $\ul{b}$ 
are 
localisations having 
left pro-adjoints $b_s^!$ and $b^!$.
They induce equivalences of categories
\[ (\ul{\Sigma}^\fin)^{-1} \ulMP^\fin \cong \ulMP
\quad \text{and} \quad
(\ul{\Sigma}^\fin)^{-1} \ulMCor^\fin \cong \ulMCor.
\]
\item[c)] A morphism in $\ulMCor^\fin$ (resp. $\ulMSm^\fin$) 
is invertible in $\ulMCor$ (resp. $\ulMP$) 
if and only if it belongs to $\ul{\Sigma}^\fin$.
A morphism $f$ in $\ulMCor$ (resp. $\ulMSm$) is an isomorphism 
if and only if it can be written as
$s=s_1 s_2^{-1}$ for some $s_1, s_2 \in \ul{\Sigma}^\fin$.
\end{itemize}
All statements hold for $\Sigma^\fin$ (without an underline) as well.
\end{prop}

\begin{proof} 
a) Same as the proof of Proposition \ref{prop:localization} a), except for (2): consider a diagram in $\ulMCor^\fin$
\[\begin{CD}
&& M'_2\\
&&@VfVV\\
M_1@>\alpha >> M_2
\end{CD}\]
with $f\in \ul{\Sigma}^\fin$ (in particular $f^\o$ is an isomorphism). 
By the properness of $f$, the finite correspondence $\alpha^\o:M_1^\o\to {M'_2}^\o$ satisfies the hypothesis of Theorem \ref{trg}. Applying this theorem,  we find a proper birational morphism $f':\Mb'_1\to \Mb_1$ which is an isomorphism over $M_1^\o$ and such that $\alpha^\o$ defines a finite correspondence $\alpha':\Mb'_1\to \Mb'_2$. If we define ${M'_1}^\infty={f'}^*M_1^\infty$, then $f'\in \ul{\Sigma}^\fin$ and $\alpha'\in \ulMCor^\fin(M'_1,M'_2)$. 

If $\alpha\in \ulMP^\fin(M_1,M_2)$, then $\alpha'$ is not in $\ulMP^\fin(M'_1,M'_2)$ in general (unless $\ol{M'_1}$ is normal, see Remark \ref{rk-graph-trick} \eqref{gt1}). However, write $\ol{M''_1}$ for the closure of the graph of the rational map $\alpha':\ol{M'_1}\tto \ol{M'_2}$, and $\pi$ for the projection $\ol{M''_1}\to \ol{M'_1}$: by hypothesis, $\pi$ is finite birational. Define a modulus pair $M''_1=(\ol{M''_1}, {M''_1}^\infty)$ by putting ${M''_1}^\infty:=\pi^*{M'_1}^\infty$. Then $\pi$ defines a minimal morphism $M''_1\to M'_1$ in $\ulMP^\fin$, hence the morphism $\alpha'':M''_1\to M'_2$ determined by $\alpha'$ is in $\ulMP^\fin$.

For b), all assertions are obvious
except for the equivalences,
for which 
it suffices as in Corollary \ref{c.cf} to show that for any $M,N\in \ulMCor$, the obvious maps
\[\colim_{M'\in \ul{\Sigma}^\fin \downarrow M} \ulMCor^\fin(M',N)\to \ulMCor(M,N) 
\]
and the corresponding map for $\ul{b}_s$
are isomorphisms.
These maps are clearly injective, and its surjectivity follows again from Theorem \ref{trg}.
It then follows from Proposition \ref{p2.6}
they have pro-left adjoints.

The first statement of c) is clear,
and the second follows from b).

The same proof works for $\Sigma^\fin$.
\end{proof}

\begin{cor}\label{cor:sigma-fin-cofil}
For any $M \in \ulMCor$,
the category $\ul{\Sigma}^\fin \downarrow M$ is
cofiltered.
\end{cor}
\begin{proof}
This follows from Propositions \ref{peff1} and \ref{p.cf}.
\end{proof}

\begin{cor}\label{cor:fiber-cat} 
Let $\sC$ be a category and let $F:\ulMCor^\fin\to \sC$, $G:\ulMP\to \sC$ be two functors whose restrictions to the common subcategory $\ulMP^\fin$ are equal. Then $(F,G)$ extends (uniquely) to a functor $H:\ulMCor\to \sC$.
\end{cor}

\begin{proof} The hypothesis implies that $F$ inverts the morphisms in $\ul{\Sigma}^\fin$; the conclusion now follows from Proposition \ref{peff1} b).
\end{proof}

\begin{cor}\label{c2.1} 
Any modulus pair in $\ulMP$ is isomorphic to a modulus pair $M$ in which $\ol{M}$ is normal. Under resolution of singularities, we may even choose $\ol{M}$ smooth and the support of $M^\infty$ to be a divisor with normal crossings. 
\end{cor}

\begin{proof} 
Let $M_0\in \ulMP$. Consider a proper morphism $\pi:\ol{M}\to \ol{M_0}$ which is an isomorphism over $M_0^\o$. Define $M^\infty:=\pi^*M_0^\infty$. Then the induced morphism $\pi:M\to M_0$ of $\ulMP^\fin$ is in $\ul{\Sigma}^\fin$, hence invertible in $\ulMP$. The corollary readily follows.
\end{proof}

We also have the following important lemma:

\begin{lemma}\label{l1.3}
Let $M, L, N \in \ulMP$.
Let $f: L \to N$ be a minimal morphism in $\ulMP^\fin$
such that $\ol{f} : \ol{L} \to \ol{N}$ is  faithfully flat. 
Then the diagram 
\[
\xymatrix{
\ulMCor(N, M) \ar[r]^{f^*} \ar@{^{(}->}[d]
&\ulMCor(L, M) \ar@{^{(}->}[d]
\\
\Cor(N^\o, M^\o) \ar[r]_{(f^\o)^*}
&\Cor(L^\o, M^\o) 
}
\]
%
%
is cartesian.
The same holds when 
$\ulMCor$ is replaced by $\ulMCor^\fin$.
\end{lemma}

\begin{proof}
As the second statement is proven in a 
completely parallel way,
we only prove the first one.
Take $\alpha \in \Cor(N^\o, M^\o)$
such that $(f^\o)^*(\alpha) \in \ulMCor(L, M)$.
We need to show $\alpha \in \ulMCor(N, M)$.

We first reduce to the case where $\alpha$ is integral.
To do this, it suffices to show that
for two distinct integral finite correspondences
$V, V' \in \Cor(N^\o, M^\o)$, $(f^\o)^*(V)$ and $(f^\o)^*(V')$ have no common component. For this, we may assume $M^\o$ and $N^\o$ integral.
By the injectivity of 
$\Cor(N^\o, M^\o) \to \Cor(k(N^\o), M^\o)$,
this can be reduced to the case  where $N^\o$ and $L^\o$ are fields, 
and then the claim is obvious.

Now assume $\alpha$ is integral
and put $\beta := (f^\o)^*(\alpha)$.
We have a commutative diagram
\[ 
\xymatrix{
\ol{\beta}^N
\ar@/^3ex/[rr]^{\phi_{\beta}}
\ar[r]^{} \ar[d]_{f^N}
& 
\ol{\beta}
\ar@/^3ex/[rr]^{a'}
\ar[r] \ar[d]
&
\ol{L} \times \ol{M}
\ar[r] \ar[d]
&
\ol{L} \ar[d]_{\ol{f}}
\\
\ol{\alpha}^N
\ar@/_3ex/[rr]_{\phi_{\alpha}}
\ar[r]^{}
&
\ol{{\alpha}}
\ar@/_3ex/[rr]_{a}
\ar[r]
&
\ol{N} \times \ol{M}
\ar[r]
&
\ol{N}.
} 
\]
Here 
$\ol{\alpha}$ (resp. $\ol{\beta}$)
is the closure of $\alpha$ (resp. $\beta$)
in $\ol{N} \times \ol{M}$ (resp. $\ol{L} \times \ol{M}$)
and $\ol{\alpha}^N$ (resp. $\ol{\beta}^N$) 
is the normalization of 
$\ol{\alpha}$ (resp. $\ol{\beta}$).
By hypothesis
$a'$ is proper and $\ol{f}$ is faithfully flat.
This implies that $a$ is proper \cite[expos\'e~VIII, corollaire~4.8]{SGA1}.
We also have
\begin{align*}
(f^N)^*&(\phi_{\alpha}^*(N^\infty \times \ol{M}))
=
\phi_{\beta}^*(\ol{f}^*(N^\infty) \times \ol{M}))
\\
&=\phi_{\beta}^*(L^\infty \times \ol{M}) \geq
 \phi_{\beta}^*(\ol{L} \times M^\infty)
=(f^N)^* (\phi_{\alpha}^*(\ol{N} \times M^\infty))
\end{align*}
(the second equality by the minimality of $f$).  Note that $f^N$ is surjective since $\ol{f}$ is. Hence Lemma \ref{lKL} shows that
$\phi_{\alpha}^*(N^\infty \times \ol{M}) \geq
\phi_{\alpha}^*(\ol{N} \times M^\infty)$,
and we are done.
\end{proof}

\subsection{Fiber products and squarable morphisms}\label{sect:squarable} We need the following elementary lemma.

\begin{lemma}\label{lem:have-sup}
Let $X$ be a scheme.
For two effective Cartier divisors $D$ and $E$ on $X$,
the following conditions are equivalent:
\begin{enumerate}
\item $D \times_X E$ is an effective Cartier divisor on $X$.
\item There exist effective Cartier divisors
$D', E'$ and $F$ on $X$ such that
$D=D'+F, ~E=E'+F$ and $|D'| \cap |E'|=\emptyset$.
\end{enumerate}
Moreover, the divisors $D', E'$ and $F$ satisfying the conditions in (2) 
are uniquely determined by $D$ and $E$.
\end{lemma}
\begin{proof}
We may suppose $X=\Spec A$ is affine
and $D, E$ are defined by non-zero-divisors $d, e \in A$, respectively.

Suppose (1).
This means that $(d, e)=(f)$ for some non-zero-divisor $f \in A$,
because  $D \times_X E = \Spec A/(d, e)$.
Thus there are $d', e' \in A$ such that $d=d'f$ and $e=e'f$.
Since $(f)=(d', e')(f)$ and $f$ is a non-zero-divisor,
we have $(d', e')=A$.
Now (2) holds by taking $D', E', F$ to be
the Cartier divisors defined by $(d'), (e')$ and $(f)$.

(2) immediately implies $F=D \times_X E$, whence (1).
This formula also implies the uniqueness of $F$,
hence $D'=D-F$ and $E'=E-F$ are unique as well.
\end{proof}

\begin{definition}
Let $D$ and $E$ be effective Cartier divisors on a scheme $X$.
If the conditions of Lemma \ref{lem:have-sup} hold,
we say that $D$ and $E$ \emph{have a universal supremum},
and write 
\[ \sup(D, E):=D'+E'+F (= D+E-F). \]
\end{definition}

\begin{remark}\label{rem:sup-div}
Let $D$ and $E$ be effective Cartier divisors on $X$
having a universal supremum.
The following are obvious from the definition.
\begin{enumerate}
\item We have $|\sup(D, E)|=|D| \cup |E|$.
\item 
If $f : Y \to X$ is a morphism
such that $f(T) \not\subset |D|\cup |E|$ 
for any irreducible component $T$ of $Y$,
then $f^*D$ and $f^*E$ have a universal supremum 
which is equal to  $f^* \sup(D, E)$ (hence the name ``universal'').
\item 
If moreover $Y$ is normal,
then $f^* \sup(D, E)$ agrees with
the supremum of $f^*D$ and $f^*E$
computed as a Weil divisor on $Y$. 
\end{enumerate}
\end{remark}

Let $u_i : U_i \to M$ be morphisms in $\ulMSm^\fin$ for $i=1, 2$
with projections
$p_i : \ol{W}_0 := \ol{U}_1 \times_{\ol{M}} \ol{U}_2 \to \ol{U}_i$.
Denote by $\ol{W}_1$ the union of irreducible components $T$
of $\ol{W}_0$ such that $p_i(T) \not\subset |U_i^\infty|$
for each $i=1, 2$.
Observe that $\ol{W}_1$ is the closure of $U:=U_1^\o \times_{M^\o} U_2^\o$ in $\ol{W}_0$.
Indeed, let $Z$ be the closure of $U$ in $\ol{W}_0$. Then any irreducible component $T$ of $Z$ meets $U$, which implies that $T\subset \ol{W}_1$. Conversely, any irreducible component $T$ of $\ol{W}_1$ meets $U$, hence $T \cap U$ is dense in $T$ and thus $T \subset  Z$.

We write $q_i : \ol{W}_1 \to \ol{U}_i$ for the composition of
the inclusion $\ol{W}_1 \to \ol{W}_0$ and $p_i$.
By definition, we have 
effective Cartier divisors $q_i^*(U^\infty_i)$ on $\ol{W}_1$
and $q_1 \times q_2$ restricts to an isomorphism
\begin{equation}\label{eq:inside}
\ol{W}_1 \setminus |q_1^*(U_1^\infty)+q_2^*(U_2^\infty)|
\simeq U_1^\o \times_{M^\o} U_2^\o.
\end{equation}

\begin{prop}\label{prop:fiber-prod}
Suppose that
$U_1^\o \times_{M^\o} U_2^\o$ is smooth over $k$.
\begin{enumerate}
\item 
If $q_1^*U_1^\infty$ and $q_2^*U_2^\infty$ have a universal supremum,
then 
\[ W_1:=(\ol{W}_1, \sup(q_1^*U_1^\infty, q_2^*U_2^\infty)) 
\in \ulMSm^\fin
\]
represents the fiber product of $U_1$ and $U_2$ over $M$
in $\ulMSm^\fin$ as well as in $\ulMSm$. 
If further $U_1, U_2, M \in \MSm^\fin$,
then it holds in $\MSm^\fin$ as well as in $\MSm$.
\item If $u_1$ is minimal and $\ol{U}_2$ is normal, then $q_1^\ast U_1^\infty$ and $q_2^\ast U_2^\infty$ have a  universal supremum,  namely $q_2^\ast U_2^\infty$, and the morphism $W_1 \to U_2$ is a minimal morphism in $\ulMSm^\fin$. If moreover $\ol{u}_1$ is flat\footnote{By the local criterion of flatness \cite[Chapter~III, Lemma~10.3.A]{hartshorne}, this is equivalent to the flatness of $U_1^\infty\to M^\infty$.}, we have $\ol{W}_1=\ol{W}_0$.
\item 
In general, 
there is a proper birational morphism
$\pi : \ol{W}_2 \to \ol{W}_1$
which restricts to an isomorphism over
$\ol{W}_1 \setminus |q_1^*(U_1^\infty)+q_2^*(U_2^\infty)|$,
and such that
$r_1^*U_1^\infty$ and $r_2^*U_2^\infty$ have a universal supremum,
where $r_i := q_i \pi$ for $i=1, 2$.
For such $\ol{W}_2$,
\[ W_2:=(\ol{W}_2, \sup(r_1^*U_1^\infty, r_2^*U_2^\infty)) 
\in \ulMSm
\]
represents the fiber product of $U_1$ and $U_2$ over $M$
in $\ulMSm$.
If further $U_1, U_2, M \in \MSm$,
then it holds in $\MSm$.
\end{enumerate}
\end{prop}
\begin{proof}
(1) 
Let $f_i : N \to U_i$ be morphisms in $\ulMSm^\fin$ for $i=1, 2$
such that $u_1f_1=u_2f_2$. 
Then the morphisms $\ol{f}_i : \ol{N} \to \ol{U}_i$ for $i=1, 2$
induce a unique morphism $\ol{h} :\ol{N} \to \ol{W}_0$ with $\ol{f}_i = p_i\ol{h}$ for $i=1, 2$.
Since $f_i$ are morphisms in $\ulMSm^\fin$,
for any irreducible component $T$ of $\ol{N}$
we have $\ol{f}_i(T) \not\subset |U_i^\infty|$,
and hence $\ol{h}$ factors though $\ol{g} : \ol{N} \to \ol{W}_1$
so that we have $\ol{f}_i=q_i \ol{g}$.
It remains to prove $\nu^*N^\infty \ge \nu^* \ol{g}^* W_1^\infty$,
where $\nu : \ol{N}^N \to \ol{N}$ is the normalization.
As we have 
$\nu^* \ol{g}^*W_1^\infty 
= \nu^*\ol{g}^*\sup(q_1^* U_1^\infty, q_2^* U_2^\infty)
= \sup(\nu^*\ol{f}_1^* U_1^\infty, \nu^*\ol{f}_1^* U_2^\infty)$
by definition and Remark \ref{rem:sup-div},
this follows from
the admissibility of $f_i$,
that is,
$\nu^* \ol{f}_i^\ast U_i^\infty \leq \nu^* N^\infty$.
We have shown that $W_1$ represents
the fiber product in $\ulMSm^\fin$.
Propositions \ref{peff1} and \ref{p1.10}
show that the same holds in $\ulMSm$ as well.
(This also follows from (3) below.)
The last statement is an immediate consequence of the first. 

(2) Let $p_W : \ol{W}_1^N \to \ol{W}_1$ and $p_{U_1} : \ol{U}_1^N \to \ol{U}_1$ be the normalizations.
By the minimality of $u_1$, we have $q_1^\ast U_1^\infty = q_1^\ast \ol{u}_1^\ast M^\infty = q_2^\ast \ol{u}_2^\ast M^\infty \leq q_2^\ast U_2^\infty$, where the last inequality holds by the admissibility  of $u_2$ and the normality of $\ol{U}_2$.
Then $q_1^\ast U_1^\infty$ and $q_2^\ast U_2^\infty$ have a universal supremum since $q_1^\ast U_1^\infty \subset q_2^\ast U_2^\infty$ implies Condition (1) of Lemma \ref{lem:have-sup}, which also implies that $W_1^\infty = \sup (q_1^\ast U_1^\infty,q_2^\ast U_2^\infty) = q_2^\ast U_2^\infty$.
This shows the minimality of $W_1 \to U_2$.

Suppose now $\ol{u}_1$ flat, and let $T$ be an irreducible component of $\ol{W}_0$. Then $p_2:\ol{W}_0\to \ol{U}_2$ is also flat, hence $T$ dominates an irreducible component $E$ of $\ol{U}_2$ \cite[Chapter~III, Proposition~9.5]{hartshorne} and we cannot have $p_2(T)\subset |U_2^\infty|$ since $U_2^\infty$ is everywhere of codimension $1$ in $\ol{U}_2$. Suppose that $p_1(T)\subset |U_1^\infty|$. By the minimality of $u_1$, this implies $u_2p_2(T)=u_1p_1(T)\subset |M^\infty|$, hence $u_2(E)\subset |M^\infty|$, contradicting the admissibility of $u_2$.

(3)
If $\pi$ is the blow-up of $\ol{W}_1$
with center $q_1^*(U_1^\infty) \times_{\ol{W}_1} q_2^*(U_1^\infty)$,
then
$r_1^*U^\infty_1 \times_{\ol{W}_2} r_2^*U^\infty_2$
is precisely the exceptional divisor by definition,
which is therefore an effective Cartier divisor,
showing the first assertion.
Note that $W_2^\o \cong U_1^\o \times_{M^\o} U_2^\o$
by \eqref{eq:inside}.

Now let $f_i : N \to U_i$ be morphisms in $\ulMSm$ for $i=1, 2$
such that $u_1f_1=u_2f_2$. 
Then the morphisms $f_i^\o : N^\o \to U_i^\o$ for $i=1, 2$
induce a unique morphism 
$h^\o :N^\o \to W_2^\o$ 
with $f_i^\o = p_ih^\o$ for $i=1, 2$.
It suffices to prove that $h^\o$ defines a [unique] morphism in $\ulMSm$.  
By the graph trick (Lemma \ref{l-gt}), we may assume that $f_i^\o$ and $h^\o$ extend to  morphisms $\ol{f}_i: \ol{N} \to \ol{U}_i$ and $\ol{h} : \ol{N} \to \ol{W}_2$. 
Moreover we may assume that $\ol{N}$ is normal by Remark \ref{rk-graph-trick} \eqref{gt2}. 
It remains to prove $N^\infty \ge \ol{h}^* W_2^\infty$.
As we have 
$\ol{h}^*W_2^\infty 
= \sup(\ol{f}_1^* U_1^\infty , \ol{f}_1^*U_2^\infty)$
by the assumption and Remark \ref{rem:sup-div},
this follows from
the admissibility of $f_i$,
that is,
$\ol{f}_i^\ast U_i^\infty \leq N^\infty$.
\end{proof}

\begin{remark}\label{rem:fiberprod-int}
If $W$ represents a fiber product $U_1 \times_M U_2$
(either in $\ulMSm$ or in $\ulMSm^\fin$),
then we have $W^\o = U_1^\o \times_{M^\o} U_2^\o$.
Indeed, the functors $\ulMSm \to \Sm$
and $\ulMSm^\fin \to \Sm$
given by $M \mapsto M^\o$
have the left adjoint
$X \mapsto (X, \emptyset)$
(Lemma \ref{l1.2}),
hence commute with limits.
\end{remark}

\begin{exs}
Let $B=k[x_1,x_2]$, $\A^2=\Spec B$, $D_i=\Spec(B/x_iB)$ and $P=D_1\cap D_2$. Let now $M=(D_1\cup D_2, P)$ and $U_i=(D_i,P)$ for $i=1, 2$. 
Then $\ol{W}_0$ is a point but $\ol{W}_1=\emptyset$,
and
$W_1=(\emptyset,  \emptyset)$  indeed represents
the fiber product $U_1 \times_M U_2$.
In particular, fiber products do not commute with
the forgetful functor $M\mapsto \ol{M}$ from $\ulMSm^\fin$ to $\Sch$
of Definition \ref{d1.1} (2).  Another counterexample: let $M=(\A^2,D_1)$, $\ol{U}_1=\Bl_{P}(\A^2)$, $u_1:U_1\to M$ be the minimal induced modulus structure, $U_2=(D_2,P)$ and $U_2\to M$ be given by the inclusion. Then $\ol{W}_1\subsetneq \ol{W}_0$ is the proper transform of $u_1$.
 See however Corollary \ref{exist-pullback} (1).
\end{exs}

Recall \cite[expos\'e~IV, d\'efinition~1.4.0]{SGA3} that a morphism $f:M\to N$ in a category $\sC$ is \emph{squarable} if, for any $g:N'\to N$, the fibred product $N'\times_N M$ is representable in $\sC$. We have:

\begin{cor}\label{exist-pullback}
The following assertions hold.
\begin{enumerate}
\item
If $f : U \to M$ is a minimal morphism in $\ulMSm^\fin$
(see Definition \ref{def:minimality})
such that $f^\o$ is smooth,
then $f$ is squarable in $\ulMSm^\fin$. If $f\in \MSm^\fin$, it is squarable in this category.
 If moreover $\ol{f}$ is flat, then the pull-back by $f$ of morphisms from normal modulus pairs commutes with the forgetful functor $M\mapsto \ol{M}$ from $\ulMSm^\fin$ to $\Sch$  of Definition \ref{d1.1} (2). 
\item If $f : U \to M$ is a morphism in $\ulMSm$ such that $f^\o$ is smooth, then $f$ is squarable in $\ulMSm$. If $f\in \MSm$, it is squarable in this category.
\end{enumerate}
\end{cor}
\begin{proof} (1) follows from Proposition \ref{prop:fiber-prod} (1) and (2); (2) follows from 
Proposition \ref{prop:fiber-prod} (3).
\end{proof}

\begin{cor} Finite products exist in $\ulMSm$ and $\MSm$.
\end{cor}

\begin{proof} This is the special case $M=(\Spec k,\emptyset)$ in Corollary \ref{exist-pullback} (2).
\end{proof}

\section{Presheaf theory}

\subsection{Modulus presheaves with transfers}

\begin{definition}\label{d2.7}
By a presheaf we mean a contravariant functor to the category of abelian groups.
\begin{enumerate}
\item
The category of presheaves on $\MP$ (resp. $\ulMP$, $\ulMP^\fin$)
is denoted by $\MPS$ (resp. $\ulMPS$, $\ulMPS^\fin$).
\item
The category of additive presheaves on $\MCor$ 
(resp. $\ulMCor$, $\ulMCor^\fin$)
is denoted by $\MPST$ (resp. $\ulMPST$, $\ulMPST^\fin$.)
\end{enumerate}
\end{definition}

All these categories are abelian Grothendieck, with projective sets of generators: this is classical for those of (1) and follows from Theorem \ref{t.mon} for those of (2). (See also proof of Proposition \ref{eq:c-functor} below.)

\begin{nota}\label{n2.1} 
We write 
\begin{align*}
\Z_\tr:& \ulMCor\to\ulMPST, 
\quad \MCor\to \MPST,
\\
\Z_\tr^\fin:& \ulMCor^\fin\to \ulMPST^\fin, 
\\
\Z_\tr:&\Cor\to \PST
\end{align*}
for the associated representable presheaves
(\emph{i.e.} $\Z_\tr(M) \in \ulMPST$ is given by
$\Z_\tr(M)(N)=\ulMCor(N, M)$, etc.)
We shall use the common notation $\Z_\tr$
but they will be distinguished by the context.
\end{nota}

We now briefly describe the main properties of the functors induced by those of the previous section.

\subsection{$\protect\MPST$ and $\protect\PST$}

We say $(f_1, f_2, \dots, f_n)$ is a string of adjoint functors
if $f_i$ is a left adjoint of $f_{i+1}$ for each $i=1, \dots, n-1$.

\begin{prop}\label{lem:counit}
The functor $\omega:\MCor\to \Cor$ of \S \ref{s1.2} yields a string of three adjoint functors $(\omega_!,\omega^*,\omega_*)$:
\[\MPST\begin{smallmatrix}\omega_!\\\longrightarrow\\\omega^*\\\longleftarrow\\\omega_*\\ \longrightarrow\end{smallmatrix}\PST, 
\]
where $\omega^*$ is fully faithful and $\omega_!,\omega_*$ are localisations; $\omega_!$ 
has a pro-left adjoint $\omega^!$, hence is exact. 

Similarly, $\omega_s:\MSm\to \Sm$
yields a string of three adjoint functors $(\omega_{s!},\omega_s^*,\omega_{s*})$; 
$\omega_s^*$ is fully faithful and $\omega_{s!},\omega_{s*}$ are localisations; 
$\omega_{s!}$ 
has a pro-left adjoint $\omega_s^!$, hence is exact. 
\end{prop}

\begin{proof} 
This follows from Theorems \ref{prop:localization} and 
\ref{lem:omega-sh}. 
\end{proof}

Let $X \in \Sm$ and let $M \in \MP(X)$. 
Lemma \ref{ln2} and Proposition \ref{p.funct} show that
the inclusions $\{M^{(n)}\mid n>0\}\subset \MP(\ol{M}!X) \subset \MP(X)$ induce isomorphisms
(see Definition~\ref{def:mpx})
\begin{equation}\label{rem:single-cpt}
\omega_!(F)(X)\simeq\colim_{N \in \MP(X)} F(N)\simeq \colim_{N \in \MP(\ol{M}!X)} F(N)\simeq \colim_{n>0} F(M^{(n)}).
\end{equation}

\bigskip

\subsection{$\protect\ulMPST$ and $\protect\PST$} 

\begin{prop}\label{p4.2} The adjoint functors $(\lambda,\ulomega)$ of Lemma \ref{l1.2} induce a string $(\lambda_!=\ulomega^!,\lambda^*=\ulomega_!,\lambda_*=\ulomega^*,\ulomega_*)$ of four adjoint functors:
\[\ulMPST\begin{smallmatrix}\ulomega^!\\\longleftarrow\\\ulomega_!\\\longrightarrow\\\ulomega^*\\\longleftarrow\\\ulomega_*\\ \longrightarrow\end{smallmatrix}\PST, \]
where $\ulomega_!,\ulomega_*$ are localisations while $\ulomega^!$ and $\ulomega^*$ are fully faithful. 
Moreover, if $X\in \Cor$ is proper, we have a canonical isomorphism $\ulomega^*\Z_\tr(X)\allowbreak\simeq \Z_\tr(X,\emptyset)$.
\end{prop}

\begin{proof} The only non obvious statement is the last claim, which follows from Lemma \ref{l1.2}.
\end{proof}

\subsection{\protect{$\MPST$ and $\ulMPST$}}

\begin{prop}\label{eq.tau}
The functor $\tau:\MCor\to \ulMCor$ of  \eqref{eq.taulambda} yields a string of three adjoint functors $(\tau_!,\tau^*,\tau_*)$:
\[\MPST\begin{smallmatrix}\tau_!\\\longrightarrow\\\tau^*\\\longleftarrow\\\tau_*\\ \longrightarrow\end{smallmatrix}\ulMPST, \]
where $\tau_!,\tau_*$ are fully faithful and $\tau^*$ is a localisation; $\tau_!$ 
has a pro-left adjoint $\tau^!$, hence is exact.
There are natural isomorphisms
\[\omega_!\simeq \ulomega_!\tau_!, \quad\omega_* \simeq \ulomega_* \tau_*, \quad \omega^! \simeq \tau^! \ulomega^!.\]
The same holds for the functor 
$\tau_s$ from Theorem \ref{t2.1}.
Namely, we have
a string of three adjoint functors $(\tau_{s!},\tau_s^*,\tau_{s*})$
and they satisfy
\[\omega_{s!}\simeq \ulomega_{s!}\tau_{s!}, 
\quad \omega_{s*} \simeq \ulomega_{s*} \tau_{s*}, 
\quad \omega_s^! \simeq \tau_s^! \ulomega_s^!.
\]
\end{prop}

\begin{proof} This follows from Theorem \ref{t2.1} and Proposition \ref{p.funct}.
\end{proof}

\begin{lemma}\label{lem:tau-colim}\
\begin{enumerate}
\item
For 
$G \in \MPST, G' \in \MPS$ and $M \in \ulMP$,
we have
\[ \colim_{N\in \Comp(M)} G(N) \simeq \tau_!G(M),
\quad
\colim_{N\in \Comp(M)} G'(N) \simeq \tau_{s !}G'(M).
\]
\item
The unit maps 
$\id \to \tau^* \tau_!$ and 
$\id \to \tau_s^* \tau_{s !}$ 
are isomorphisms.
\item
There is an natural isomorphism 
$\tau_! \omega^*\simeq \ulomega^*$.
\end{enumerate}
\end{lemma}
\begin{proof}
(1) This follows from Lemma \ref{l1.1}, Theorem \ref{t2.1}
and Proposition \ref{p.funct}.

(2) This follows from (1)
since $\Comp(M)=\{ M \}$ for $M \in \MP$.

(3)
For $F\in \PST$ and $M\in \MCor$, we compute
$$
\tau_! \omega^* F(M) = \colim_{N\in \Comp(M)} \omega^*F(N)
 =\colim_{N\in \Comp(M)} F(N^\o) = F(M^\o)=\ulomega^*F(M).
$$
We are done.
\end{proof}

\begin{rk}\label{rem:formula-tau}
By Lemma \ref{l1.1} we have the formulas
\[\tau^! \Z_\tr(M) = \underset{N\in \Comp(M)}{``\lim"} \Z_\tr(N), \quad \tau^* \Z_\tr(M) = \lim_{N\in \Comp(M)} \Z_\tr(N),
\]
where the latter inverse limit is computed in $\MPST$.
\end{rk}

\begin{qn}\label{q.exact} Is $\tau^!$ exact?
\end{qn}

\subsection{\protect{$\ulMPST^\fin$ and $\ulMPST$}}

\begin{prop} \label{eq:bruno-functor} 
Let 
$\ul{b}_s : \ulMSm^\fin \to \ulMSm$ and 
$\ul{b}:\ulMCor^\fin\to \ulMCor$ be the inclusion functors from \eqref{eq:def-b}.
Then $\ul{b}_s$ and $\ul{b}$ yield strings of three adjoint functors  $(\ul{b}_{s !},\ul{b}_s^*,\ul{b}_{s *})$ and $(\ul{b}_!,\ul{b}^*,\ul{b}_*)$:
\[
\ulMPS^\fin\begin{smallmatrix}\ul{b}_{s !}\\\longrightarrow\\ \ul{b}_s^*\\\longleftarrow\\ \ul{b}_{s *}\\ \longrightarrow\end{smallmatrix}\ulMPS ,
\quad
\ulMPST^\fin\begin{smallmatrix}\ul{b}_!\\\longrightarrow\\ \ul{b}^*\\\longleftarrow\\ \ul{b}_*\\ \longrightarrow\end{smallmatrix}\ulMPST, 
\]
where $\ul{b}_{s !},\ul{b}_{s *}$, $\ul{b}_!,\ul{b}_*$ are localisations;  $\ul{b}_s^*$, $\ul{b}^*$ are exact and fully faithful;  $\ul{b}_{s !}$, $\ul{b}_!$  
have pro-left adjoints, hence are exact.
The counit maps $\ul{b}_{s !} \ul{b}_s^\ast \to \id$ and $\ul{b}_! \ul{b}^\ast \to \id$ are isomorphisms.
For $F_s \in \ulMPS^\fin$, $F\in \ulMPST^\fin$ and $M\in \mathrm{Ob}(\ulMSm) = \mathrm{Ob}(\ulMCor)$, we have
(see Def. \ref{deff})
\begin{equation}\label{eq:b-sh-explicit}
\ul{b}_{s !} F_s (M) = \colim_{N\in \ul{\Sigma}^{\fin}\downarrow M} F_s (N), 
\quad 
\ul{b}_! F(M) = \colim_{N\in \ul{\Sigma}^{\fin}\downarrow M} F(N).
\end{equation}
\end{prop}

\begin{proof} This follows from the usual yoga applied with Proposition \ref{peff1} and Lemma \ref{lA.6}.
\end{proof}

\subsection{With and without transfers}

\begin{prop} \label{eq:c-functor} 
Let $\ul{c} :\ulMSm \to \ulMCor$ be the functor from \eqref{eq:def-c}.
Then 
$\ul{c}$ yields a string of three adjoint functors $(\ul{c}_!,\ul{c}^*,\ul{c}_*)$:
\[\ulMPS\begin{smallmatrix}\ul{c}_!\\\longrightarrow\\ \ul{c}^*\\\longleftarrow\\ \ul{c}_*\\ \longrightarrow\end{smallmatrix}\ulMPST,
\]
where $\ul{c}^*$ is exact and faithful (but not full). We have
\begin{equation}\label{eq2.6}
\ul{c}_! \Z^p(M) = \Z_\tr(M)
\end{equation}
for any $M\in \ulMSm$, 
where $\Z^p(M)$ is 
\footnote{
We put a superscript $p$ to distinguish it from
its associated sheaf $\Z(M)$, to be introduced in  \eqref{Z}.}
the presheaf $N\mapsto \Z[\ulMSm(N,M)]$.

The same statements hold for
$c : \MSm \to \MCor$ and 
$\ul{c}^\fin ; \ulMSm^\fin \to \ulMCor^\fin$ 
from \eqref{eq:def-c}.
Precisely, 
they yield strings of three adjoint functors 
$(c_!, c^*, c_*)$
and $(\ul{c}^\fin_!,\ul{c}^{\fin *},\ul{c}^\fin_*)$;
$c^*$ and $\ul{c}^{\fin *}$ are exact and faithful.
(The analogue of \eqref{eq2.6} also holds for $c$ and $\ul{c}^\fin$,
but we will not need it.)
\end{prop}

\begin{proof} To define $\ul{c}_!,\ul{c}^*$ and $\ul{c}_*$, we use the free additive category $\Z\ulMSm$ on $\ulMSm$ \cite[Chapter~VIII, Section~3, Exercises~5 \& 6]{mcl}: it comes with a canonical functor $\gamma: \ulMSm \to \Z\ulMSm$ and is $2$-universal for contravariant functors to additive categories. In particular:
\begin{itemize}
\item The functor $\ul{c}$ induces an additive functor $\tilde{\ul{c}}:\Z\ulMSm\to \ulMCor$.
\item 
By the $2$-universality,
the functor $\gamma$ induces an equivalence 
$\gamma^* : \Mod\Z\ulMSm \cong \ulMPS$,
where $\Mod\Z\ulMSm$ denotes the category
of additive contravariant functors $\Z \ulMSm \to \Ab$
\item For $M,N\in \ulMSm$, we have a canonical isomorphism 
\[\Z\ulMSm(\gamma(N),\gamma(M))\simeq \Z[\ulMSm(N,M)].\]
\end{itemize}
As usual, $\tilde{\ul{c}}$ induces
a string of three adjoint functors 
$(\tilde{\ul{c}}_!, \tilde{\ul{c}}^*, \tilde{\ul{c}}_*)$
(see \S \ref{s.presh}).
We then define $\ul{c}_!$ as 
$\tilde{\ul{c}}_! \circ (\gamma^*)^{-1}$, etc. 
Everything follows from this except the faithfulness of $\ul{c}^*$,
which is a consequence of the essential surjectivity of $\ul{c}$. The cases of $\ul{c}^\fin$ and $c$ are dealt with similarly.
\end{proof}

\begin{lemma}\label{lem:b-c-tau}
\begin{enumerate}
\item 
We have
\begin{equation}\label{eq:b-and-c}
\ul{c}^{\fin *} \ul{b}^* = \ul{b}_s^* \ul{c}^*,
\quad
\ul{b}_! \ul{c}^\fin_! = \ul{c}_! \ul{b}_{s !},
\quad
\ul{c}^* \ul{b}_!= \ul{b}_{s !} \ul{c}^{\fin *}.
\end{equation}
\item 
We have
\begin{equation}\label{eq:c-and-tau}
c^* \tau^* = \tau_s^* \ul{c}^*,
\quad
\ul{c}^* \tau_! = \tau_{s !} c^*,
\quad
\ul{c}^{\fin *} \tau_!^\fin = \tau_{s !}^\fin c^{\fin *}.
\end{equation}
\end{enumerate}
\end{lemma}
\begin{proof}
The first two equalities of (1) follows from the equality
$\ul{b} ~\ul{c}^\fin = \ul{c} ~\ul{b}_s$ 
(see \eqref{eq:six-cat-diag0}).
Similarly, the first equality of (2) follows from
$\tau c = \ul{c} \tau_s$.
By \eqref{eq:b-sh-explicit}, we have
\[ \ul{c}^* \ul{b}_! F(M)
\cong \colim_{N \in \ul{\Sigma}^\fin \downarrow M} F(N) 
\cong \ul{b}_{s !} \ul{c}^{\fin *}F(M)
\]
for any $F \in \MPST^\fin$ and $M \in \ulMSm$.
(Note that all morphisms of $\ul{\Sigma}^\fin \downarrow M$
are in $\ulMSm^\fin$,
and that both of $\ul{b}_!$ and $\ul{b}_{s !}$
can be computed by using the same $\ul{\Sigma}^\fin \downarrow M$.)
This proves the last formula of (1).
Lemma \ref{lem:tau-colim} (1) shows that
\[ \ul{c}^*\tau_!F(M)
\cong \colim_{N \in \Comp(M)} F(N) 
\cong \tau_{s !}c^*F(M)
\]
for any $F \in \MPST$ and $M \in \ulMSm$.
The last one of (2) is similar.
\end{proof}

\subsection{A patching lemma}

By the previous lemma, 
we obtain a commutative diagram of categories
(cf. \eqref{eq:six-cat-diag0}):
\begin{equation}\label{eq:six-cat-diag}\vcenter{
\xymatrix{
\MPST \ar[r]^{\tau_!} \ar[d]^{c^*} &
\ulMPST \ar[r]^{\ul{b}^*} \ar[d]^{\ul{c}^*} &
\ulMPST^{\fin} \ar[d]^{\ul{c}^{\fin *}}
\\
\MPS \ar[r]^{\tau_{s!}} &
\ulMPS \ar[r]^{\ul{b}_s^*}  &
\ulMPS^{\fin}.
}}
\end{equation}
All vertical arrows are faithful
and horizontal ones fully faithful.

\begin{lemma}\label{lem:patching}
Both squares of \eqref{eq:six-cat-diag} are ``$2$-Cartesian''.
More precisely, the following assertions hold.
\begin{enumerate}
\item
Let 
$\ulMPS \times_{\ulMPS^\fin} \ulMPST^\fin$
be the category of pairs
$(F_s, F_t)$ consisting of
$F_s \in \ulMPS$ and $F_t \in \ulMPST^\fin$ 
such that their restriction to the common subcategory
$\ulMSm^\fin$ are equal.
The functor
\[ \ulMPST \to \ulMPS \times_{\ulMPS^\fin} \ulMPST^\fin,
\]
defined by $F \mapsto (\ul{c}^*F, \ul{b}^* F)$
is an equivalence of categories.
\item 
Let $\MPS \times_{\ulMPS} \ulMPST$
be the category of triples
$(F_s, F_t, \phi)$ consisting of
$F_s \in \MPS, ~F_t \in \ulMPST$ and
an isomorphism $\phi : \tau_{s!} F_s \cong \ul{c}^* F_t$
in $\ulMPS$.
The functor
\[ \MPST \to \MPS \times_{\ulMPS} \ulMPST,
\]
defined by $F \mapsto (c^*F, \tau_! F, \theta_F)$,
where $\theta_F : \tau_{s!} c^*F \cong \ul{c}^* \tau_! F$
is from \eqref{eq:c-and-tau},
is an equivalence of categories.
\end{enumerate}
\end{lemma}
\begin{proof}
(1) is the content of Corollary \ref{cor:fiber-cat}.
We show (2). Given $(F_s, F_t, \phi)$,
we shall construct $F \in \MPST$ as follows.
Set $F(M):=F_s(cM)$ for any $M \in \MCor$.
Since $M$ is proper, we have
an isomorphism
\[ F(M)=F_s(cM) = \tau_{s!}F_s(\tau_s cM)
\overset{\phi_M}{\longrightarrow}
\ul{c}^*F_t(\ul{c}\tau M) = F_t(\tau M),
\]
which we denote by $\til{\phi}_M$.
For $\gamma \in \MCor(M, N)$, we define
$F(\gamma):=\til{\phi}_M^{-1} F_t(\gamma) \til{\phi}_N$.
It is straightforward to see that
$(F_s, F_t, \phi) \mapsto F$ gives a quasi-inverse.
\end{proof}

\subsection{The functors $n_!$ and $n^*$}

As in \S \ref{s.presh}, the functor $(-)^{(n)}$ of Definition \ref{dn1} 
induces a string of adjoint endofunctors $(n_!,n^*,n_*)$ of $\MPST$,
where $n^*$ is given by $n^*(F)(M)=F(M^{(n)})$. We shall not use $n_*$ in the sequel.



\begin{lemma}\label{ln3} The functor $n_!$ is fully faithful.
\end{lemma}

\begin{proof} This follows formally from the same properties of $(-)^{(n)}$. 
\end{proof}

\begin{prop}\label{p4.1} For any $F\in \MPST$, there is a natural isomorphism
\[\omega^*\omega_! F \simeq \infty^*F, \]
where $\infty^* F(M):= \colim_n F(M^{(n)})$ $($for the natural transformations \eqref{eqn1}$)$.
\end{prop}

\begin{proof} Let $M\in \MCor$ and $X=\omega M$. Then
\[\omega^*\omega_! F(M) = \colim_{M'\in \MP(X)} F(M'), \]
and the claim follows from Lemma \ref{ln2}.
\end{proof}

\begin{prop}\label{pn1} For all $n\ge 1$, the natural transformation $\omega_!\to\omega_!n^*$ stemming from \eqref{eqn1} is an isomorphism.
\end{prop}

\begin{proof} Let $F\in \MPST$. For $X\in \Cor$, we have
\[\omega_!n^*F(X)=\colim_{M\in \MP(X)} n^*F(M)
=\colim_{M\in \MP(X)} F(M^{(n)})=\colim_{M\in \MP(X)} F(M),
\]
where the last isomorphism follows from Lemma \ref{ln2}.
\end{proof}

\section{Sheaves on {$\protect\ulMP^\protect\fin$ and $\protect\ulMCor^\protect\fin$} }

\subsection{Nisnevich topology on $\protect\ulMP^{\protect\fin}$}

\begin{definition}\label{def:groth-top-mp}
We call a morphism $p : U \to M$ in $\ulMP^{\fin}$
a Nisnevich cover if
\begin{thlist}
\item $\ol{p} : \ol{U} \to \ol{M}$ is a Nisnevich cover
of $\ol{M}$ in the usual sense;
\item $p$ is minimal (that is, $U^\infty = \ol{p}^*(M^\infty)$).
\end{thlist}
Since the morphisms appearing in
the Nisnevich covers are squarable by
Corollary \ref{exist-pullback} (1),
we obtain a Grothendieck topology on $\ulMP^{\fin}$. 
The category $\ulMP^{\fin}$ endowed with
this topology will be called 
the big Nisnevich site of $\ulMP^{\fin}$ and denoted by 
$\ulMP^{\fin}_\Nis$.
\end{definition}

\begin{definition}\label{d3.3}
Let us fix $M \in \ulMP^{\fin}$.
Let $M_{\Nis}$ be the category
of minimal morphisms $f : N \to M$ in $\ulMP^{\fin}$
such that $\ol{f}$ is \'etale,
endowed with the topology induced by $\ulMP^{\fin}_{\Nis}$.
\end{definition}

The following lemma is obvious from the definitions:

\begin{lemma}\label{lem:equiv-smallsites}
Let $M \in \ulMP^{\fin}$.
Let $(\ol{M})_\Nis$ be 
the (usual) small Nisnevich site on $\ol{M}$.
Then we have an isomorphism of sites
\[ M_\Nis \to (\ol{M})_\Nis, \qquad N \mapsto \ol{N},
\]
whose inverse is given by 
$(p : X \to \ol{M}) \mapsto (X, p^*(M^\infty))$. (This isomorphism of sites depends on the choice of $M^\infty$.) \qed
\end{lemma}

\begin{lemma}\label{lem:refine-cover}
Let $\alpha : M \to N$ be a morphism in $\ulMCor^\fin$ and
let $p : U \to N$ be a Nisnevich over in $\ulMP^\fin$.
Then there is a commutative diagram
\[
\begin{CD}
V @>\alpha'>> U \\
@V{p'}VV @VV{p}V \\
M @>>\alpha> N,
\end{CD}
\]
where $\alpha' : V \to U$ is a morphism in $\ulMCor^\fin$
and $p' : V \to M$ is a Nisnevich cover in $\ulMP^\fin$.
\end{lemma}
\begin{proof}
We may assume $\alpha$ is integral.
Let $\ol{\alpha}$ be the closure of $\alpha$
in $\ol{M} \times \ol{N}$.
Since $\ol{\alpha}$ is finite over $\ol{M}$,
we may find a Nisnevich cover $p' : \ol{V} \to \ol{M}$
such that 
$\tilde{p}$ in the diagram
(all squares being cartesian)
\[
\begin{CD}
\ol{V} \times_{\ol{M}} \times (\ol{\alpha} \times \ol{U})
@>>> \ol{\alpha} \times_{\ol{N}} \ol{U} @>>> \ol{U}
\\
@V{\tilde{p}}VV @VVV @VV{\ol{p}}V
\\
\ol{V} \times_{\ol{M}} \times \ol{\alpha}
@>>> \ol{\alpha} @>>> \ol{N}
\\
@VVV @VVV
\\
\ol{V} 
@>>{p'}> \ol{M}
\end{CD}
\]
has a splitting $s$.
Put $V:=(\ol{V}, {p'}^*(M^\infty)) \in \ulMP$.
The image of $s$ gives us a desired correspondence
$\alpha'$.
\end{proof}

\begin{remark}
One can also define the Zariski and \'etale topologies
on $\ulMSm^\fin$.
Most results of this section
(notably Theorems \ref{thm:cech}, \ref{thm:sheaf-transfer},
and Corollary \ref{rem:faith-exact-c})
remain true for the \'etale topology,
but not for the Zariski topology
(e.g. Lemma \ref{lem:refine-cover} already fails for it).

However, from the next section onward
we will make essential use of cd-structures.
As the \'etale topology cannot be defined by a cd-structure,
we decide to stick to the Nisnevich topology from the beginning.
\end{remark}

\subsection{A cd-structure on $\protect\ulMP^{\protect\fin}$}\label{s3.2}Let $\Sq$ be the product category 
of $[0]=\{ 0 \to 1 \}$ with itself, depicted as
\[
\xymatrix{
00 \ar[r] \ar[d] &01 \ar[d]\\
10 \ar[r] & 11.
}
\]
For any category $\sC$,
denote by $\sC^\Sq$ for the category of 
functors from $\Sq$ to $\sC$.
A functor $f : \sC \to \sC'$ 
induces a functor $f^\Sq : \sC^\Sq \to {\sC'}^\Sq$.

We refer to \S \ref{sa-cd} for the notion of cd-structure, and its properties.

\begin{defn}\label{d3.2} \
\begin{enumerate}
\item 
A Cartesian square
\begin{equation}\label{eq.cd}
\begin{CD}
W@>v>> V\\
@VqVV @Vp VV\\
U@>u>> X,
\end{CD}
\end{equation}
in $\Sch$ 
is called an \emph{elementary Nisnevich square}
if $p$ is \'etale, $p^{-1}(X \setminus U)_\red \to (X \setminus U)_\red$
is an isomorphism and $u$ is an open embedding,.
In this situation, we say $U \sqcup V \to X$
is an \emph{elementary Nisnevich cover}.
Recall that an additive presheaf 
is a Nisnevich sheaf if and only if 
it transforms any elementary Nisnevich square
into a cartesian square 
\cite[Corollary~2.17]{cdstructures}, \cite[Theorem~2.2]{unstableJPAA}.
\item 
A diagram \eqref{eq.cd} in $\ulMP^\fin$ is 
called 
an \emph{$\ulMVfin$-square} 
if all morphisms are minimal and it becomes an elementary Nisnevich square (in $\Sch$) 
after applying the forgetful functor of Definition \ref{d1.1} (2). 
\end{enumerate}
\end{defn}

\begin{lemma}\label{lem;d3.2} 
A $\ulMVfin$-square \eqref{eq.cd} is cartesian in $\ulMSm^\fin$.
\end{lemma}
\begin{proof}
The last part of Proposition \ref{prop:fiber-prod} (2)
shows that no irreducible component of $\ol{X}$ has its image
inside $|U^\infty|$ or $|V^\infty|$
(\emph{i.e.} $\ol{W}_1$ in \emph{loc.~cit.} agrees with $\ol{W}$),
and then Proposition \ref{prop:fiber-prod} (1)
shows that $X$ is the fiber product 
since $q^*U^\infty=v^*V^\infty=W^\infty$ by minimality.
\end{proof}

\begin{prop}\label{p3.1} The following assertions hold.
\begin{enumerate}
\item  The topology on $\ulMSm_\Nis^\fin$ (cf. Def. \ref{def:groth-top-mp}) coincides with the topology associated with the cd-structure $P_{\ulMVfin}$ consisting of $\ulMVfin$-squares. 
\item The cd-structure $P_{\ulMV^\fin}$ is strongly complete and strongly regular in the sense of Definition \ref{dA.5}, hence complete and regular in the sense of \cite{cdstructures} (cf. Definition \ref{d3.1}).
\end{enumerate}
\end{prop}

\begin{proof} (1) follows from Lemma \ref{lem:equiv-smallsites} and \cite[Remark after Proposition~2.17]{unstableJPAA}. 
The first assertion of (2) follows from (the proof of) \cite[Theorem~2.2]{unstableJPAA}.
The second assertion of (2) follows from \cite[Lemma~2.5, Lemma~2.11]{cdstructures}
\end{proof}

\subsection{Sheaves on $\protect\ulMP^{\protect\fin}$}

\begin{definition}
We define $\ulMNS^{\fin}$ to be
the full subcategory of $\ulMPS^{\fin}$
consisting of Nisnevich sheaves.
\end{definition}

\begin{thm}\label{t.cd} 
Let $F\in\ulMNS^{\fin}$. Then
$H^i_\Nis(X,F)=0$ for any $X \in \ulMP^\fin$ 
and $i>\dim X$ $($where $\dim X$ is defined as $\dim X:=\dim X^\o =\dim \Xb)$. 
\end{thm}

\begin{proof} This is clear from Lemma \ref{lem:equiv-smallsites} 
and the known properties of the Nisnevich site.
\end{proof}

\begin{defn}\label{dA.2} An additive functor $F$ between additive categories is \emph{strongly additive} if it commutes with infinite direct sums.
\end{defn}


This property is not used in the present paper, but it will be essential in \cite{modsheafII} when we deal with unbounded derived categories. 


\begin{lemma}\label{lcom1} The category $\ulMNS^{\fin}$ is closed under infinite direct sums and the inclusion functor denoted by
$\ul{i}_{s  \Nis}^{\fin} : \ulMNS^{\fin} \to \ulMPS^{\fin}$  is strongly additive.
\end{lemma}

\begin{proof} Indeed, the sheaf condition is tested on finite diagrams, hence the presheaf given by a direct sum of sheaves is a sheaf (small filtered colimits commute with finite limits, \cite[Chapter~IX, Section~2, Theorem~1]{mcl}).
\end{proof}

\begin{proposition}\label{prop:rep-sheaf} 
For any $M \in \ulMP$ 
we have
\[ \ul{c}^{\fin *} \Z_{\tr}^\fin(M), \quad 
\ul{c}^{\fin *} \ul{b}^* \Z_{\tr}(M) 
\in \ulMNS^{\fin},
\]
where 
$\Z_\tr^\fin,\Z_\tr$ are the representable presheaves 
(notation \ref{n2.1}) and 
the functors 
$\ul{b}^*$ and
$\ul{c}^{\fin *}$
are from  Propositions \ref{eq:bruno-functor}, \ref{eq:c-functor}.
\end{proposition}

\begin{proof}
We show the stronger statement
that $\Z_\tr(M)$ restricts to an \'etale sheaf
on $\ul{N}_\et$ for any $N \in \ulMCor^\fin$.
Let $p : \ol{U} \to \ol{N}$ be an \'etale cover
and let $U:=(\ol{U}, p^*N^\infty)$.
We have a commutative diagram
\[
\xymatrix{
0 \ar[r] &
\ulMCor^\fin(N, M) \ar[r] \ar@{^{(}->}[d] &
\ulMCor^\fin(U, M) \ar[r] \ar@{^{(}->}[d] &
\ulMCor^\fin(U \times_N U, M) \ar@{^{(}->}[d]
\\
0 \ar[r] &
\ulMCor(N, M) \ar[r] \ar@{^{(}->}[d] &
\ulMCor(U, M) \ar[r] \ar@{^{(}->}[d] &
\ulMCor(U \times_N U, M) \ar@{^{(}->}[d] 
\\
0 \ar[r] &
\Cor(N^\o, M^\o) \ar[r] &
\Cor(U^\o, M^\o) \ar[r] &
\Cor(U^\o \times_{N^\o} U^\o, M^\o). 
}
\]
%
%
%
The bottom row is exact by \cite[Lemma 6.2]{mvw}.
The exactness of the top and middle row now follows from Lemma \ref{l1.3}.
\end{proof}

\subsection{\protect\v{C}ech complex}

Let $p : U \to M$ be a Nisnevich cover in $\ulMP^{\fin}$.
We write $U \times_M U$ for the modulus pair corresponding 
to $\ol{U}\times_{\ol{M}}\ol{U}$ 
under the isomorphism of sites from Lemma \ref{lem:equiv-smallsites}. 
Note that it is a fibre product in 
$\ulMSm^\fin$ and in $\ulMSm$, 
thanks to Proposition \ref{prop:fiber-prod}.
Iterating this construction, 
we obtain the \v{C}ech complex
\begin{equation}\label{eq:ceck}
 \dots \to \ul{c}^{\fin *}\Z_{\tr}^\fin(U \times_M U)
 \to \ul{c}^{\fin *}\Z_{\tr}^\fin(U)
 \to \ul{c}^{\fin *}\Z_{\tr}^\fin(M)
 \to 0
\end{equation}
in $\ulMNS^{\fin}$.

\begin{thm}\label{thm:cech}
The complex
 \eqref{eq:ceck}
is exact in $\ulMNS^{\fin}$.
\end{thm}

\begin{remark} 
This result will be refined several times,
see Corollary \ref{rem:faith-exact-c} and 
Theorem \ref{thm:cech2}. 
Its proof is adapted from \cite[Proposition~3.1.3]{voetri}.
\end{remark}

Before starting the proof of Theorem \ref{thm:cech},
it is convenient to generalize the notion of relative cycles to the modulus setting.

\begin{defn}\label{d3.4} Let $S=(\ol{S},D)$, $Z=(\ol{Z},Z^\infty)$ be two pairs formed of a scheme and an effective Cartier divisor, and let $f:\ol{Z}\to\ol{S}$ be a morphism. (We don't put any regularity requirement on $\ol{S}-|D|$ or $\ol{Z}-|Z^\infty|$.) We write $L(Z/S)$ for the free abelian group with basis the closed integral subschemes $T\subset \ol{Z}$ such that
$T$ is finite and surjective over an irreducible component of $\ol{S}$ and $D|_{T^N} \ge Z^\infty|_{T^N}$,
where $T^N \to T$ is normalization and $(-)|_{T^N}$ denotes pull-back of Cartier divisors.
\end{defn}

\begin{ex} If $S$ is a modulus pair and $M=(\ol{M},M^\infty)$ is another modulus pair, then we have a canonical isomorphism $\MCor^\fin(S,M) \simeq L(\ol{S}\times M/S)$, where $\ol{S}\times M$ is the modulus pair $(\ol{S}\times \ol{M},\ol{S}\times M^\infty)$: this follows from Remark \ref{rk-graph-trick} \eqref{gt4}.
\end{ex}

Define a category $\sD(S)$ as follows: objects are pairs $(Z,f)$ as in Definition \ref{d3.4}. A morphism in $\sD(S)$, $(Z,f)\to (Z',f')$, is a minimal morphism $\phi:Z\to Z'$ such that $f=f'\circ \ol{\phi}$. Composition is obvious.

\begin{lemma}\label{l3.4} The push-forward of cycles makes $(Z,f)\mapsto L(Z/S)$ a covariant functor on $\sD(S)$.
\end{lemma}

\begin{proof} Let $\phi:(Z,f)\to (Z',f')$ be a morphism in $\sD(S)$, and let $T\in L(Z/S)$. Then $\phi(T)$ is still finite and surjective over a component of $\ol{S}$ \cite[Lemma 1.4]{mvw}. Moreover, it still verifies the modulus condition: this follows from the minimality of $\phi$ and from  Lemma \ref{lKL}. We set as usual $\phi_* T = [k(T):k(\phi(T))]\phi(T)$: this defines $\phi_*:L(Z/S)\to L(Z'/S)$, and the functoriality $(\psi\circ \phi)_*=\psi_*\circ \phi_*$ is obvious.
\end{proof}

\begin{proof}[Proof of Theorem \ref{thm:cech}] 
In view of Lemma \ref{lem:equiv-smallsites}, 
it suffices to show the exactness of
\eqref{eq:ceck} evaluated at
$S$
\begin{equation}\label{eq:ceck2}
\dots 
\to \ulMCor^\fin(S, U \times_M U)
\to \ulMCor^\fin(S, U)
\to \ulMCor^\fin(S, M)
\to 0
\end{equation}
for the henselisation $S=(\ol{S}, D)$ of any modulus pair $N=(\ol{N},N^\infty)$ at any point of $\ol{N}$. As in \cite{voetri}, the strategy is to write \eqref{eq:ceck2} as a filtered colimit of contractible chain complexes.

Write $\sE(S,M)$ for the collection of integral closed subsets of $S^\o\times M^\o$ which belong to $\ulMCor^\fin(S,M)$ (this is the canonical basis of $\ulMCor^\fin(S,M)$). 
Let $\sC(M)$ be the set of closed subschemes of $\ol{S} \times \ol{M}$
that are quasi-finite over  
$\ol{S}$ and not contained in $\ol{S}\times M^\infty$, viewed as an (ordered, cofiltered) category. To $Z\in \sC(M)$
we associate the subset $\sE(Z)\subset \sE(S,M)$ of those $F$ such that $F\subset Z$. 

Provide $Z\in \sC(M)$ with the minimal modulus structure induced by the projection $Z\to \ol{M}$ (in a sense slightly generalized from Remark \ref{r2.1} (4), as in Definition \ref{d3.4}: the open subset $Z-Z^\infty$ is not necessarily smooth). This yields a functor
\[\sC(M)\to \sD(S)\]
where $\sD(S)$ is the category defined above. In particular, we have a subgroup $L(Z/S)\subset \ulMCor^\fin(S,M)$: it is the free abelian group on $\sE(Z)$.

Let $u:M'\to M$ be an \'etale morphism in $\ulMSm^\fin$, as in Definition \ref{d3.3}. For $Z\in \sC(M)$, define $u^*Z = Z\times_{\ol{M}} \ol{M}' $. Then $u^*Z\in \sC(M')$, and there is a commutative diagram
\[\begin{CD}
L(u^*Z/S)@>>> L(Z/S)\\
@VVV @VVV\\
\ulMCor^\fin(S,M')@>u_*>> \ulMCor^\fin(S,M),
\end{CD}\]
where the bottom horizontal map is composition 
by the graph of $u$. This yields subcomplexes
\begin{equation}\label{eq:ceck3}
 \dots \to
L(Z \times_{\ol{M}} (\ol{U} \times_{\ol{M}} \ol{U})) \to
L(Z \times_{\ol{M}} \ol{U}) \to
L(Z) \to 0
\end{equation}
of \eqref{eq:ceck2}, for $Z\in \sC(M)$.

Let $\sC_f(M)\subset  \sC(M)$ be the subset of those $Z$ which are finite over $\ol{S}$. It is a filtered subcategory, and we have
\[ \sE(S,M')= \bigcup_{Z\in \sC_f(M)} \sE(u^*Z).\]

Indeed, for $Z'\in \sC(M')$, let $Z=\ol{(\id_{\ol{S}} \times \ol{u})(Z')}$ and let $Z_f= \bigcup_{F\in \sE(Z)} F$. Then $\sE(Z')\subset \sE(u^*Z_f)$ since $(\id_{\ol{S}} \times \ol{u})(F)$ is finite over $\ol{S}$ for $F\in \sE(Z')$.

This proves that
\eqref{eq:ceck2} is obtained as the filtered inductive limit of the complexes \eqref{eq:ceck3}
when $Z$ ranges over $\sC_f(M)$. 
It suffices to show the exactness of \eqref{eq:ceck3} for such a $Z$.

Since $Z$ is finite over the henselian local scheme $\ol{S}$,
$Z$ is a disjoint union of henselian local schemes.
Thus the Nisnevich cover
$Z \times_{\ol{M}} \ol{U} \to Z$ admits a section
$s_0 : Z \to Z \times_{\ol{M}} \ol{U}$.
Define for $k \geq 1$
\[ 
s_k := s_0 \times_{\ol{M}} \id_{\ol{U}^k} :
Z \times_{\ol{M}} \ol{U}^k \to 
Z \times_{\ol{M}} \ol{U} \times_{\ol{M}} \ol{U}^k =
Z \times_{\ol{M}} \ol{U}^{k+1}
\]
where
$\ol{U}^{k} := \ol{U} \times_{\ol{M}} \dots \times_{\ol{M}} \ol{U}$.
Then the maps
\[
L(Z \times_{\ol{M}} \ol{U}^k)\to
L(Z \times_{\ol{M}} \ol{U}^{k+1})
\]
induced by $s_k$ via Lemma \ref{l3.4}
give us a homotopy from the identity to zero.
\end{proof}

\subsection{Sheafification preserves finite transfers}

Let $\ul{a}^{\fin}_{s \Nis} : \ulMPS^{\fin}\allowbreak \to \ulMNS^{\fin}$
be the sheafification functor,
that is, the left adjoint of the inclusion functor
$\ul{i}^{\fin}_{s \Nis} : \ulMNS^{\fin}  \hookrightarrow \ulMPS^{\fin}$. It exists for general reasons and is exact \cite[expos\'e~II, th\'eor\`eme~3.4]{SGA4}.

\begin{definition}\label{def:mpst-fin}
Let $\ulMNST^\fin$ be the
full subcategory of $\ulMPST^\fin$ consisting of all objects
$F \in \ulMPST^\fin$ such that 
$\ul{c}^{\fin *} F \in \ulMNS^\fin$
(see Proposition \ref{eq:c-functor} for $\ul{c}^{\fin *}$).
\end{definition}

\begin{lemma}\label{lcom2} The category $\ulMNST^{\fin}$ is closed under infinite direct sums in $\ulMPST^{\fin}$, and the inclusion functor 
$\ul{i}_{\Nis}^{\fin} : \ulMNST^{\fin} \to \ulMPST^{\fin}$  is strongly additive
$($Definition \ref{dA.2}\,$)$. 
The objects $\Z_\tr^\fin(M)$ and $b^*\Z_\tr(M)$ belong to  $\ulMNST^{\fin}$ for any $M\in \ulMCor$.
\end{lemma}

\begin{proof} This follows from Lemma \ref{lcom1}, 
because $\ul{c}^{\fin *}$ is strongly additive as a left adjoint. The last claim follows from Proposition \ref{prop:rep-sheaf}.
\end{proof}

We write $\ul{c}^{\fin \Nis} : \ulMNST \to \ulMNS$
for the functor induced by $\ul{c}^{\fin *}$.
By definition, we have
\begin{equation}\label{eq:c-i-fin}
\ul{c}^{\fin *} \ul{i}_{\Nis}^{\fin}
=\ul{i}_{s  \Nis}^{\fin} \ul{c}^{\fin \Nis}.
\end{equation}

\begin{thm}\label{thm:sheaf-transfer}
The following assertions hold.
\begin{enumerate}
\item
Let $F \in \ulMPST^{\fin}$.
There exists a unique object $F_\Nis \in \ulMPST^{\fin}$
such that 
$\ul{c}^{\fin *}(F_\Nis)
=\ul{a}^{\fin}_{s \Nis}(\ul{c}^{\fin *}(F))$ 
and such that
the canonical morphism 
$u: \ul{c}^{\fin *}(F) \to \ul{a}^{\fin}_{s \Nis}(\ul{c}^{\fin *}(F))
=\ul{c}^{\fin *}(F_\Nis)$
extends to a morphism in $\ulMPST^{\fin}$.
\item 
The functor $\ul{i}_{ \Nis}^{\fin}$ has an exact left adjoint 
$\ul{a}_{ \Nis}^{\fin} : \ulMPST^{\fin} \to \ulMNST^{\fin}$
satisfying
\begin{equation}\label{eq:c-a-fin}
\ul{c}^{\fin \Nis} \ul{a}_{ \Nis}^{\fin}
= \ul{a}_{s \Nis}^{\fin} \ul{c}^{\fin *}.
\end{equation}
In particular the category $\ulMNST^{\fin}$ 
is Grothendieck (\S \ref{s.groth}).
\item
The functor $\ul{c}^{\fin \Nis}$ 
has 
a left adjoint $\ul{c}^{\fin}_\Nis=\ul{a}^\fin_{\Nis}\ul{c}^{\fin}_! \ul{i}_{s  \Nis}^\fin$.
Moreover, $\ul{c}^{\fin \Nis}$ 
is exact, strongly additive $($Definition~\ref{dA.2}\,$)$, and faithful.
\end{enumerate}
\end{thm}
\begin{proof}
This can be shown by a rather trivial modification
of \cite[Theorem~3.1.4]{voetri},
but for the sake of completeness we include a proof.
To ease the notation, 
put $F':={\ul{a}_{s  \Nis}^\fin}\ul{c}^{\fin *}F \in \ulMPS^\fin$.
First we construct a homomorphism 
\[ 
\Phi_M: 
F'(M) \to \ulMPS^{\fin}(\ul{c}^{\fin *}\Z_{\tr}^\fin(M), F')
\]
for any $M \in \ulMP$.
Take $f \in F'(M)$.
There exists a Nisnevich cover $p : U \to M$ in $\ulMP^{\fin}$
and $g \in \ul{c}^{\fin *}F(U)=F(U)$ such that
$f|_U = u(g)$ in $F'(U)$.
There also exists a Nisnevich cover $W \to U \times_M U$
such that $g|_W=0$ in $F(W)$.
We have
$\ul{a}_{s  \Nis}^\fin \ul{c}^{\fin *} \Z_{\tr}^\fin(M)
=\ul{c}^{\fin *} \Z_{\tr}^\fin(M)$
because $\ul{c}^{\fin *}\Z_{\tr}^\fin(M) \in \ulMPS^{\fin}_\Nis$ by Proposition~\ref{prop:rep-sheaf}.
Thus we get a commutative diagram
in which the horizontal maps are induced by
$\ul{a}^{\fin}_{s \Nis}\ul{c}^{\fin *}$
\[
\xymatrix{
0  \ar[d]
\\
\ulMPS^\fin(\ul{c}^{\fin *}\Z_{\tr}^\fin(M), F')
\ar[d]_{s}
\\
\ulMPS^\fin(\ul{c}^{\fin *}\Z_{\tr}^\fin(U), F')
\ar[d]
&
\ulMPST^\fin(\Z_{\tr}^\fin(U), F)
\ar[l]_{s'} \ar[d] \ar@/^17ex/[dd]^{s''}
\\
\ulMPS^\fin(\ul{c}^{\fin *}\Z_{\tr}^\fin(U \times_M U), F')
\ar@{^{(}->}[d]_{l}
&
\ulMPST^\fin(\Z_{\tr}^\fin(U \times_M U), F)
\ar[l] \ar[d] 
\\
\ulMPS^\fin(\ul{c}^{\fin *}\Z_{\tr}^\fin(W), F')
&
\ulMPST^\fin(\Z_{\tr}^\fin(W), F).
\ar[l]
}
\]
Since $F'$ is a sheaf,
Theorem \ref{thm:cech} implies that
the left vertical column is exact except at the last spot,
and that the map $l$ is injective.
Since $g \in F(U)=
\ulMPST^\fin(\Z_{\tr}^\fin(U), F)$
satisfies $s''(g)=g|_W=0$,
there exists a unique 
$h \in 
\ulMPS^\fin(c^*\Z_{\tr}^\fin(M), F')$
such that $s(h)=s'(g)$.
One checks that $h$ does not depend on
the choices of $p:U \to M$, $g \in F(U)$ and $W \to U \times_M U$
by taking a refinement of covers.
We define $\Phi_M(f):=h$.

Now we define $G$.
On objects we put $G(M) = F'(M)$ for $M \in \ulMP$.
For $\alpha \in \ulMCor^{\fin}(M, N)$,
we define $\alpha^* : F'(N) \to  F'(M)$
as the composition of
$$
F'(N) \overset{\Phi_N}{\longrightarrow}
\ulMPS^\fin(\ul{c}^{\fin *}\Z_{\tr}^\fin(N), F')
\longrightarrow
\ulMPS^\fin(\ul{c}^{\fin *}\Z_{\tr}^\fin(M), F')
\longrightarrow F'(M),
$$
where the middle map is induced by 
$\ul{c}^{\fin *}(\alpha) : 
\ul{c}^{\fin *}\Z_\tr^\fin(M) \to \ul{c}^{\fin *}\Z_\tr^\fin(N)$,
and the last map is given by
$f \mapsto f_M(\id_M)$.
One checks that, with this definition,
$G$ becomes an object of $\ulMPST^\fin$.

To prove uniqueness, take $G, G' \in \ulMPST^\fin$
which enjoy the stated properties.
For any $M \in \ulMP$ we have $G(M)=G'(M)=F'(M)$.
We also have $G(\ul{c}^{\fin *}(q))=G'(\ul{c}^{\fin *}(q))=F'(q)$
 for any morphism $q$ in $\ulMP^\fin$.
Let $\alpha : M \to N$ be a morphism in $\ulMCor^\fin$
and let $f \in F'(N)$.
Take a Nisnevich cover $p : U \to N$ of $\ulMP^\fin$
and $g \in \ul{c}^{\fin *}F(U)=F(U)$ such that $f|_U = u(g)$ in $F'(U)$.
Apply Lemma \ref{lem:refine-cover} to
get a morphism $\alpha' : V \to U$ in $\ulMCor^\fin$
and a Nisnevich cover $p' : V \to M$ of $\ulMP^\fin$
such that $\alpha p'=p \alpha'$.
Then we have
\begin{align*}
G(p')G(\alpha)(f) &=
G(\alpha')G(p)(f) =
G(\alpha')(u(g)) =
 u(F(\alpha')(g))
\\
&=
G'(\alpha')(u(g)) =
G'(\alpha')G'(p)(f) =
G'(p')G'(\alpha)(f) =
G(p')G'(\alpha)(f).
\end{align*}
Since $p': V \to M$ is a Nisnevich cover 
and $G$ is separated,
this implies
$G(\alpha)(f) =G'(\alpha)(f)$.
This completes the proof or (1).

(2) is a consequence of (1)
and the fact that $\ulMPST^\fin$ is Grothendieck 
as a category of modules (see Theorem \ref{t.groth} d)).  
Then (3) follows from Lemma \ref{lem:lr-adjoint}.
\end{proof}

\begin{rk} A different argument may be given by mimicking the proof of \cite[Corollary~2.2.26]{ayoubrig}.
\end{rk}

\begin{defn} An additive functor
$\phi:\mathcal{C}\to\mathcal{C}'$ between abelian categories is \emph{faithfully exact}
if a complex $F'\to F\to F''$ is exact 
if and only if $\phi F'\to \phi F\to \phi F''$ is.
\end{defn}

This happens if $\phi$ is exact and either faithful or conservative.
By Theorems \ref{thm:sheaf-transfer} and \ref{thm:cech},
we get:

\begin{cor}\label{rem:faith-exact-c}
The functor $\ul{c}^{\fin \Nis}$ is faithfully exact.
In particular, 
if $p : U \to M$ is a Nisnevich cover in $\ulMP^{\fin}$,
then the \v{C}ech complex
\begin{equation}\label{eq:ceck-without-c}
 \dots \to \Z_{\tr}^\fin(U \times_M U)
 \to \Z_{\tr}^\fin(U)
 \to \Z_{\tr}^\fin(M)
 \to 0
\end{equation}
is exact in $\ulMNST^{\fin}$.
\end{cor}

\subsection{Cohomology in $\protect\ulMNST^\protect\fin$}

\begin{nota}\label{n3.7a} 
Let $M \in \ulMP^\fin$ and 
let $F \in \ulMNS^\fin$ (resp. $F \in \ulMNST^\fin$).
We write $F_M$ for the sheaf on $(\ol{M})_\Nis$ induced from $F$ 
(resp. $\ul{c}^{\fin \sigma} F$)
via the isomorphism of sites from Lemma \ref{lem:equiv-smallsites}. 
(Note that $F_M$ depends not only on  $\ol{M}$, but also on $M^\infty$.)
We thus have canonical isomorphisms
\begin{align}\label{eq3.62}
&H^i_\Nis(M, F)\simeq H^i_\Nis(\ol{M}, F_M),
\\
&
\label{eq3.6}
H^i_\Nis(M, \ul{c}^{\fin \Nis} F)\simeq H^i_\Nis(\ol{M}, F_M),
\end{align}
where the right hand sides denote the cohomology of 
the (usual) small site $(\ol{M})_\Nis$.
\end{nota}

\begin{definition}\
\begin{enumerate}
\item
Let $S$ be a scheme.
We say a sheaf $F$ on $S_\Nis$ is \emph{flasque} if
$F(V) \to F(U)$ is surjective 
for any open dense immersion $U \to V$.
Flasque sheaves are flabby in the sense of 
Definition \ref{def:flabby}
(see \cite[lemme 1.40]{riou}).
\item
We say $F \in \ulMNS^\fin$ is flasque
if $F_M$ is flasque for any $M \in \ulMP^\fin$  (see Notation \ref{n3.7a}).
Again, flasque sheaves are flabby by \eqref{eq3.62}.
\end{enumerate}
\end{definition}

\begin{lemma}\label{lif1} 
Let  $I\in \ulMNST^\fin$ be an injective object. 
Then $\ul{c}^{\fin \Nis}(I)\allowbreak\in \ulMNS^\fin$ is flasque,
and hence flabby.
\end{lemma}

\begin{proof}
Let $j:U\inj M$ be a minimal open immersion of modulus pairs in $\ulMP^\fin$. The morphism of sheaves $\Z_\tr^\fin(j)$ is a monomorphism, hence $j^*:I(M)\to I(U)$ is surjective.
Alternatively, one can apply
Lemma \ref{lem:inj-flabby} with \eqref{eq:ceck-without-c}
to show that $\ul{c}^{\fin \sigma}(I)$ is flabby.
(This proof also works for the \'etale topology.)
\end{proof}

\section{Sheaves on {$\protect\ulMP$ and $\protect\ulMCor$} }

\subsection{A cd-structure on $\protect\ulMSm$}\label{sec:sd-str-ulMSm}  Let $P_{\ul{\MV}}$ be the collection of commutative squares in $\ulMSm$ which are isomorphic in $\ulMSm^\Sq$ to 
$\ul{b}_s^\Sq(Q)$ for some $\ulMVfin$-square $Q$ 
in Definition \ref{d3.2}.  Then $P_{\ul{\MV}}$ defines a cd-structure on $\ulMSm$ (see  \S \ref{s3.2}).

\begin{defn}\label{d4.1}
The squares which belong to $P_{\ul{\MV}}$ 
are called \emph{$\ulMV$-squares}. 
\end{defn}

\begin{thm}\label{thm;cd-str-ulMSm}
The cd-structure $P_{\ul{\MV}}$ is strongly complete and strong\-ly regular, in particular complete and regular (see Definitions \ref{d3.1} and \ref{dA.5}).
\end{thm}
\def\Hka{H_\kappa}
\def\HNis{H_\Nis}

\begin{proof}
This follows from Propositions \ref{p3.1} and \ref{pA.4}.
\end{proof}

\subsection{Sheaves on $\protect\ulMP$}

\begin{definition}\label{def;sheavesulMSmul}
Consider 
the Grothendieck topology on $\ulMSm$ 
generated by the squares in $P_{\ul{\MV}}$. 
The resulting site will be denoted by $\ulMSm_\Nis$.
We write $\ulMNS$ for  
the full subcategory of sheaves in $\ulMPS$. 
We denote by 
$\ul{i}_{s  \Nis}: \ulMNS \to \ulMPS$
the inclusion functor.
\end{definition}

By the general properties of Grothendieck topologies \cite[expos\'e2]{SGA4}, we have:

\begin{thm}\label{thm:sheafification-ulMNS}
The inclusion functor
$\ul{i}_{s  \Nis}: \ulMNS \to \ulMPS$
has an exact left adjoint 
$\ul{a}_{s  \Nis}$. 
The category $\ulMNS$ is Grothendieck  $($\S \ref{s.groth}\,$)$.\qed
\end{thm}

\begin{lemma}\label{lem:sheaves-MNS}
The following conditions are equivalent 
for $F \in \ulMPS$.
\begin{thlist}
\item 
$F\in \ulMNS$.
\item 
$\ul{b}_s^*F \in \ulMNS^\fin$; in other words, 
$(\ul{b}_s^*F)_M$ is a Nisnevich 
sheaf for any $M \in \ulMSm$ 
$($see \eqref{eq:six-cat-diag0} for $\ul{b}_s$ and Notation \ref{n3.7a} for $(-)_M)$.
\item
$F$ transforms 
any $\ulMVfin$-square
\begin{equation}\label{Q0} Q_0 : \vcenter{\xymatrix{
W_0 \ar[r]^{} \ar[d]_{} & V_0 \ar[d]^{} \\
U_0 \ar[r]^{} & M
}}\end{equation}
into an exact sequence
\[ 0\ \to F(M) \to F(U_0) \times F(V_0) \to F(W_0).\]
\end{thlist}
\end{lemma}
\begin{proof}
In view of Theorem \ref{thm;cd-str-ulMSm} and \cite[Corollary~2.17]{cdstructures}, we have (i) $\iff$ (iii). On the other hand, (ii) $\iff$ (iii) by adjunction and Proposition \ref{p3.1}.   
\end{proof}

\begin{cor}\label{lcom3} 
The category $\ulMNS$ is closed under infinite direct sums in $\ulMPS$
and $\ul{i}_{s  \Nis}$ is strongly additive $($Definition~\ref{dA.2}\,$)$.
\end{cor}
\begin{proof} This follows from Lemmas \ref{lcom1}, \ref{lem:sheaves-MNS} ((i) $\iff$ (ii)) and \ref{lem:lr-adjoint} (2)
because $\ul{b}_s^*$ is strongly additive as a left adjoint.
\end{proof}

\subsection{The adjunction $(\protect\ul{b}_{ s, \protect\Nis},\protect\ul{b}_s^\protect\Nis)$}

\begin{definition}
A map in $\ulMSm_\Nis$ 
is called a \emph{strict Nisnevich cover}
if it is the image of a cover of $\ulMSm^\fin_\Nis$
by $\ul{b}_s:\ulMSm^\fin\to \ulMSm$.
\end{definition}

By definition,
a strict Nisnevich cover is 
evidently a cover in $\ulMSm_\Nis$.
Up to isomorphism,
any cover of $\ulMSm_\Nis$ 
can be refined to such a cover.
More precisely, we have the following lemma.

\begin{lemma}\label{lem:cov-ulMSm}
Any cover  $U \to M$ in $\ulMSm_\Nis$
admits a refinement of the form
$V \to N \to M$,
where $V \to N$ is a strict Nisnevich cover
and $N \to M$ is a morphism in $\ul{\Sigma}^\fin$
$($see Definition \ref{deff}\,$)$.
\end{lemma}
\begin{proof}
By Definition \ref{def;sheavesulMSmul} and Proposition \ref{peff1},
there is a refinement of $U \to M$ of the form
\[ U_n \overset{f_n}{\to} U_{n-1} \overset{f_{n-1}}{\to} 
\cdots \overset{f_1}{\to} U_0=M,
\]
where for each $i$ we have either
(i) $f_i \in \ul{\Sigma}^\fin$,
(ii) $f_i=g^{-1}$ for some $g \in \ul{\Sigma}^\fin$,
or 
(iii) $f_i$ is a strict Nisnevich cover.
We proceed by induction on $n$,
the case $n=0$ being trivial.
Suppose $n>0$.
By induction, we have a refinement of $U_n \to U_1$
of the form $V' \to N' \to U_1$
where $V' \to N'$ is a strict Nisnevich cover
and $N' \to U_1$ is in $\ul{\Sigma}^\fin$.

If $f_1 \in \ul{\Sigma}^\fin$,
then we can take $V=V'$ and $N=N'$,
as the composition 
$N' \to U_1 \to U_0$ belongs to $\ul{\Sigma}^\fin$.
Next, suppose $f_1=g^{-1}$ with $g \in \ul{\Sigma}^\fin$.
Then we can take 
$V=V' \times_{U_1} U_0$ and
$N=N' \times_{U_1} U_0$,
where $U_0$ is regarded as a $U_1$-scheme by $g$.
Finally, suppose $f_1$ is a strict Nisnevich cover.
By Lemma \ref{mainlem;blowup},
we may find a morphism $N \to U_0$ in $\ul{\Sigma}^\fin$
such that $N'':=N \times_{U_0} U_1 \to U_1$
factors through $N'$.
Then we can take $V=V' \times_{N'} N''$.
This completes the proof.
\end{proof}

We define
$\ul{b}_s^{\Nis}: \ulMNS \to \ulMNS^\fin$
to be 
the restriction of $\ul{b}_s^*$, cf. Lemma \ref{lem:sheaves-MNS} (ii).
By definition, 
we have 
\begin{equation}\label{eq:b-and-i}
\ul{b}_s^* \ul{i}_{s  \Nis} = \ul{i}_{s  \Nis}^\fin \ul{b}_s^\Nis.
\end{equation}

\begin{prop}\label{lem;b!ulMNS}
The following assertions hold.
\begin{enumerate}
\item 
We have $\ul{b}_{s  !}(\ulMNS^\fin) \subset \ulMNS$.
In particular,
$\ul{b}_{s  !}$ restricts to 
$\ul{b}_{s  \Nis} : \ulMNS^\fin \to \ulMNS$
so that we have 
\begin{equation}\label{eq:b-and-i2}
\ul{b}_{s  !} \ul{i}_{s  \Nis}^\fin = \ul{i}_{s  \Nis} \ul{b}_{s  \Nis}.
\end{equation}
\item
The functor
$\ul{b}_{s  \Nis}$ is an exact left adjoint of $\ul{b}_s^\Nis$.
The functor
$\ul{b}_s^\Nis$ is fully faithful and preserves injectives.
The counit map
$\ul{b}_{s  \Nis} \ul{b}_s^\Nis \to \id$ is an isomorphism and $\ul{b}_{s  \Nis} R^q\ul{b}_s^\Nis=0$ for $q>0$.
\end{enumerate}
\end{prop}

\begin{proof}
Let $F \in \ulMNS^\fin$ and 
take $M \in \ulMSm$.
We shall show that $(\ul{b}_s^*\ul{b}_{s !}F)_M$ is a Nisnevich sheaf on $\ol{M}$.
For a given $\ulMVfin$-square in $\ulMSm^\fin$
\[\xymatrix{
W \ar[r]\ar[d] & V \ar[d]\\
U\ar[r] & M\\}\] 
its pullback via $(N\to M) \in \ul{\Sigma}^{\fin}\downarrow M$ 
(which exists by Corollary \ref{exist-pullback} (1))
\[\xymatrix{
W\times_M N  \ar[r]\ar[d] & V\times_M N \ar[d]\\
U\times_M N\ar[r] & N\\}\]
is also an $\ulMVfin$-square. 
By Proposition \ref{p3.1} (2) and by \cite[Corollary~2.17]{cdstructures}, the sequence
\[ 0\to F(N) \to F(U\times_M N) \oplus F(V\times_M N) \to F(W\times_M N)\]
is exact. By Lemma \ref{mainlem;blowup}, the pullback of 
$\ul{\Sigma}^{\fin}\downarrow M$ via $U \to M$ is cofinal in 
$\ul{\Sigma}^{\fin}\downarrow U$, and similarly for $V\to M$ and $W\to M$.
Hence, by taking its colimit over $N\in \ul{\Sigma}^{\fin}\downarrow M$, 
the above exact sequences 
and \eqref{eq:b-sh-explicit}
imply the desired exact sequence 
\[ 0\to b_{s !} F(M) \to b_{s !} F(U) \oplus b_{s !} F(V) \to b_{s !} F(W).\]
In view of Lemma \ref{lem:sheaves-MNS}, this finishes the proof of (1).

(2) The adjunction $(\ul{b}_{s \Nis} , \ul{b}_{s}^{\Nis})$ follows from the adjunction $(\ul{b}_{s !} , \ul{b}_s^\ast)$ (see Proposition \ref{eq:bruno-functor}), by the full faithfullness of $\ul{i}_{s  \Nis}$ and $\ul{i}_{s  \Nis}^\fin$, and by the formulas \eqref{eq:b-and-i} and \eqref{eq:b-and-i2}.
The full faithfulness of $\ul{b}_s^\Nis$ follows from that of $\ul{b}_s^\ast$ (see Proposition \ref{eq:bruno-functor}), $\ul{i}_{s  \Nis}$ and $\ul{i}_{s  \Nis}^\fin$.
Then the counit map $\ul{b}_{s  \Nis} \ul{b}_s^\Nis \to \id$ is an isomorphism by Lemma \ref{lA.6}.

We prove the exactness of $\ul{b}_{s \Nis}$ as follows. 
Since it is right exact as a left adjoint, it suffices to show its left exactness.

Assume given an exact sequence in $\ulMNS^\fin$:
\[ 0\to F\to G\to H \to 0.\]
Applying the left exact functor $\ul{i}_{s  \Nis}^\fin :\ulMNS^\fin\to \ulMPS^\fin$ and the exact functor $\ul{b}_{s !}: \ulMPS^\fin \to \ulMPS$
and using \eqref{eq:b-and-i2},
we get an exact sequence
\[ 0\to \ul{i}_{s  \Nis} \ul{b}_{s  \Nis} F\to \ul{i}_{s  \Nis} \ul{b}_{s  \Nis} G\to \ul{i}_{s  \Nis} \ul{b}_{s  \Nis} H.
\]
For every $Q\in \ulMNS$, this gives rise to an exact sequence
\begin{align*} 
0&\to \Hom_{\ulMPS}(\ul{i}_{s  \Nis} Q, \ul{i}_{s  \Nis} \ul{b}_{s  \Nis} F)
\to \Hom_{\ulMPS}(\ul{i}_{s  \Nis} Q, \ul{i}_{s  \Nis} \ul{b}_{s  \Nis} G)
\\
&\to 
\Hom_{\ulMPS}(\ul{i}_{s  \Nis} Q, \ul{i}_{s  \Nis} \ul{b}_{s  \Nis} H ).
\end{align*}
Since $\ul{i}_{s \Nis}$ is fully faithful, this gives an exact sequence
\[ 0\to \Hom_{\ulMNS}(Q, \ul{b}_{s  \Nis}  F)
\to \Hom_{\ulMNS}(Q, \ul{b}_{s  \Nis}  G)\to 
\Hom_{\ulMNS}(Q, \ul{b}_{s  \Nis} H),\]
which shows the exactness of 
\[ 0\to \ul{b}_{s  \Nis} F \to \ul{b}_{s  \Nis} G \to \ul{b}_{s  \Nis} H,\]
as desired. 
Therefore, $\ul{b}_{s \Nis}$ is exact. 

Then $\ul{b}_s^{\Nis}$ preserves injectives since it has an exact left adjoint $\ul{b}_{s \Nis}$.
Moreover, applying $R^q$ ($q>0$) to the counit isomorphism $\ul{b}_{s  \Nis} \ul{b}_s^\Nis \to \id$, we have  
\[
\ul{b}_{s  \Nis} R^q \ul{b}_s^\Nis  \simeq R^q (\ul{b}_{s  \Nis} \ul{b}_s^\Nis ) \simeq R^q \id \simeq 0,
\]
by Example \ref{exA.3} and the exactness of $\ul{b}_{s  \Nis}$.
This concludes the proof.
\end{proof}

\begin{cor} We have a natural isomorphism
$\ul{a}_{s  \Nis}\simeq \ul{b}_{s  \Nis} \ul{a}_{s  \Nis}^\fin \ul{b}_s^*$.
\end{cor}

\begin{proof}
By the uniqueness of left adjoints, it suffices to check that the right hand side is also left adjoint to $\ul{i}_{s  \Nis}$.
We first apply double adjunction by
$(\ul{b}_{s  \Nis}, \ul{b}_s^\Nis)$ (Proposition \ref{lem;b!ulMNS} (2))
and $(\ul{a}_{s  \Nis}^\fin, \ul{i}_{s  \Nis}^\fin)$, 
then use \eqref{eq:b-and-i}
and the full faithfulness of $\ul{b}_s^*$ (Proposition \ref{eq:bruno-functor}).
\end{proof}

\subsection{Cohomology in $\protect\ulMNS$}

\def\Zp{\Z^p}
\def\Zpfin{\Z^{p,\fin}}
\def\Zfin{\Z^\fin}

\begin{nota}\label{n:pre}
\begin{enumerate}
\item 
Let $M\in \ulMSm$ and $F\in \ulMNS$.
Using Notation \ref{n3.7a},
we define
$F_M :=(\ul{b}_s^\Nis F)_M$ 
which is a sheaf on $(\ol{M})_\Nis$.
\item 
For $M\in \ulMSm$, 
let $\Zp(M)\in \ulMPS$ be the associated representable additive presheaf 
(see \eqref{eq2.6}) and let
\begin{equation}\label{Z}
 \Z(M) =\ul{a}_{s \Nis} \Zp(M) \in \ulMNS
\end{equation}
be the associated sheaf. 
\end{enumerate}
\end{nota}

\begin{prop}\label{lem;cohMsigmaS}
For $M\in \ulMSm$, $F\in \ulMNS$ and $i\ge 0$, 
we have a natural isomorphism
\begin{equation}\label{eq:coh-ulMNS}
\Ext_{\ulMNS}^i(\Z(M), F)
\simeq
\colim_{N\in \ul{\Sigma}^{\fin}\downarrow M} H_\Nis^i(\ol{N},F_N)
:=\colim_{N\in \ul{\Sigma}^{\fin}\downarrow M} H_\Nis^i(\ol{N},(\ul{b}_s^\Nis F)_N).
\end{equation}
Moreover, we have
\begin{equation}\label{eq:coh-ulMNS1}
\colim_{N\in \ul{\Sigma}^{\fin}\downarrow M} H_\Nis^i(\ol{N},(R^q\ul{b}_s^\Nis F)_N)=0\text{ for all $q>0$.}
\end{equation}
\end{prop}

\begin{proof}
Define functors
$\Gamma_M^\downarrow : \ulMNS^\fin \to \Ab$
and
$\ul{\Gamma}_M : \ulMNS \to \Ab$
by
\[ \Gamma_M^\downarrow(G)
=\colim_{N \in \ul{\Sigma}^\fin \downarrow M} G(N),
\quad
\ul{\Gamma}_M(F)=F(M).
\]
We have $\Gamma_M^\downarrow=\ul{\Gamma}_M \ul{b}_{s  \Nis}$.
By Theorem \ref{tA.2} and Lemma \ref{lem:claim} below,
we get $(R^p \ul{\Gamma}_M)\ul{b}_{s  \Nis}=R^p \Gamma_M^\downarrow$
for any $p \ge 0$
since $\ul{b}_{s  \Nis}$ is exact.
Thus, by Lemma \ref{lem:claim2} below we obtain 
\[
\Ext_{\ulMNS}^p(\Z(M), \ul{b}_{s  \Nis} G)
\cong \colim_{N \in \ul{\Sigma}^\fin \downarrow M}
H^p_\Nis(\ol{N}, G_N)
\]
for any $G \in \ulMNS^\fin$ and $p \ge 0$.
Setting $G=R^q\ul{b}_s^\Nis F$,
we get \eqref{eq:coh-ulMNS} for $q=0$ (resp.  \eqref{eq:coh-ulMNS1} for $q>0$) thanks to 
Proposition \ref{lem;b!ulMNS} (2).
\end{proof}

\begin{lemma}\label{lem:claim}
For an injective $I\in \ulMNSfin$, 
$\ul{b}_{s  \Nis} I\in \ulMNS$ is flabby
$($see Definition \ref{def:flabby}\,$)$.
\end{lemma}
\def\chH#1{\check{H}^{#1}}
\begin{proof}
Write $F=\ulb_\Nis I$. 
By Lemma \ref{lem:milne}, it suffices to show 
the vanishing of the canonical map
$\check{H}^q(U/M, F) \to \check{H}^q(M, F)$
for any cover $U \to M$ in $\ulMSm_\Nis$ and any $q>0$.
By Lemma \ref{lem:cov-ulMSm},
we may assume $U \to M$ is a strict Nisnevich cover
(as any morphism in $\ulSigma^\fin$
is an isomorphism in $\ulMSm$).
Denote by $U^n_M \in \ulMSm$ the $n$-fold fiber product of $U$ over $M$ in $\ulMSm$
(which exists by Corollary \ref{exist-pullback} (1)).
Then $\check{H}^q(U/M, F)$ is computed as the cohomology of
the complex whose term in degree $q$ is given by
\[ \colim_{L_q\in \ul{\Sigma}^{\fin}\downarrow U_M^{q+1}} I(L_q).\]
By Lemma \ref{mainlem;blowup}, 
for any integer $n>0$ and given $L_q\in \ulSigma^{\fin}\downarrow U_M^{q+1}$ for $0\leq q\leq n$, there exists 
$L\in \ulSigma^{\fin}\downarrow M$ in such that $L\times_M U_M^{q+1} \to U_M^{q+1}$
factor through $L_q$ for all $q=0, \dots, n$.
This implies that for $0\leq q\leq n-1$
the canonical map 
$\chH q(U/M, \ul{b}_{s  \Nis} I) \to  \chH q(M, \ul{b}_{s  \Nis} I)$
factors through 
\[ \colim_{L\in \ulSigma^{\fin}\downarrow M} \chH q (U\times_M L/L, I),
\]
where $\chH q (U\times_M L/L, I)$ is
the \v{C}ech cohomology of $I$
with respect to the cover $U\times_M L \to L$ in $\ulMSm^\fin$,
but it vanishes since $I$ is injective in $\ulMNS^{\fin}$.
This proves the desired vanishing and completes the proof of Lemma \ref{lem:claim}.
\end{proof}

\begin{lemma}\label{lem:claim2}
For any $G \in \ulMNS^\fin$ and $p \ge 0$,
we have
\[ R^p\Gamma_M^\downarrow(G)\cong
\colim_{N \in \ul{\Sigma}^\fin \downarrow M}
H^p_\Nis(\ol{N}, G_N).
\]
\end{lemma}
\begin{proof}
Take an injective resolution $G \to I^\bullet$ in $\ulMNSfin$.
Then we have
$$
R^p\Gamma_M^\downarrow(G)
= H^p(\Gamma_M^\downarrow I^\bullet)
= H^p(\colim_{N \in \ul{\Sigma}^\fin \downarrow M} I^\bullet)
\cong \colim_{N \in \ul{\Sigma}^\fin \downarrow M}H^p(I^\bullet(N))
\cong \colim_{N \in \ul{\Sigma}^\fin \downarrow M}H^p_\Nis(\ol{N}, G_N),
$$
where 
we used Corollary \ref{cor:sigma-fin-cofil}
for the last-but-one isomorphism,
and \eqref{eq3.62} for the last one.
\end{proof}

\subsection{Sheaves on $\protect\ulMCor$}

\begin{lemma}\label{lem:mnst-condition}
For $F \in \ulMPST$,
one has $\ul{c}^*F \in \ulMNS$ if and only if $b^*F \in \ulMNST^\fin$.
\end{lemma}
\begin{proof}
This follows from \eqref{eq:b-and-c}
and Definitions \ref{def:mpst-fin} and \ref{def;sheavesulMSmul}.
\end{proof}

\begin{definition}\label{def;sheavesulMCor}
We define $\ulMNST$ to be 
the full subcategory of $\ulMPST$ consisting of
those $F$ enjoying the conditions of Lemma \ref{lem:mnst-condition}.
We denote by 
$\ul{i}_{\Nis}: \ulMNST \to \ulMPST$
the inclusion functor.
\end{definition}

\begin{lemma}\label{lcom3-2} 
The category $\ulMNST$ 
is closed under infinite direct sums in  $\ulMPST$,
and $\ul{i}_\Nis$ is strongly additive $($Definition \ref{dA.2}\,$)$. 
It contains $\Z_\tr(M)$ for any $M\in \ulMCor$.
\end{lemma}

\begin{proof} This follows from Lemma \ref{lcom2}, 
because $\ul{b}^*$ is strongly additive as a left adjoint. The last statement follows from Lemma \ref{lcom2}.
\end{proof}

By Definition \ref{def;sheavesulMCor} and Lemma \ref{lem:mnst-condition},
the functors $\ul{b}^*$
and $\ul{c}^*$ restrict to
$\ul{b}^{\Nis}: \ulMNST \to \ulMNST^\fin$
and
$\ul{c}^\Nis : \ulMNST \to \ulMNS$.
It holds that
\begin{align}\label{eq:b-and-i4}
&\ul{b}^* \ul{i}_\Nis = \ul{i}_\Nis^\fin \ul{b}^\Nis,
\quad
\ul{c}^* \ul{i}_\Nis = \ul{i}_{s  \Nis} \ul{c}^\Nis,
\\
\notag
&\ul{b}_s^\Nis \ul{c}^{\Nis} 
= \ul{c}^{\fin \Nis} \ul{b}^\Nis,
\quad
\ul{b}_{s  \Nis} \ul{c}^{\fin \Nis}
=\ul{c}^\Nis \ul{b}_\Nis,
\end{align}
where for the last two formulas we used \eqref{eq:b-and-c}.

\begin{prop}\label{lem;b!ulMNST}
The following assertions hold.
\begin{enumerate}
\item 
We have $\ul{b}_{!}(\ulMNST^\fin) \subset  \ulMNST$.
\item 
Let
$\ul{b}_{\Nis} : \ulMNST^\fin \to \ulMNST$
be the restriction of $\ul{b}_{!}$ so that we have 
\begin{equation}\label{eq:b-and-i3}
\ul{b}_{!} \ul{i}_{\Nis}^\fin = \ul{i}_{\Nis} \ul{b}_{\Nis}.
\end{equation}
Then, the functor
$\ul{b}_{\Nis}$ is an exact left adjoint of $\ul{b}^\Nis$, which is fully faithful.
\item 
The functor 
$\ul{b}^\Nis$ preserves injectives.
\end{enumerate}
\end{prop}

\begin{proof}
(1) It suffices to show that $\ul{c}^*\ul{b}_{!}(\ulMNST^\fin) \subset  \ulMNS$. By \eqref{eq:b-and-c}, we have $\ul{c}^*\ul{b}_{!}=\ul{b}_{s!}\ul{c}^{\fin *}$. Moreover, $\ul{c}^{\fin *}\ulMNST^\fin\subset \ulMNS^\fin$ by Definition \ref{def:mpst-fin} and  $\ul{b}_{s!} \ulMNS^\fin\subset \ulMNS$
by Proposition \ref{lem;b!ulMNS} (1).
In (2), the adjointness and the full faithfulness are seen by using 
Proposition \ref{eq:bruno-functor}, \eqref{eq:b-and-i4}
and \eqref{eq:b-and-i3}.
This proves that $\ul{b}_\Nis$ is right exact,
and  it is also exact by \eqref{eq:b-and-i3} and Proposition \ref{eq:bruno-functor} (see also the proof of the exactness of $\ul{b}_{s \Nis}$ in Proposition \ref{lem;b!ulMNS} (2)).
(3) is a consequence of (2).
\end{proof}

\begin{thm}\label{thm:sheafification-ulMNST}
The inclusion functor
$\ul{i}_{\Nis}: \ulMNST \to \ulMPST$
has the exact left adjoint 
$\ul{a}_\Nis=\ul{b}_\Nis \ul{a}_\Nis^\fin \ul{b}^*$.
In particular, $\ulMNST$ is Grothendieck.
\end{thm}
\begin{proof}
The formula defining $\ul{a}_\Nis$ yields a left adjoint to $\ul{i}_\Nis$ by the full faithfulness of $\ul{b}^*$ (Proposition \ref{eq:bruno-functor}) and the adjunctions $(\ul{a}_\Nis^\fin,\ul{i}_\Nis^\fin)$ and $(\ul{b}_\Nis,\ul{b}^\Nis)$ (use \eqref{eq:b-and-i4}). Its exactness follows from the exactness of the three functors.
\end{proof}

\begin{prop}\label{prop:c-Nis}
We have 
\begin{equation}\label{eq:a-c2}
\ul{b}_\Nis \ul{a}_\Nis^\fin = \ul{a}_\Nis \ul{b}_!,
\qquad
\ul{a}_{s  \Nis} \ul{c}^* = \ul{c}^\Nis \ul{a}_\Nis.
\end{equation}
Moreover, $\ul{c}^{\Nis}$ is faithful, exact, 
strongly additive $($Definition \ref{dA.2}\,$)$ and has a left adjoint 
$\ul{c}_\Nis=\ul{a}_{\Nis} \ul{c}_! \ul{i}_{s  \Nis}$ such that $\ul{c}_\Nis \ul{a}_{s \Nis} = \ul{a}_\Nis \ul{c}_!$.
\end{prop}
\begin{proof}
The first equality follows from 
the first formula of \eqref{eq:b-and-i4} by adjunction.
For the second, we use
Theorems \ref{thm:sheafification-ulMNS} 
and \ref{thm:sheafification-ulMNST},
together with
\eqref{eq:b-and-c}, 
\eqref{eq:c-a-fin}
and \eqref{eq:b-and-i4}.
The last statement follows from Lemma \ref{lem:lr-adjoint} (3).
\end{proof}

\begin{thm}\label{thm:cech2}
If $p : U \to M$ is a cover in $\ulMP_\Nis$,
then the \v{C}ech complex
\begin{equation}\label{eq:ceck3-2}
 \dots \to \Z_{\tr}(U \times_M U)
 \to \Z_{\tr}(U)
 \to \Z_{\tr}(M)
 \to 0
\end{equation}
is exact in $\ulMNST$. 
$($Note that the fiber products exist in $\ulMSm$
by Corollary \ref{exist-pullback} (1).$)$
Moreover, the sequence
\[0\to \Z_\tr(W)\to \Z_\tr(U)\oplus \Z_\tr(V)\to \Z_\tr(X)\to 0\]
is exact in $\ulMNST$ 
for any $\ulMVfin$-square \eqref{eq.cd} in $\ulMP^\fin$.
\end{thm}

\begin{proof} 
By Lemma \ref{lem:cov-ulMSm},
we may assume $M \to U$ is a strict Nisnevich cover.
Then, by \eqref{eq:ceck-without-c} the complex
\[
 \dots \to \Z_{\tr}^\fin(U \times_M U)
 \to \Z_{\tr}^\fin(U)
 \to \Z_{\tr}^\fin(M)
 \to 0
\]
is exact in $\ulMNST^{\fin}$.
Applying the exact functor $\ul{b}_\Nis$, 
we get \eqref{eq:ceck3-2}.
The second statement  follows from the first and a small computation 
(\emph{cf.} \cite[Proposition~6.14]{mvw}).
\end{proof}

\subsection{Cohomology in $\protect\ulMNST$}

\begin{lemma}\label{lif1v} 
Let  $I\in \ulMNST$ be an injective object. 
Then $\ul{c}^{\Nis}(I)\in \ulMNS$ is flabby.
\end{lemma}
\begin{proof}
This follows from Lemma \ref{lem:inj-flabby}
and Theorem \ref{thm:cech2}.
\end{proof}

\begin{nota}\label{not4}
Let $M\in \ulMCor$ and $F \in \ulMNST$.
Using Notation \ref{n3.7a},
we define
$F_M :=(\ul{b}^\Nis F)_M$,
which is a sheaf on $(\ol{M})_\Nis$.
\end{nota}

\begin{thm}\label{c3.1v}
Let $F\in \ulMNST$, and let $M\in \ulMCor$. Then there are canonical isomorphisms for any $i\ge 0$:
\[\Ext^i_{\ulMNST}(\Z_\tr(M),F)\simeq 
  \Ext^i_{\ulMNS}(\Z(M), \ul{c}^\Nis F) \simeq 
\colim_{N\in \ulSigma^{\fin}\downarrow M}H^i_\Nis(\Nb,F_N).
\]
$($See \eqref{Z} for $\Z(M)$.$)$ Moreover, we have
\[
\colim_{N\in \ul{\Sigma}^{\fin}\downarrow M} H_\Nis^i(\ol{N},(R^q(\ul{b}_s^\Nis)\ul{c}^\Nis F)_N)=0\text{ for all $q>0$.}
\]
\end{thm}

\begin{proof} Applying the last identity of Proposition \ref{prop:c-Nis} to $\Z^p(M)$, we get
\[\ul{c}_\Nis \Z(M) = \ul{a}_\Nis \ul{c}_! \Z^p(M)
= \ul{a}_\Nis \Z_\tr(M)
= \Z_\tr(M)\]
where the second equality follows from \eqref{eq2.6}, and the third one holds by Lemma \ref{lcom3-2}. This yields an isomorphism
\[
 \ulMNS(\Z(M), \ul{c}^\Nis F)\\
\simeq \ulMNST(\Z_\tr(M),  F)
\]
which is the case $i=0$ of the first isomorphism in the proposition. The general case $i\ge 0$ then follows from
Theorem \ref{tA.2}, Lemma \ref{lif1v} and 
the exactness of $\ul{c}^\Nis$ (Proposition \ref{prop:c-Nis}),
and the second isomorphism follows from
Proposition \ref{lem;cohMsigmaS} 
and \eqref{eq:b-and-i4}. The last assertion follows from \eqref{eq:coh-ulMNS1}. 
\end{proof}

\appendix

\section{Categorical toolbox, I}\label{sect:app}

This appendix gathers known and less-known results that we use constantly.

\subsection{Pro-objects (\cite[expos\'e~I, \S8]{SGA4}, \cite[Appendix~2]{am})}
\label{sec:pro-obj}
Recall that a \emph{pro-object} of  a category $\sC$ is a functor $F:A\to \sC$, where $A$ is a small cofiltered category  (dual of \cite[Chapter~IX, \S1]{mcl}). 
They are denoted by $\{ X_\alpha \}_{\alpha\in A}$ 
or by $``\lim"_{\alpha\in A} X_\alpha$ (Deligne's notation), with $X_\alpha=F(\alpha)$. Pro-objects of $\sC$ form a category $\pro{}\sC$, with morphisms given by the formula
\[\pro{}\sC(\{ X_\alpha \}_{\alpha\in A},\{Y_\beta \}_{\beta\in B})=\lim_{\beta\in B} \colim_{\alpha\in A} \sC(X_\alpha,Y_\beta). \]

There is a canonical full embedding $c:\sC\inj \pro{}\sC$, sending an object to the corresponding constant pro-object ($A=\{*\}$).

For the next lemma, we recall a special case of comma categories from Mac Lane  \cite[Chapter~II, \S6]{mcl}. If $\psi:\sA\to \sB$ is a functor and $b\in\sB$, we write $b\downarrow \psi$ for the category whose objects are pairs $(a,f)\in \sA\times \sB(b,\psi(a))$; a morphism $(a_1,f_1)\to (a_2,f_2)$ is a morphism $g\in \sA(a_1,a_2)$ such that $f_2=\psi(g)f_1$. The category $\psi\downarrow b$ is defined dually (objects: systems $\psi(a)\by{f} b$, etc.)
According to  \cite[Chapter~IX, \S3]{mcl}, $\psi$ is \emph{final} if, for any $b\in \sB$, the category $\psi\downarrow b$ is nonempty and connected; here we shall use the dual property \emph{cofinal} (same conditions for $b\downarrow \psi$). As usual, we abbreviate $Id_\sA\downarrow a$ and $a\downarrow Id_\sA$ by $\sA\downarrow a$ and $a\downarrow \sA$.

Let $F=(F:A\to \sC)=\{ X_\alpha \}_{\alpha \in A} \in \pro{}\sC$. For each $\alpha\in A$, we have a ``projection'' morphism $\pi_\alpha:F\to c(X_\alpha)$ in $\pro{}\sC$. This yields an isomorphism in $\pro{}\sC$
\[F\iso \lim_{\alpha\in A} c(X_\alpha)\]
(explaining Deligne's notation) and a functor
\[\theta:A\to F\downarrow c,
\qquad \theta(\alpha) = (X_\alpha, \pi_\alpha),
\] 
where we take $\sA = \sC$ and $\sB = \pro{}\sC$ in the above setting. 

\begin{lemma}\label{l.cof} The functor $\theta$ is cofinal.
\end{lemma}

\begin{proof} 
Let $F\by{f} c(Y)$ ($Y \in \sC$) be an object of $F\downarrow c$. An object of $\theta\downarrow (F\by{f} c(Y))$ is a pair $(\alpha,\phi)$, with $\alpha\in A$ and $\phi:F(\alpha)\to Y$ such that $f=c(\phi)\pi_\alpha$. 
This category is nonempty because an object $\alpha \in A$ and the morphism $f$ yield the object $f(\alpha) : F(\alpha) \to c(Y)(\alpha) = Y$, and we have $(\alpha, f(\alpha)) \in \theta\downarrow (F\by{f} c(Y))$.
Note also that it is cofiltered, because $A$ is.
Since any cofiltered category is obviously connected, we are done.
\end{proof}

%

(Warning: the use of co in (co)final and (co)filtered is opposite in \cite{mcl} and in \cite{ks}. We use the convention of \cite{mcl}.)

\subsection{Pro-adjoints \cite[expos\'e~I, \S8.11.5]{SGA4}}\label{s1.1} Let $u:\sC\to \sD$ be a functor: it induces a functor $\pro{}u:\pro{}\sC\to \pro{}\sD$.

Recall standard terminology for the functoriality of limits (=inverse limits) and colimits (= direct limits):

\begin{defn} A functor $u:\sC\to \sD$ is \emph{left exact} (resp. \emph{right exact}, resp. \emph{exact}) if it commutes with finite limits (resp. finite colimits, resp. finite limits and colimits).
\end{defn}

\begin{prop}[dual of \protect{\cite[expos\'e~I, proposition~8.11.4]{SGA4}}]\label{p.proadj} Consider the following conditions:
\begin{thlist}
\item The functor $\pro{}u$ has a left adjoint.
\item There exists a functor $v:\sD\to \pro{}\sC$ and an isomorphism
\[\pro{}\sC(v(d),c)\simeq \sD(d,u(c))\]
contravariant in $d\in \sD$ and covariant in $c\in \sC$.
\item $u$ is left exact.
\end{thlist}
Then {\rm (i)} $\iff$ {\rm (ii)} $\Rightarrow$ {\rm (iii)}, and {\rm (iii)} $\Rightarrow$ {\rm (i)} if $\sC$ is essentially small and closed under finite inverse limits.\qed
\end{prop}

(The condition on finite inverse limits appears in \cite[p. 158]{am}, but is skipped in \cite[expos\'e~I, proposition~8.11.4]{SGA4}.)

\begin{defn}\label{d.proadj} In Condition  (ii) of Proposition \ref{p.proadj}, we say that $v$  is  \emph{pro-left adjoint} to $u$.
\end{defn}

\subsection{Localisation (\cite[Chapter~I]{gz}, see also \cite[Chapter~7]{ks})} Let $\sC$ be a category, and let $\Sigma\subset Ar(\sC)$ be a class of morphisms: following Grothendieck and Maltsiniotis, we call $(\sC,\Sigma)$ a \emph{localiser}. Consider the functors $F:\sC\to \sD$ such that $F(s)$ is invertible for all $s\in \Sigma$. This ``$2$-universal problem'' has a solution $Q:\sC\to \sC[\Sigma^{-1}]$. One may choose $\sC[\Sigma^{-1}]$ to have the same objects as $\sC$ and $Q$ to be the identity on objects; then $\sC[\Sigma^{-1}]$ is unique (not just up to unique equivalence of categories). 
If $\sC$ is essentially small, then $\sC[\Sigma^{-1}]$ is small, but in general the sets $\sC[\Sigma^{-1}](X,Y)$ may be ``large''; one can sometimes show that it is not the case (Corollary \ref{c.small}). A functor of the form $Q:\sC\to \sC[\Sigma^{-1}]$ will be called a \emph{localisation}. We have a basic result on adjoint functors \cite[Chapter~I, Proposition~1.3]{gz}:

\begin{lemma}\label{lA.6} Let $G:\sC\leftrightarrows \sD:D$ be a pair of adjoint functors ($G$ is left adjoint to $D$). Then the following conditions are equivalent:
\begin{thlist}
\item $D$ is fully faithful.
\item The counit $GD\Rightarrow Id_\sD$ is a natural isomorphism.
\item $G$ is a localisation.
\end{thlist}
The same holds if $G$ is right adjoint to $D$ (replacing the counit by the unit).
\end{lemma}

\begin{defn}\label{d.sat} Let $(\sC,\Sigma)$ be a localiser, and let $Q:\sC\to \sC[\Sigma^{-1}]$ be the corresponding localisation functor. We write 
\[\sat(\Sigma)=\{s\in Ar(\sC)\mid Q(s) \text{ is invertible}\}.\]
This is the \emph{saturation} of $\Sigma$; we say that $\Sigma$ is \emph{saturated} if $\sat(\Sigma)=\Sigma$.\end{defn}

\begin{lemma}[\protect{\cite[Chapter~I, Lemma 1.2]{gz}}]\label{lA.5} Let $(\sC,\Sigma)$ be a localiser, $\sD$ a category, $F,G:\sC[\Sigma^{-1}]\to \sD$ two functors and $u:F\circ Q\Rightarrow G\circ Q$ a natural transformation, where $Q:\sC\to \sC[\Sigma^{-1}]$ is the  localisation functor. Then $u$ induces a unique natural transformation $\bar u:F\Rightarrow G$.
\end{lemma}

\begin{proof} Define $\bar u_X=u_X:F(X)\to G(X)$ for $X\in Ob \sC[\Sigma^{-1}]=Ob \sC$. We must show that $\bar u$ commutes with the morphisms of $\sC[\Sigma^{-1}]$. This is obvious, since $u$ commutes with the morphisms of $\sC$ and the morphisms of $\sC[\Sigma^{-1}]$ are expressed as fractions in the morphisms of $\sC$.
\end{proof}

\subsection{Presheaves and pro-adjoints}\label{s.presh}

Let $\sC$ be a category. We write $\hat{\sC}$ for the category of presheaves of sets on $\sC$ (\emph{i.e.} functors $\sC^\op\to \Set$); it comes with the Yoneda embedding
\[y:\sC\to \hat{\sC}\]
which sends an object to the corresponding representable presheaf. If $u:\sC\to \sD$ is a functor, we have the standard sequence of three adjoint functors
\[\begin{CD}
\sC @>y_\sC>> \hat{\sC}\\
@V{u}VV \begin{smallmatrix}u_!\Big\downarrow u^*\Big\uparrow u_*\Big\downarrow \end{smallmatrix}\\
\sD @>y_\sD>> \hat{\sD}
\end{CD}\]
where $u_!$ extends $u$ through the Yoneda embeddings  \cite[expos\'e~I, proposition~5.4]{SGA4}; $u_!$ and $u_*$ are computed by the usual formulas for left and right Kan extensions (\emph{loc.~cit.}, (5.1.1)). If $u$ has a left adjoint $v$, the sequence $(u_!,u^*,u_*)$ extends to
\[(v_!,v^*=u_!,v_*=u^*,u_*) \]
(\emph{ibid.}, Remark~5.5.2). 

Let $\sA$ be an essentially small additive category. Instead of presheaves of sets on $\sA$, one usually uses the category $\Mod\sA$ of \emph{additive presheaves of abelian groups}; the above results transfer to this context, mutatis mutandis. 

\begin{prop}\label{p.funct} 
\leavevmode
\begin{itemize}
\item[a)] The functor $u_!$ (resp. $u_*$, $u^*$) commutes with all representable colimits (resp. limits, limits and colimits). If $u$ has a left adjoint, then $u_!$ also commutes with all limits. If $u$ has a pro-left adjoint $v$ $($Definition \ref{d.proadj}\,$)$, so does $u_!$ which is therefore exact. Moreover, $u_!$ is then given by the formula
\[(u_!F)(Y) = \colim (F\circ v(Y)), \quad F\in \hat{\sC},  Y\in \sD.\]
\item[b)] If $u$ is fully faithful, so is $u_!$.
\item[c)] If $u$ is a localisation or is full and essentially surjective, then $u_!$ is a localisation.
\item[d)] In the case of c), for $C\in \sC$ the following conditions are equivalent:
\begin{thlist}
\item The representable functor $y_\sC(C)\in \hat{\sC}$ induces a functor on $\sD$ via $u$.
\item The unit map $y_\sC(C)\to u^*u_! y_\sC(C)\simeq u^*y_\sD(u(C))$ is an isomorphism.
\item For any $C'\in \sC$, the map $\sC(C',C)\to \sD(u(C'),u(C))$ induced by $u$ is bijective.
\end{thlist}
\end{itemize}
\end{prop}

\begin{proof} a) follows from general properties of adjoint functors, except for the case of a pro-left adjoint. Let $u$ admit a pro-left adjoint $v$, and let $Y\in \sD$: so there is an isomorphism of categories $Y\downarrow u\simeq v(Y)\downarrow c$. Hence, we get by Lemma \ref{l.cof} a cofinal functor 
\[
A\to Y\downarrow u,
\]
 where $A$ is the indexing set of $v(Y)$. Thus, for $F\in \hat{\sC}$, $u_! F(Y)$ may be computed as
\[u_!F(Y)=\colim_{\alpha\in A} F(v(Y)(\alpha))=\pro{}\sC(y_\sC(v(Y)),c(F)). \]

The first equality is the formula in the proposition. The second one shows that the pro-left adjoint $v_!$ of $u_!$ is defined at $y_\sD(Y)$ by $y_\sC(v(Y))$; since any object of $\hat{\sD}$ is a colimit of representable objects, this shows that $v_!$ is defined everywhere.

For b), see \cite[expos\'e~I, proposition~5.6]{SGA4}. In c), it is equivalent to show that $u^*$ is fully faithful by Lemma \ref{lA.6}. Let $F,G\in \hat{\sD}$, and let $\phi:u^*F\to u^*G$ be a morphism of functors.  In both cases, $u$ is essentially surjective: given $X\in \sD$ and an isomorphism $\alpha:X\iso u(Y)$, we get a morphism
\[\psi_X:F(X)\by{{\alpha^*}^{-1}} F(u(Y))\by{\phi_Y} G(u(Y))\by{\alpha^*} G(X).\]

The fact that $\psi_X$ is independent of $(Y,\alpha)$ and is natural in $X$ is an easy consequence of each hypothesis (see Lemma \ref{lA.5} in the first case).

In d), the equivalence (ii) $\iff$ (iii) is tautological and (iii) $\Rightarrow$ (i) is obvious. The implication (i) $\Rightarrow$ (iii) was proven in \cite[Chapter~I, \S4.1.2]{gz} assuming that $u$ is a localisation enjoying a calculus of left fractions; let us prove (i) $\Rightarrow$ (ii) in general. Under (i), we have $y_\sC(C)\simeq u^* F$ for some $F\in \hat{D}$; the unit map becomes
\[\eta_{u^*F}:u^*F\to u^*u_!u^*F.\]

On the other hand,  the counit map $\epsilon_F:u_!u^*F\to F$ is invertible by the full faithfulness of $u^*$. By the adjunction identities, we have $u^*(\epsilon_F)\circ \eta_{u^*F}=1_{u^*F}$. Hence the conclusion.
\end{proof}

We shall usually write $u^!$ for the pro-left adjoint of $u_!$, when it exists.

\subsection{Calculus of fractions}

\begin{defn}[dual of \protect{\cite[Chapter~I, Lemma~1.2]{gz}}]\label{d.cf} A localiser $(\sC,\Sigma)$ (or simply $\Sigma$)  enjoys a \emph{calculus of right fractions} if:
\begin{thlist}
\item The identities of $\sC$ are in $\Sigma$.
\item $\Sigma$ is stable under composition.
\item (Ore condition.) For each diagram $X'\by{s} X\yb{u} Y$ where $s\in \Sigma$, there exists a commutative square
\[
\begin{CD}
Y'@>u'>> X'\\
@Vt VV @Vs VV\\
Y@>u>> X
\end{CD}
\qquad \text{where } t\in\Sigma.
\]
\item (Cancellation.) If $f,g:X\rightrightarrows Y$ are morphisms in $\sC$ and $s:Y\to Y'$ is a morphism of $\Sigma$ such that $sf=sg$, there exists a morphism $t:X'\to X$ in $\Sigma$ such that $ft=gt$.
\end{thlist}
\end{defn}

\begin{prop}\label{p.cf} Suppose that $\Sigma$ enjoys a calculus of right fractions. For $c\in \sC$, let $\Sigma\downarrow c$ denote the full subcategory of the comma category $\sC\downarrow c$ given by the objects $c'\by{s}c$ with $s\in \Sigma$. Then
\begin{itemize}
\item[a)] $\Sigma\downarrow c$ is cofiltered.
\item[b)]\cite[Chapter I, 2.3]{gz}  For any $d\in \sC$, the obvious map
\begin{equation}\label{eq.cf}
\colim_{c'\in \Sigma\downarrow c}\sC(c',d)\to \sC[\Sigma^{-1}](c,d)
\end{equation}
is an isomorphism.
\item[c)] Any morphism in $\sC[\Sigma^{-1}]$ is of the form $Q(f)Q(s)^{-1}$ for $f\in Ar(\sC)$ and $s\in \Sigma$; if $f_1,f_2$ are two parallel arrows in $\sC$, then $Q(f_1)=Q(f_2)$ if and only if there exists $s\in \Sigma$ such that $f_1s=f_2s$.
\end{itemize}
\end{prop}

\begin{proof} 
a) 
We need to check the two conditions
(which are dual to those from \cite[p. 211]{mcl}):
(1) 
given two objects $d, d' \in \Sigma \downarrow c$,
there are arrows $d \leftarrow e \rightarrow d'$ in $\Sigma \downarrow c$;
(2) 
given two parallel arrows $f, g : e \to d$ in  $\Sigma \downarrow c$,
there is an arrow $h : e' \to e$ in $\Sigma \downarrow c$ such that $fh=gh$.
(1) (resp. (2)) follows from Axioms (iii) and (ii) 
(resp. (iv) and (ii)) of Definition \ref{d.cf}.

b) The ``obvious map'' \eqref{eq.cf} sends a pair $(c'\by{s} c,c'\by{f}d)$ with $s\in \Sigma$ and $f\in \sC(c',d)$ to $Q(f)Q(s)^{-1}$. 
To show it is an isomorphism, we follow the strategy of \cite[pp. 13/14]{gz}.
We consider a category $\Sigma^{-1}\sC$ with the same objects as $\sC$ 
and for $c, d \in \sC$ the Hom set $\Sigma^{-1}\sC(c, d)$ is given by 
the left hand side of \eqref{eq.cf}.
Using Axioms (ii) and (iii), we define for three objects $c,d,e\in \sC$ a composition
\[
\colim_{c'\in \Sigma\downarrow c}\sC(c',d)\times \colim_{d'\in \Sigma\downarrow d}\sC(d',e)\to \colim_{c'\in \Sigma\downarrow c}\sC(c',e)
\]
which is shown to be well-defined and associative thanks to Axiom (iv). 
Now \eqref{eq.cf} yields a functor $\Sigma^{-1}\sC\to \sC[\Sigma^{-1}]$. 
But there is also an obvious functor $\sC\to \Sigma^{-1}\sC$
that is the identity on objects.
(We use Axiom (i) to define the maps for the Hom sets.)
It is easily seen to have the universal property of $\sC[\Sigma^{-1}]$. 
Hence \eqref{eq.cf} is an isomorphism for all $(c,d)$.


c) The first statement has already been observed; the second one follows readily from \eqref{eq.cf}.
\end{proof}

\begin{nota} We shall write $\Sigma^{-1}\sC$ instead of $\sC[\Sigma^{-1}]$ if $\Sigma$ enjoys a calculus of fractions.
\end{nota}

\begin{cor} \label{c.small} If $\Sigma$ admits a calculus of right fractions and if for any $c\in \sC$, the category $\Sigma\downarrow c$ contains a small cofinal subcategory, then the $\Hom$ sets of $\Sigma^{-1}\sC$ are small.\qed
\end{cor}

\begin{cor}\label{c.cf} Let $(\sC,\Sigma)$ be a localiser such that $\Sigma$ enjoys a calculus of right fractions. Let $F:\sC\to \sD$ be a functor. Suppose that $F$ inverts the morphisms of $\Sigma$ and that, for any $c,d\in \sC$, the obvious map
\[\colim_{c'\in \Sigma\downarrow c}\sC(c',d)\to \sD(F(c),F(d))\]
is an isomorphism. Then the functor $\Sigma^{-1}F:\Sigma^{-1}\sC\to \sD$ induced by $F$ is fully faithful.\qed
\end{cor}

\begin{prop}\label{p1.10} 
\leavevmode
\begin{itemize}
\item[a)]  Let $(\sC,\Sigma)$ be a localiser. Assume that $\Sigma$ enjoys a calculus of right fractions. Then the localisation functor $Q:\sC\to \Sigma^{-1}\sC$ is left exact; if limits indexed by a finite category $I$ exist in $\sC$, they also exist in $\Sigma^{-1}\sC$.
\item[b)] Let $\sC$ be an essentially small category closed under finite limits, and let $G:\sC\to \sD$ be a left exact functor.  Let $\Sigma=\{s\in Ar(\sC)\mid G(s) \text{ is invertible}\}$. Then $\Sigma$ enjoys a calculus of right fractions; the induced functor $\Sigma^{-1}\sC\to \sD$ is conservative and left exact.
\end{itemize}
\end{prop}

\begin{proof} After passing to the opposite categories, a) is \cite[Chapter~I, Proposition~3.1 and Corollary~3.2]{gz} and b) is \cite[Chapter~I, Proposition~3.4]{gz}.\end{proof}




\subsection{Pro-$\Sigma$-objects}

\begin{defn}\label{d2.1}
Let $(\sC,\Sigma)$ be a localiser. We write $\pro{\Sigma}\sC$ for the full subcategory of the category $\pro{}\sC$ of pro-objects of $\sC$ consisting of filtered inverse systems  whose transition morphisms belong to $\Sigma$. An object of $\pro{\Sigma}\sC$ is called a \emph{pro-$\Sigma$-object}. 
\end{defn}

\begin{prop}\label{p2.6} Suppose that $\Sigma$ has a calculus of right fractions and, for any $c\in \sC$, the category $\Sigma\downarrow c$ contains a small cofinal subcategory. Then $Q:\sC\to \Sigma^{-1}\sC$ has a pro-left adjoint $Q^!$, which takes an object $Y\in \Sigma^{-1} \sC$ to $``\lim"_{X\in \Sigma\downarrow Y} X$
$($see \S \ref{sec:pro-obj} for the notation $``\lim")$. 
In particular, $Q^!(\Sigma^{-1}\sC)\subset \pro{\sat(\Sigma)}\sC$, where $\sat(\Sigma)$ is the saturation of $\Sigma$ $($Definition \ref{d.sat}\,$)$.
\end{prop}

\begin{proof} In view of Corollary \ref{c.small} and Proposition \ref{p1.10}, this follows from Proposition \ref{p.cf} b).
\end{proof}

\begin{rk}\label{r.cf} Consider the localisation functor $Q:\sC\to \Sigma^{-1} \sC$: it has a left Kan extension 
$$\hat Q:\pro{\sat(\Sigma)}\sC\to \Sigma^{-1}\sC$$
\cite[Chapter~X]{mcl} along the constant functor $\sC\to \pro{\sat(\Sigma)}\sC$, given by the formula
\[\hat Q(``\lim\nolimits" C_\alpha)= \lim Q(C_\alpha).\]
(The right hand side makes sense as an inverse limit of isomorphisms.) Then one checks easily that $Q^!$ is left adjoint to $\hat Q$. 
\end{rk}

\begin{thm}\label{lem:omega-sh0}\ Let $(\sC,\Sigma)$ be a localiser verifying the conditions of Proposition \ref{p2.6}. Let $Q:\sC\to \Sigma^{-1}\sC$ denote the localisation functor, and consider the string of adjoint functors $(Q_!,Q^*,Q_*)$ between $\hat{\sC}$ and $\widehat{\Sigma^{-1}\sC}$ from \S \ref{s.presh}. Then:
\begin{enumerate}
\item $Q_!$ has a pro-left adjoint, and is therefore exact.
\item For $F \in \hat{\sC}$ and $Y \in \Sigma^{-1}\sC$, we have
\[ Q_! F(Y) = \colim_{X \in \Sigma\downarrow Y} F(X). \]
\end{enumerate}
\end{thm}

\begin{proof} This follows from Propositions \ref{p.funct} a) and \ref{p2.6}. 
\end{proof}

If $(\sA,\Sigma)$ is a localiser with $\sA$ additive and $\Sigma$ enjoys a calculus of right fractions, then $\Sigma^{-1}\sA$ is additive and so is the functor $Q:\sA\to \Sigma^{-1}\sA$ \cite[Chapter~I, Corollary~3.3]{gz}.  
For future reference, we give the additive analogue of Theorem \ref{lem:omega-sh0}
(see the paragraph before Proposition~\ref{p.funct} for $\Mod\sA$):

\begin{thm}\label{lem:omega-sh}\ Let $(\sA,\Sigma)$ be a localiser; assume that $\sA$ is an additive category and that $\Sigma$ has a calculus of right fractions. Let $Q:\sA\to \Sigma^{-1}\sA$ denote the localisation functor, as well as the string of adjoint functors $(Q_!,Q^*,Q_*)$ between $\Mod\sA$ and $\Mod\Sigma^{-1}\sA$. Then:
\begin{enumerate}
\item $Q_!$ has a pro-left adjoint, and is therefore exact.
\item For $F \in \Mod\sA$ and $Y \in \Sigma^{-1}\sA$, we have
\[ Q_! F(Y) = \colim_{X \in \Sigma\downarrow Y} F(X). \]
\end{enumerate}
\end{thm}

\subsection{cd-structures} \label{sa-cd}

Let $\sC$ be a category with an initial object. According to \cite{cdstructures}, a \emph{cd-structure} on $\sC$ is given by a collection of commutative squares stable under isomorphisms, called \emph{distinguished squares}. Any cd-structure defines a topology on $\sC$: the smallest Grothendieck topology such that for a distinguished square of the form 
\begin{equation}\label{eq.cd2}
S: \qquad
\begin{CD}
W@>v>> V\\
@VqVV @Vp VV\\
U@>u>> X,
\end{CD}
\end{equation}
the sieve generated by the morphisms
$\{p:V \to X, ~u:U\to X\}$
is a cover sieve and such that the empty sieve is a cover sieve of the initial object $\emptyset$.

Recall from \cite{cdstructures} some important properties of cd-structures.

\begin{defn}\label{d3.1} Let $\sC$ be a category with an initial object $\emptyset$. 
\begin{enumerate}
\item Let $P$ be a cd-structure on $\sC$. 
The class $\sS_P$ of \emph{simple covers} is the smallest class of families of morphisms of the form $\{U_i \to X\}_{i \in I}$ satisfying the following two conditions:
\begin{itemize}
\item for any isomorphism $f$, $\{f\}$ is in $\sS_P$
\item for a distinguished square $Q$ of the form \eqref{eq.cd2} and families $\{p_i : V_i \to V\}_{i \in I}$ and $\{q_j : U_j \to U\}_{j \in J}$ in $\sS_P$ the family $\{p \circ p_i , u \circ q_j\}_{i \in I, j \in J}$ is in $\sS_P$.
\end{itemize}
\item A cd-structure on $\sC$ is called \emph{complete} if any cover sieve of an object $X \in \sC$ which is not isomorphic to $\emptyset$ contains a sieve generated by a simple cover.
\item A cd-structure $P$ is called \emph{regular} if for $S \in P$ of the form \eqref{eq.cd2} one has
\begin{itemize}
\item $S$ is a pullback square (\emph{i.e.}, is cartesian)
\item $u$ is a monomorphism 
\item the morphisms of sheaves
\[
\Delta \bigsqcup \rho (v) : \rho (V) \bigsqcup \rho (W) \times_{\rho (U)} \rho (W) \to \rho (V) \times_{\rho (X)} \rho (V)
\]
is surjective, where for $C \in \sC$ we denote by $\rho (C)$ the sheaf associated with the presheaf represented by $C$,
and $\Delta$ is induced by the diagonal map.
\end{itemize}
\end{enumerate}
\end{defn}

\begin{lemma}[\protect{\cite[Lemma 2.5]{cdstructures}}]\label{lA.7} A cd-structure is complete provided: 
\begin{enumerate}
\item any morphism with values in $\emptyset$ is an isomorphism, and 
\item for any distinguished square $S$ of the form \eqref{eq.cd2}
and for any morphism $X ' \to X$, the square $S' = X' \times_X S$ is defined and distinguished.\qed
\end{enumerate}
\end{lemma}

\begin{lemma}[\protect{\cite[Lemma 2.11]{cdstructures}}]\label{lA.8} A cd-structure is regular provided, for any distinguished square $S$ of the form \eqref{eq.cd2} we have
\begin{enumerate}
\item $S$ is cartesian,
\item $u$ is a monomorphism, and
\item the objects $V \times_X V$ and $W \times_{U} W$ exist in $\sC$ and the derived square
\begin{equation}\label{uld(S)} d(S) : \qquad \vcenter{\xymatrix{
W \ar[r]^v \ar[d]_{\Delta_{W/U}} & V \ar[d]^{\Delta_{V/X}} \\
W \times_{U} W \ar[r]^{} & V \times_{X} V
}}\end{equation}
is distinguished.\qed
\end{enumerate}
\end{lemma}

\begin{defn}\label{dA.5} A cd-structure verifying the conditions of Lemma \ref{lA.7} (resp. \ref{lA.8}) is called \emph{strongly complete} (resp. \emph{strongly regular}).
\end{defn}

\begin{remark}\label{lA.1}
The square \eqref{uld(S)} is cartesian.
This is a formal consequence of Lemma \ref{lA.8}, since any distinguished square with respect to a regular cd-structure is cartesian by definition. 
However, there is a more direct proof: let $Z\by{a} V$ and $Z\by{b} W \times_{U} W$ be two morphisms making the corresponding square commute. Then $b$ amounts to two morphisms $b_1, b_2:Z\to W$ such that (with the notation of \eqref{eq.cd2}) $qb_1=qb_2$ and $a=vb_1=vb_2$. Since $S$ is cartesian by (1), we have $b_1=b_2 : Z\to W$, which is a solution to the universal problem.
\end{remark}


%

\begin{prop}\label{pA.4} Let $(\sC,\Sigma)$ be a localiser such that  $\Sigma$ admits a calculus of right fractions.
\begin{enumerate}
\item If $\sC$ has an initial object verifying Conditon (1) of Lemma \ref{lA.7}, so does $\Sigma^{-1} \sC$.
\item Assume (1) and let $Q:\sC\to \Sigma^{-1} \sC$ be the localisation functor. Suppose given a cd-structure $P$ on $\sC$, and let $P'$ be the cd-structure on $\Sigma^{-1} \sC$ given by all squares isomorphic to a square of the form $Q(S)$, where $S\in P$. If $P$ is strongly complete (resp. strongly regular), so is $P'$. 
\end{enumerate}
\end{prop}

\begin{proof} (1) Let $\emptyset$ be an initial object of $\sC$. Since $Q$ is (essentially) surjective, $Q(\emptyset)$ admits a morphism to any object; Condition (1) of Lemma \ref{lA.7} for $\emptyset$ implies that this morphism is unique, and this in turn implies the same condition for $Q(\emptyset)$.

(2) By Proposition \ref{p1.10} a),  $Q$ commutes with finite limits. This implies Condition (2) of Lemma \ref{lA.7}. Conditions (1), (3) of Lemma \ref{lA.8} for $P'$ follow from the same conditions for $P$ (note that the diagonals are preserved by $Q$, since they are finite limits).  It remains to show that $Q$ carries a monomorphism $u:U\to X$ to a monomorphism. Let $f,g:V\to U$ be two morphisms in $\Sigma^{-1} \sC$ such that $Q(u)f=Q(u)g$. By calculus of fractions, we may write $f=Q(\tilde f)Q(s)^{-1}$ and  $g=Q(\tilde g)Q(s)^{-1}$ for some $\tilde f, \tilde g\in Ar(\sC)$ and $s\in \Sigma$. Then $Q(u\tilde f)=Q(u \tilde g)$. By Proposition \ref{p.cf} c), we may find $t\in \Sigma$ such that $u\tilde ft = u\tilde g t$, which implies $\tilde f t=\tilde gt$ since $u$ is a monomorphism. This shows $f=g$, as desired. 
\end{proof}

\subsection{A pull-back lemma} 
We shall use 
the following elementary lemma several times.

\begin{lemma}\label{lem:lr-adjoint} 
Let $\sC,  \sD$ be abelian categories
and let $\sC' \subset \sC, \sD' \subset \sD$
be full abelian subcategories.
Let $c :\sC \to \sD$ and $c' :\sC' \to \sD'$ 
be additive functors satisfying $c i_C = i_D c'$,
where $i_C : \sC' \to \sC$
and $i_D : \sD' \to \sD$ are inclusion functors.
\begin{enumerate}
\item If $c$ is faithful, so is $c'$.
\item Suppose that $i_D$ is strongly additive or  has a strongly additive left inverse (for example, a left adjoint). If $c$ and $i_C$ are strongly additive, so is $c'$.
\item Suppose that $i_C$ has a left adjoint $a_C$.
If $c$ has a left adjoint $d$,
then $d'=a_C d i_D$ is a left adjoint of $c'$. If $d$ and $a_\sC$ are exact, so is $d'$. Moreover, $a_\sC d=d'a_\sD$ if $i_\sD$ has a left adjoint $a_\sD$.
\item 
Suppose that 
$i_C$ and $i_D$ have left adjoints $a_C$ and $a_D$,
that $a_D$ is exact, and that $a_D c= c' a_C$.
If $c$ is exact, then so is $c'$.
\end{enumerate}
\end{lemma}
\begin{proof}
(1) is obvious.
(2) Let $\{F_i \}_{i \in I}$ be a family of objects of $\sC'$. We must show that the natural map
\[f:\bigoplus_{i \in I} c'(F_i )\to c'\big(\bigoplus_{i \in I} F_i\big)\]
is an isomorphism. The composition
\[\bigoplus_{i \in I} i_D c' F_i\by{g} i_D\big(\bigoplus_{i \in I} c'(F_i) \big)\by{i_Df} i_D c'\big(\bigoplus_{i \in I} F_i\big)\]
is an isomorphism by the strong additivity of $c$ and $i_C$.  If $i_D$ is strongly additive,  $g$ is also an isomorphism and we are done. If now $i_D$ has a strongly additive left inverse $a_D$, we apply it to the diagram and get a composition
\[\bigoplus_{i \in I} a_D\ i_D c' F_i\iso a_D\bigoplus_{i \in I} i_D c' F_i\by{a_D g} a_Di_D\big(\bigoplus_{i \in I} c'(F_i) \big)\by{a_Di_Df} a_Di_D c'\big(\bigoplus_{i \in I} F_i\big)\]
which is an isomorphism and naturally isomorphic to $f$. This concludes the proof of (2).

(3) For $F \in \sC'$ and $G \in \sD'$, we have
$\sC'(a_C d i_D F, G)
=\sD(i_D F, c i_C G)
=\sD(i_D F, i_D c' G)
=\sD'(F, c'G)$. This proves the first claim; therefore if $d$ and $a_\sC$ are exact, $d'$ is left exact, hence exact since it is right exact as a left adjoint. The last isomorphism follows from taking left adjoints of the isomorphism $ci_\sC=i_\sD c'$.

(4) Let us take an exact sequence
$0 \to F \to G \to H \to 0$ in $\sC'$.
Put $K:=\Coker(i_C G \to i_C H) \in \sC$.
Since $a_D c K = c' a_C K = 0$,
we get an exact sequence
$0 \to a_D c i_C F \to 
a_D c i_C G \to a_D c i_C H \to 0$
by the exactness of $c$ and $a_D$.
Using $a_D c= c' a_C$ and $a_C i_C = \id$ (Lemma \ref{lA.6}),
we conclude 
$0 \to c' F \to c' G \to c' H \to 0$ is exact.
\end{proof}

The proof of Lemma \ref{lem:lr-adjoint} (2) implicitly used the following (trivial) lemma, which we state for the sake of clarity.

\begin{lemma}\label{lA.9} Let $\sD\subseteq \sC$ be a full embedding of categories. Suppose that a direct (resp. inverse) system $(d_\alpha)$ of objects of $\sD$ has a colimit (resp. a limit) in $\sC$, which is isomorphic to an object $d$ of $\sD$. Then $d$ represents the (co)limit of $(d_\alpha)$ in $\sD$.
\end{lemma}


\subsection{Homological algebra}  Recall Gro\-then\-dieck's theorem \cite[th\'eor\`eme~2.4.1]{tohoku}:

\begin{thm}\label{tA.2} Let $\sA\by{F}\sB\by{G}\sC$ be a string of left exact functors between abelian categories. Suppose that $\sA$ and $\sB$ have enough injectives and that $F$ carries injectives of $\sA$ to $G$-acyclics. Then, for any $A\in \sA$, there is a convergent spectral sequence
\[E_2^{p,q}=R^pG R^qF(A)\Rightarrow R^{p+q}(GF)(A).\]
\end{thm}

\begin{exs}\label{exA.3} If $F$ has an exact left adjoint, it carries injectives to injectives. If $G$ is exact, the hypothesis on $F$ is automatically verified.
\end{exs}




The following is a slight generalization of \cite[Chapter~III, Proposition~2.12]{milne},
(where the underlying category of $\sS$ 
is supposed to be a category of schemes).

\begin{lemma}\label{lem:milne}
Let $F$ be a sheaf of abelian groups on a site $\sS$.
The following conditions are equivalent.
\begin{enumerate}
\item 
We have $H^q(X, F)=0$
for any $X \in \sS$ and $q>0$. 
\item
We have $\check{H}^q(X, F)=0$
for any $X \in \sS$ and $q>0$. 
\item
We have $\check{H}^q(U/X, F)=0$
for any cover $U \to X$ in $\sS$ and $q>0$. 
\item
The sheaf $F$ is $i_{\sS}$-acyclic,
where $i_{\sS}$ is the inclusion functor
of the category of sheaves to that of presheaves.
\end{enumerate}
\end{lemma}

\begin{proof}
%
%
%
%
For $X \in \sS$,
we write $\Gamma_X$ (resp. $\Gamma^\pr_X$)
for the functor $F \mapsto F(X)$
from the category of sheaves
(resp. presheaves) to $\Ab$.
We have $\Gamma_X = \Gamma_X^\pr i_{\sS}$.
Since $\Gamma_X^\pr$ is exact,
Theorem \ref{tA.2} implies 
$R^q \Gamma_X = \Gamma_X^\pr R^q i_{\sS}$,
and hence $H^q(X, F)=R^q i_{\sS}F(X)$.
This proves the equivalence of (1) and (4).
The rest is shown in the same way as
\cite[Chapter~III, Proposition~2.12]{milne}.
\end{proof}

\begin{definition}\label{def:flabby}
We say $F$ is  \emph{flabby}
if the conditions of Lemma \ref{lem:milne} are satisfied.
\end{definition}

\begin{lemma}\label{lem:inj-flabby}
Let $\sS$ be the category of abelian sheaves on a site $\sC$,
$\sT$ an abelian category, 
and $c^* : \sT \to \sS$ an additive functor
which has a left adjoint $c_! : \sS \to \sT$.
Suppose that any cover in $\sC$
admits a refinement $U \to X$ such that
$c_!(\check{C}(U/X))$ is exact in $\sT$,
where 
\[\check{C}(U/X) = ( \dots \to y(U \times_X U)
\to y(U) \to y(X) \to 0)
\]
is the \v{C}ech complex associated to $U \to X$
($y$ denotes the Yoneda functor).
Then $c^*I$ is flabby
for any injective object $I \in \sT$.
\end{lemma}
\begin{proof}
(Compare \cite[Proposition~3.1.7]{voetri}.)
It suffices to show 
$\check{H}^q(U/X, c^*I)=0$ for any $q>0$
and for any $U \to X$ as in the assumption.
If we denote by $U_X^n$ the
$n$-fold fiber product of $U$ over $X$,
then $\check{H}^q(U/X, c^*I)$
is computed as the cohomology of the complex
\begin{align*}
c^*I(U_X^{\bullet+1}) 
= \sS(y(U_X^{\bullet+1}), c^*I)
= \sT(c_! y(U_X^{\bullet+1}), I),
\end{align*}
which is acyclic by the assumption and
the injectivity of $I$.
\end{proof}

\subsection{Grothendieck categories} 
\label{s.groth}

Recall that a \emph{Gro\-then\-dieck abelian category} (for short, a Grothendieck category) is an abelian category verifying Axiom AB5 of \cite{tohoku}:  small colimits are representable and exact, and having a set of generators (equivalently, a generator). These generators are generators by strict epimorphisms.
We have the following basic facts:

\begin{thm}\label{t.groth} 
\leavevmode
\begin{itemize}
\item[a)] Any Grothendieck category is complete and has enough injectives.
\item[b)] Let $\sF:\sC\to \sD$ be a functor, where $\sC$ is a Grothendieck category. Then $F$ has a right adjoint if and only if it commutes with all colimits.
\item[c)] Let $\sC$ be a Grothendieck category, $\sB\subset \sC$ be a Serre subcategory, $\sD=\sC/\sB$ and $G:\sC\to \sD$ the (exact) localisation functor. Then $G$ has a right adjoint $D$ if and only if $\sB$ is stable under infinite direct sums. In this case, $\sB$ and $\sD$ are Grothendieck.
\item[d)] Let $G:\sC\leftrightarrows \sD:D$ be a pair of adjoint additive functors between additive categories, with $D$ fully faithful. If $\sC$ is Grothendieck and $G$ is exact, $\sD$ is Grothendieck.
\end{itemize}
\end{thm}

\begin{proof} a) See \cite[th\'eor\`eme~1.10.1]{tohoku}, \cite[expos\'e~V, remarque~0.2.1]{SGA4} or \cite[Theorem~8.3.27 (i) and 9.6.2]{ks}. b) See \cite[Proposition~8.3.27 (iii)]{ks}. c) See  \cite[chapitre~III, proposition~8 and 9]{gabriel}. d) Let $\sB$ be the kernel of $G$. Then $\sB$ is easily seen to be a Serre subcategory (e.g. \cite[chapitre~III, proposition~5]{gabriel}), so the claim follows from c). 
\end{proof}

\begin{thm}\label{t.mon} For any additive category $\sA$,
$\Mod\sA$ is a Gro\-then\-dieck category 
with a set of projective generators. 
\end{thm}
\begin{proof} See e.g. \cite[Proposition~1.3.6]{ak} for the first statement; the projective generators are given by $\sE=\{y(A)\mid A\in \sA\}$.
\end{proof}


\end{document}